\documentclass[10pt,reqno]{amsart}
\usepackage[margin=1.2in]{geometry}
\usepackage{amssymb,mathrsfs}
\usepackage{ifthen}
\usepackage{colortbl}
\usepackage{cases,appendix}
\usepackage{epsfig,subfigure}
\definecolor{black}{rgb}{0.0, 0.0, 0.0}

\provideboolean{shownotes} 
\setboolean{shownotes}{true} 
\newcommand{\margnote}[1]{
\ifthenelse{\boolean{shownotes}}%
{\marginpar{\raggedright\tiny\texttt{#1}}}%
{}%
}
\newcommand{\hole}[1]{
\ifthenelse{\boolean{shownotes}}%
{\begin{center} \fbox{ \rule {.25cm}{0cm} \rule[-.1cm]{0cm}{.4cm}
\parbox{.85\textwidth}{\begin{center} \texttt{#1}\end{center}} \rule
{.25cm}{0cm}}\end{center}} {} }

\title[A hydrodynamic model for synchronization phenomena]{A hydrodynamic model for synchronization phenomena}

\author[Choi]{Young-Pil Choi}
\address[Young-Pil Choi]{\newline Department of Mathematics and Institute of Applied Mathematics\newline
Inha University, Incheon 22212, Republic of Korea}
\email{ypchoi@inha.ac.kr}

\author[Lee]{Jaeseung Lee}
\address[Jaeseung Lee]{\newline The Research Institute of Basic Sciences \newline Seoul National University, Seoul 08826, Republic of Korea}
\email{jaeseunglee@snu.ac.kr}



\numberwithin{equation}{section}

\newtheorem{theorem}{Theorem}[section]
\newtheorem{lemma}{Lemma}[section]

\newtheorem{proposition}{Proposition}[section]
\newtheorem{remark}{Remark}[section]

\newcommand{\R}{\mathbb R}

\newcommand{\ls}{\lesssim}

\newcommand{\T}{\mathbb T}
\newcommand{\N}{\mathbb N}
\newcommand{\Z}{\mathbb Z}

\newcommand{\mc}{\mathcal C}

\newcommand{\bq}{\begin{equation}}
\newcommand{\eq}{\end{equation}}
\newcommand{\e}{\varepsilon}
\newcommand{\lt}{\left}
\newcommand{\rt}{\right}

\newcommand{\pa}{\partial}

\newcommand{\ml}{\mathcal{L}}

\newcommand{\me}{\mathcal{E}}

\newcommand{\mms}{\mathcal{S}}

\newcommand{\tot}{\theta,\Omega,t}
\newcommand{\tots}{\theta_*,\Omega_*,t}

\begin{document}
\allowdisplaybreaks

\date{\today}


\begin{abstract} We present a new hydrodynamic model for synchronization phenomena which is a type of pressureless Euler system with nonlocal interaction forces. This system can be formally derived from the Kuramoto model with inertia, which is a classical model of interacting phase oscillators widely used to investigate synchronization phenomena, through a kinetic description under the mono-kinetic closure assumption. For the proposed system, we first establish local-in-time existence and uniqueness of classical solutions. For the case of identical natural frequencies, we provide synchronization estimates under suitable assumptions on the initial configurations. We also analyze critical thresholds leading to finite-time blow-up or global-in-time existence  of classical solutions. In particular, our proposed model exhibits the finite-time blow-up phenomenon, which is not observed in the classical Kuramoto models, even with a smooth distribution function for natural frequencies. Finally, we numerically investigate synchronization, finite-time blow-up, phase transitions, and hysteresis phenomena.
\end{abstract}

\maketitle \centerline{\date}

\tableofcontents

%
%
%
%
\section{Introduction} 
The phenomenon of collective synchronization exhibited by various biological systems is ubiquitous in nature, and it has been extensively studied in many different scientific disciplines such as applied mathematics, physics, biology, sociology, and control theory due to their biological and engineering applications, \cite{ABPRS, BB, CS09, Erm, Kura, PLR, Str, Ward, Win}. The mathematical treatment of synchronization phenomena was pioneered by Winfree \cite{Win} and Kuramoto \cite{Kura}. Winfree first introduced a first-order model for collective synchronization of weakly coupled nonlinear oscillators. Subsequently, Kuramoto proposed a mathematically tractable model consisting of a population of coupled phase oscillators having natural frequencies extracted from a given distribution, and all of them are coupled by a mean-field interaction, sinusoidal coupling. The Kuramoto model contains all the main features of interest. In particular, the Kuramoto model displays a phase transition between coherent and incoherent states: the oscillators rotate on a circle incoherently when the coupling strength is weak enough, while the collective synchronization occurs when the coupling strength is beyond some threshold. Since then, the Kuramoto model has become a paradigmatic model for synchronization phenomena.

There already exist various extensions, such as additive/multiplicative noises, time-delayed coupling, inertia, frustration, and networks, extensively explored in \cite{BM16, CCP19, CHY1, CHY2, CHN, CL19, Erm, HKLZ, HLZ, TLO1, TLO2}. However, to the best of our knowledge, there is no available literature on hydrodynamic models for synchronization phenomena. In the current work, we present a new hydrodynamic model, which is a pressureless Euler-type system, for the synchronization phenomena. More specifically, let $\rho = \rho(\theta,\Omega,t)$ and $u = u(\theta,\Omega,t)$ be the density and velocity functions of Kuramoto oscillators, respectively, in $\theta \in \T := \R/(2\pi \Z)$ with a natural frequency $\Omega$ extracted from a given distribution function $g = g(\Omega)$ at time $t > 0$. Then our main system is given by
\begin{align}\label{hydro_Ku}
\begin{aligned}
&\pa_t \rho + \pa_\theta(\rho u) = 0, \quad (\theta,\Omega) \in \T \times \R, \quad t >0,\cr
&\pa_t (\rho u) + \pa_\theta(\rho u^2) = \frac1m \left( -\rho u +  \rho \Omega + K\rho\int_{\T \times \R} \sin(\theta_* - \theta)\rho(\theta_*,\Omega_*,t)g(\Omega_*)\,d\theta_* d\Omega_*\right),
\end{aligned}
\end{align}
with the initial data
\bq\label{ini_hydro_Ku}
(\rho,u)(\theta,\Omega,0) =: (\rho_0(\theta,\Omega),u_0(\theta,\Omega)), \quad (\theta,\Omega) \in \T \times \R.
\eq
Here $m > 0$ and $K >0$ denote the strength of inertia and coupling strength, respectively. 

The system \eqref{hydro_Ku} can be formally derived from the second-order system of ordinary differential equations for synchronization, called the Kuramoto model with inertia, through a kinetic description under a mono-kinetic closure assumption. More precisely, our starting point is the $N$-particle Kuramoto oscillators with inertia. Let $\theta_i \in \R$ be the phase of the $i$-th oscillator with the natural frequency $\Omega_i$. Then the dynamics of second-order Kuramoto oscillators is governed by the following system:
\bq\label{particle_Ku}
m\ddot{\theta}_i(t) + \dot{\theta}_i(t) = \Omega_i + \frac KN\sum_{j=1}^N \sin(\theta_j(t) - \theta_i(t)), \quad i=1,\cdots, N, \quad t >0.
\eq

The particle model \eqref{particle_Ku} is introduced in \cite{Erm} as a phenomenological model to describe the slow relaxation in the synchronization process in certain biological systems, e.g., fireflies of the Pteroptyx malaccae. Note that the classical Kuramoto model can be simply obtained by disregarding the inertial effect, i.e., setting $m = 0$. A different set of applications of the second-order phase model \eqref{particle_Ku} includes power grids, superconducting Josephson junction arrays, and explosive synchronization \cite{CL19, DB11, DB12, JPMRK13, WS94, WS97, WCS96}. Furthermore, it is known that the model \eqref{particle_Ku} exhibits rich phenomena such as the discontinuous phase transition and hysteretic dynamics \cite{BM16, TLO1, TLO2}. For mathematical results on \eqref{particle_Ku}, we refer to \cite{CHY1, CHM18, CHN, CLHXY14, DB12}.

On the other hand, when the number of oscillators $N$ is very large, the microscopic description \eqref{particle_Ku} is computationally complicated, and thus understanding how this complexity can be reduced is an important issue. The classical strategy to reduce this complexity is to derive a mesoscopic description, i.e., continuum model of the dynamics, by introducing a distribution function. Let $f = f(\theta,\omega,\Omega,t)$ be the one-oscillator distribution function on the space $\T \times \R$ with the natural frequency $\Omega$ at time $t$ and satisfy the normalized condition $\int_{\T \times \R}f(\theta,\omega,\Omega,t)\,d\theta d\omega = 1$. At the formal level, we can expect that as the number of oscillators $N$ goes to infinity, the $N$-particle system \eqref{particle_Ku} will be replaced by the following Vlasov-type equation:
\bq\label{kinetic_Ku}
\pa_t f + \pa_\theta(\omega f) +\pa_\omega (F[f]f) = 0, \quad (\theta,\omega,\Omega) \in \T \times \R \times \R, \quad t > 0,
\eq
where the interaction term $F[f]$ is given by
\[
F[f](\theta,\omega,\Omega,t) = \frac1m\lt( -\omega + \Omega + K\int_{\T \times \R \times \R}\sin(\theta_* - \theta) f(\theta_*,\omega_*,\Omega_*,t)g(\Omega_*)\,d\theta_*d\omega_*d\Omega_*\rt).
\]
The kinetic equation \eqref{kinetic_Ku} is often used in the physics literature to study the phase transition phenomena \cite{ABS00, AS98}. The rigorous derivation of the equation \eqref{kinetic_Ku} from \eqref{particle_Ku} is established in \cite{CHY2}. The global-in-time existence of measure-valued solutions and its long-time behavior are also studied in \cite{CHY2}. We refer to \cite{Do, Ne, Sp} for the rigorous derivation of kinetic equations.

Note that the mesoscopic description model \eqref{kinetic_Ku} is posed in $3+1$ dimensions, thus obtaining a numerical solution of \eqref{kinetic_Ku} is computationally expensive. For this reason, we next derive a macroscopic description model from \eqref{kinetic_Ku} by taking moments together with zero temperature closure or mono-kinetic assumption for the local hydrodynamic solutions. In this way, we can remove $\omega$-variable in solutions. For this, we first set local density $\rho$, moment $\rho u$, and energy $\rho E$, which is the sum of kinetic and internal energies:
\[
\rho(\theta,\Omega,t) = \int_\R f\,d\omega, \quad (\rho u)(\theta,\Omega,t) = \int_\R \omega f\,d\omega, 
\]
and
\[
(\rho E)(\theta,\Omega,t) = \frac12 \rho u^2 + \rho e, \quad \mbox{where} \quad \rho e = \frac12 \int_{\R} |\omega - u|^2 f\,d\omega.
\]
Then straightforward computations yield 
$$\begin{aligned}
\pa_t \rho &= -\pa_\theta\lt( \int_\R \omega f \,d\omega\rt) = -\pa_\theta(\rho u),\cr
\pa_t (\rho u) &= -\pa_\theta\lt( \int_\R \omega^2 f\,d\omega\rt) + \int_\R F[f]f\,d\omega\cr
&=-\pa_\theta(\rho u^2 + p) - \frac1m\lt(\rho u - \rho \Omega - \rho K\int_{\T \times \R} \sin(\theta_* - \theta)\rho(\theta_*,\Omega_*,t)g(\Omega_*)\,d\theta_*d\Omega_*\rt),
\end{aligned}$$
where we denote by $p$ the pressure given by
\[
p = \int_\R |\omega - u|^2 f\,d\omega,
\]
and we used 
\[
 \int_\R \omega^2 f\,d\omega = \int_\R u^2 f\,d\omega + \int_\R |\omega - u|^2 f\,d\omega = \rho u^2 + p.
\]
Since 
\[
\frac12\omega^3 = \frac12(\omega - u + u)^3 = \frac12(\omega- u)^3 + \frac32(\omega - u)^2 u + \frac32 (\omega- u) u^2 + \frac12 u^3,
\]
we also find
$$\begin{aligned}
\pa_t (\rho E) &= -\pa_\theta\lt(\frac12\int_\R \omega^3 f\,d\omega \rt) + \int_\R \omega F[f]f\,d\omega\cr
&= -\pa_\theta \lt(q + \frac{3pu}{2} + \frac12 \rho u^3 \rt) - \frac2m \rho E + \frac1m  \rho u \Omega \cr
&\quad + \frac Km \rho u \int_{\T \times \R} \sin(\theta_* - \theta)\rho(\theta_*,\Omega_*,t)g(\Omega_*)\,d\theta_*d\Omega_*.
\end{aligned}$$
Here we denote by $q$ the heat flux given by
\[
q = \frac12 \int_\R (\omega - u)^3 f\,d\omega.
\]
By collecting the above observations, we have the following local conservation laws:
\begin{align}\label{lcl}
\begin{aligned}
&\pa_t \rho + \pa_\theta(\rho u) = 0,\cr
&\pa_t (\rho u) + \pa_\theta(\rho u^2 + p) \cr
&\qquad = - \frac{\rho}{m}\lt(u + \Omega - K\int_{\T \times \R} \sin(\theta_* - \theta)\rho(\theta_*,\Omega_*,t)g(\Omega_*)\,d\theta_*d\Omega_* \rt),\cr
&\pa_t (\rho E) + \pa_\theta \lt(q + \frac{3pu}{2} + \frac12 \rho u^3 \rt) \cr
&\qquad = -\frac{\rho}{m}\lt(2E + u\Omega - K u\int_{\T \times \R} \sin(\theta_* - \theta)\rho(\theta_*,\Omega_*,t)g(\Omega_*)\,d\theta_*d\Omega_* \rt).
\end{aligned}
\end{align}
In order to close the local conservation laws \eqref{lcl}, we use the mono-kinetic closure assumption:
\[
f(\theta,\omega,\Omega,t) \simeq \rho(\theta,\Omega,t)\delta_{u(\theta,\Omega,t)}(\omega).
\]
Then this reduces to our main system \eqref{hydro_Ku}. Although the mono-kinetic assumption is not fully justified, it is known that the hydrodynamic system derived gives quantitative results comparable to the particle simulations, see \cite{CCP17, CCZ, CKM10}.

\begin{remark}We can also employ another closure assumption, for instance, a local Maxwellian-type ansatz, 
\[
f(\theta,\omega,\Omega,t) \simeq \rho(\theta,\Omega,t) \exp\lt(- \frac{(\omega - u(\theta,\Omega,t))^2}{2}\rt),
\]
then we derive the isothermal Euler-type equations from \eqref{lcl}.
\end{remark}
In the current work, we first establish local-in-time existence and uniqueness of classical solutions to the system \eqref{hydro_Ku}. For this, we consider the moving domain problem and reformulate our main system \eqref{hydro_Ku} into the Lagrangian coordinate. To be more precise, let us define the characteristic flow $\eta(\theta,\Omega,t)$ by
\bq\label{char}
\pa_t\eta(\theta,\Omega,t) = u(\eta(\theta,\Omega,t),\Omega, t) \quad \mbox{with} \quad \eta(\theta,\Omega,0) = \theta.
\eq
Set 
\[
h(\theta,\Omega,t) := \rho(\eta(\theta,\Omega,t),\Omega,t) \quad \mbox{and} \quad v(\theta,\Omega,t) := u(\eta(\theta,\Omega,t),\Omega,t),
\]
and let us denote by the time-varying set $\mms_t := \{ (\theta,\Omega) \in \T \times \R : \rho(\theta,\Omega,t) \neq 0 \}$ for given initially bounded open set $\mms_0$. Using these newly defined notations, we can rewrite the system \eqref{hydro_Ku} along the characteristic flow given in \eqref{char} as
\begin{align}\label{lag_Ku}
\begin{aligned}
&h(\theta,\Omega,t)\pa_\theta \eta(\theta,\Omega,t) = \rho_0(\theta,\Omega), \quad (\theta,\Omega) \in \mathcal S_0, \quad t>0,\cr
&m\pa_t v(\theta,\Omega,t) + v(\tot) \cr
&\qquad = \Omega  + K\int_{\T \times \R} \sin(\eta(\tots) - \eta(\tot))\rho_0(\theta_*,\Omega_*)g(\Omega_*)\,d\theta_* d\Omega_*,
\end{aligned}
\end{align}
with the initial data
\bq \label{init_lag}
(h, v)(\theta,\Omega,0) = (\rho_0,u_0)(\theta,\Omega), \quad (\theta,\Omega) \in \mms_0.
\eq

For the system \eqref{lag_Ku}, we introduce a weighted Sobolev space $H^s_g$ by the distribution function $g$ and construct a unique $H^s_g$-solution. This newly defined solution space together with our careful analysis allows us to apply directly our strategy for the identical case, i.e., the distribution function $g$ is the form of the Dirac measure on $\R$ giving unit mass to the point $\Omega_0$, $g(\Omega) = \delta_{\Omega_0}(\Omega)$ for some $\Omega_0 \in \R$. We construct the approximated solutions and provide that they are Cauchy sequences in the proposed weighted Sobolev spaces by obtaining uniform bound estimates of approximated solutions. We then show that the limiting functions are solutions to \eqref{lag_Ku}. The details of proof are discussed in Section \ref{sec_local_thm}. It is worth noticing that our system is a type of pressureless Euler equations with nonlocal forces, and it is well known that the pressureless Euler system may develop a singularity such as a $\delta$-shock in finite-time, i.e., fail to admit a global classical solution, no matter how smooth the initial data are. This is one of the main difficulties in analyzing the Euler equations. However, for the identical case, in general the case where the distribution function $g$ is a sum of Dirac measures, we expect that the density $\rho$ converges toward a Dirac measure, see Remark \ref{rmk_dirac}. This infers that the existence time of $H^s_g$-solutions cannot be infinity for general $g$. 

\begin{remark}
Since $\pa_t v = \pa_{tt}\eta$, we can also rewrite the momentum equation in \eqref{lag_Ku} as 
\begin{align}\label{eq_eta}
\begin{aligned}
&m\pa_{tt}\eta(\tot) + \pa_t\eta(\tot) \cr
&\quad = \Omega + K\int_{\T \times \R} \sin(\eta(\tots) - \eta(\tot))\rho_0(\theta_*,\Omega_*)g(\Omega_*)\,d\theta_* d\Omega_*,
\end{aligned}
\end{align}
subject to the initial data
\[
(\eta,\pa_t\eta)(\theta,\Omega,0) = (\theta,u_0(\theta,\Omega)),\quad (\theta,\Omega) \in \mms_0.
\]
Note that the equation \eqref{eq_eta} is a closed equation, i.e., the continuity equation in \eqref{lag_Ku} is decoupled from the equation for $v$. 
\end{remark}

After we construct the local-in-time existence of classical solutions, we discuss the synchronization estimate for the case of identical oscillators in Section \ref{sec_sync}. In this case, upon rotating frame if necessary, we may assume that $g(\Omega) = \delta_0(\Omega)$ and the system \eqref{hydro_Ku} reduces to
\begin{align}\label{hydro_iKu}
\begin{aligned}
&\pa_t \rho + \pa_\theta(\rho u) = 0, \quad \theta \in \T, \quad t >0,\cr
&\pa_t (\rho u) + \pa_\theta(\rho u^2) = \frac1m \left( -\rho u +  \rho \Omega + K\rho\int_\T \sin(\theta_* - \theta)\rho(\theta_*,t)\,d\theta_*\right).
\end{aligned}
\end{align}
For the system \eqref{hydro_iKu}, we present two different methods for the synchronization estimates. Inspired by a recent work \cite{CHM18}, in Section \ref{sec_energy}, we propose a strategy based on kinetic energy combined with order parameter $r$ estimates, where $r$ is defined by
\bq\label{order_pa}
r(t)e^{\mathrm i \varphi(t)} = \int_{\T \times \R} e^{\mathrm i\theta}\rho(\theta,\Omega,t)g(\Omega)\,d\theta d\Omega,
\eq
i.e.,
\[
r(t) = \int_{\T \times \R} \cos(\theta - \varphi(t))\rho(\theta,\Omega,t)g(\Omega)\,d\theta d\Omega.
\]
Here $\varphi$ represents the average phase associated to the system \eqref{hydro_Ku}. Note that the order parameter $r$ is employed to measure the phase transition from incoherent to coherent states mentioned above. To be more specific, $r^\infty := \lim_{t \to \infty} r(t)$ as a function of the coupling strength $K$, i.e., $r^\infty = r^\infty(K)$, changes from zero (incoherent state or disordered state) to a non-zero value (coherent state or ordered state) when the coupling strength $K$ exceeds a critical value $K_c$. It is known that the critical coupling strength is $K_c = 2/(\pi g(0))$ for the classical Kuramoto model \eqref{particle_Ku} with $m=0$ when $g$ is unimodal and symmetric about $\Omega=0$. Our synchronization estimate provides the convergences of the velocity toward the mean velocity and the order parameter to some positive constant under  general assumptions on the initial configurations, this strategy can be applied for the case $\mms_0 = \T \times \R$, see Remark \ref{rmk_inic} and Theorem \ref{thm_gene}. However, it does not give any information about the decay rate of convergence and the limit profiles $\rho^\infty := \lim_{t \to +\infty} \rho$. 
\begin{remark}
Using newly defined notations in \eqref{order_pa}, we can rewrite the momentum equation in \eqref{hydro_Ku} as
\[
\pa_t (\rho u) + \pa_\theta(\rho u^2) = -\frac\rho m \lt( u -  \Omega - K r(t) \sin(\varphi(t) -\theta)\rt).
\]
For the system \eqref{lag_Ku}, the order parameter $r$ and the average phase $\varphi$ can be expressed by
\[
r(t) e^{\mathrm i\varphi(t)} = \int_{\T \times \R} e^{\mathrm i\eta(\theta,\Omega,t)}\rho_0(\theta,\Omega)g(\Omega)\,d\theta d\Omega,
\]
and the equation of velocity in \eqref{lag_Ku} can be rewritten as 
\begin{equation} \label{lag_order_form}
m\pa_t v(\theta,\Omega,t) + v(\tot) = \Omega + Kr(t)\sin(\varphi(t) - \eta(\theta,\Omega,t)).
\end{equation}
\end{remark}
In order to complement the drawbacks of the previous strategy, in Section \ref{sec_pv}, we provide a second-order Gr\"onwall-type inequality estimates on phase and velocity diameters for the synchronization. Note that the equation for velocity in \eqref{lag_Ku} is an integro-differential equation, and the equation \eqref{eq_eta} resembles the particle Kuramoto model with inertia \eqref{particle_Ku}. In view of this fact, we use the idea of \cite{CHY1} and estimate the phase and velocity diameters to show the exponential synchronization behavior under certain assumptions on the initial data. Although this approach requires more restricted class of initial data than that in the previous approach, it gives decay rates of convergences and shows the limit profiles $\rho^\infty$ is the form of Dirac measure, see Remark \ref{rmk_dirac}. 

In Section \ref{sec_cri}, we show that our main system \eqref{hydro_Ku} exhibits critical threshold phenomena. For this, we first study the local-in-time well-posedness of the system \eqref{hydro_Ku} if the initial data are sufficiently regular and the initial density has no vacuum. Then, we analyze critical thresholds determining regions of initial conditions for global-in-time existence and finite-time blow-up of solutions to the system \eqref{hydro_Ku}. More precisely, we provide thresholds between the supercritical region with finite-time breakdown and the subcritical region with global-in-time existence of the classical solutions. The critical threshold phenomenon for Eulerian dynamics is studied in \cite{ELT01,LT02,TW08} for Euler-Poisson equations and \cite{CCTT,TT14} for pressureless Euler equations with nonlocal velocity alignment forces. We want to emphasize that the finite-time blow-up of solutions cannot be observed in the Kuramoto model with inertia at both microscopic level \eqref{particle_Ku} and mesoscopic level \eqref{kinetic_Ku}.  As mentioned above, it is an important issue for the global existence of solutions to the Euler-type equations how to prevent the formation of singularity. However, this implies that our hydrodynamic model \eqref{hydro_Ku} may describe the finite-time synchronization phenomena \cite{LT03, YWC13} commonly found in some natural networks, see also Remark \ref{rem_comm_cri}. We investigate the supercritical region for the system \eqref{hydro_Ku} so that the classical solution will blow up in finite time if its initial data belong to that region. On the other hand, we show that if the initial data is in the subcritical region, then the initial regularity of solutions is preserved in time. These results are stated in Theorem \ref{thm_cri}. 

Numerical experiments validating our theoretical results and giving further insights are presented in Section \ref{sec_numer}. We employ a finite-volume type scheme for the numerical simulations. We use initial data and parameters lying in sub/supercritical regions to illustrate the time evolution of solutions $\rho$ and $u$. The numerical simulations show that the finite-time blow-up of solutions may imply that the formation of finite-time synchronization, see Figures \ref{fig_id}, \ref{fig_nid}, \ref{fig_nidsup}, and \ref{fig_nidsup2}. It is also very interesting that our main system \eqref{hydro_Ku} also exhibits the hysteresis phenomenon. Depending on the strength of $m$, we show different types of phase transitions of the order parameter $r^\infty(K)$, see Figure \ref{fig_hyst}. We would like to emphasize that our hydrodynamic model \eqref{hydro_Ku} is much more efficient than the $N$-particle system \eqref{particle_Ku} in terms of computational cost when $N$ is large. We finally summarize our main results and report future research directions in the last section.

Before closing this section, we introduce several notations used throughout the paper. For a function $f(\theta, \Omega)$, $\|f\|_{L^p}$ denotes the usual $L^p(\T \times \R)$-norm. $L^p_g(\T \times \R)$ represents the space of weighted measurable functions whose $p$-th powers weighted by $g = g(\Omega)$ are integrable on $\T \times \R$, with the norm
\[
\|f\|_{L^p_g} := \lt(\int_{\T \times \R} f^p(\theta,\Omega) g(\Omega)\,d\theta d\Omega\rt)^{1/p}.
\]
$f \ls g$ implies that there exists a positive constant $C>0$ such that $f \leq Cg$. We also denote by $C$ a generic, not necessarily identical, positive constant. For any nonnegative integer $s$, $H^s$ and $H^s_g$ denote the $s$-th order $L^2$ and $L^2_g$ Sobolev spaces, respectively.   $\mc^s([0,T];E)$ is the set of $s$ times continuously differentiable functions from an interval $[0,T] \subset \R$ into a Banach space $E$, and $L^p(0,T;E)$ is the set of $L^p$ functions from an interval $(0,T)$ to a Banach space $E$.

Throughout this paper, we assume that the distribution function $g = g(\Omega)$ satisfies
\bq\label{condi_g}
\int_\R g(\Omega)\,d\Omega = 1 \quad \mbox{and} \quad \int_\R \Omega^2 g(\Omega)\,d\Omega < \infty.
\eq
Note that the case of identical oscillators, $g(\Omega) = \delta_{\Omega_0}$ for some $\Omega_0 \in \R$, satisfies the above conditions \eqref{condi_g}. Furthermore, we assume that the initial density $\rho_0$ satisfies
\bq\label{condi_ir}
\int_\T \rho_0(\theta,\Omega)\,d\theta = 1 \quad \mbox{for} \quad \Omega \in \R.
\eq

%
%
%
%
\section{Preliminaries}\setcounter{equation}{0}
In this section, we present {\it a priori} energy estimates and some useful lemmas which will be frequently used later.
\subsection{A priori estimates}
We first provide {\it a priori} energy estimates for the system \eqref{hydro_Ku}.
\begin{lemma}\label{lem_energy} Let $(\rho,u)$ be a global classical solution to the system \eqref{hydro_Ku}. Then we have
$$\begin{aligned}
&(i)\,\,\,\,\,\, \frac{d}{dt}\int_{\T \times \R} \rho(\tot)g(\Omega)\,d\theta d\Omega = 0,\cr
&(ii)\,\,\,\, \frac{d}{dt}\int_{\T \times \R} \rho(\tot)\lt(u(\tot) - \Omega\rt) g(\Omega) \,d\theta d\Omega \cr
&\qquad = - \frac1m \int_{\T \times \R} \rho(\tot)\lt(u(\tot) - \Omega\rt) g(\Omega) \,d\theta d\Omega, \cr
&\mbox{ and}\cr
&(iii) \,\, \frac{d}{dt}\int_{\T \times \R} (\rho u^2)(\tot)g(\Omega) \,d\theta d\Omega\cr
&\qquad = - \frac2m\int_{\T \times \R} (\rho u^2)(\tot)g(\Omega)\,d\theta d\Omega + \frac1m\int_{\T \times \R} (\rho u)(\tot)\Omega g(\Omega)\,d\theta d\Omega\cr
&\qquad  \,\,-\frac {K}{m} \int_{\T^2 \times \R^2} \sin(\theta - \theta_*)(u(\tot) - u(\tots))\cr
&\hspace{4cm} \times \rho(\tots)g(\Omega_*)\rho(\tot)g(\Omega)\,d\theta d\theta_* d\Omega  d\Omega_*.
\end{aligned}$$
\end{lemma}
\begin{proof} (i) It clearly follows from the continuity equation in \eqref{hydro_Ku}. \newline

(ii) Multiplying the momentum equation in \eqref{hydro_Ku} by $g(\Omega)$ and integrating the resulting relation with respect to $\theta$ and $\Omega$, we find
\begin{align}\label{est_mv}
\begin{aligned}
&\frac{d}{dt}\int_{\T \times \R} (\rho u)(\tot)g(\Omega) \,d\theta d\Omega \cr
&\quad = -\frac1m\int_{\T \times \R} (\rho u)(\tot) g(\Omega) \,d\theta d\Omega + \frac1m \int_{\T \times \R} \rho(\tot)\Omega g(\Omega)\,d\theta d\Omega,\cr
&\quad = - \frac1m \int_{\T \times \R} \rho(\tot)\lt(u(\tot) - \Omega\rt) g(\Omega) \,d\theta d\Omega, \cr
\end{aligned}
\end{align}
since $\sin(-\theta) = -\sin(\theta)$ for $\theta \in \T$. On the other hand, it again follows from the continuity equation in \eqref{hydro_Ku} that
\[
\frac{d}{dt}\int_{\T \times \R} \rho(\tot)\Omega g(\Omega)\,d\theta d\Omega = 0.
\]
This together with \eqref{est_mv} asserts (ii). \newline

(iii) A straightforward computation yields
\begin{align*}
\frac 12 \frac{d}{dt}\int_{\T \times \R}\rho u^2\,d\theta d\Omega &= \frac 12 \int_{\T \times \R}u^2 \pa_t \rho \,d\theta d\Omega +  \int_{\T \times \R}\rho u \pa_t u \,d\theta d\Omega \\
&=-\frac 12  \int_{\T \times \R}u^2 \pa_\theta(\rho u) \,d\theta d\Omega +  \int_{\T \times \R}\rho u \pa_t u \,d\theta d\Omega \\
&= \int_{\T \times \R}\rho u^2 \pa_\theta u \,d\theta d\Omega +  \int_{\T \times \R}\rho u \pa_t u \,d\theta d\Omega\\
&=  \int_{\T \times \R}u(\pa_t(\rho u) + \pa_\theta (\rho u^2)) \,d\theta d\Omega.
\end{align*}
Then, we use the momentum equation in \eqref{hydro_Ku} to find
\begin{align}\label{est_energy1}
\begin{aligned}
&\frac12\frac{d}{dt}\int_{\T \times \R} (\rho u^2)(\tot)g(\Omega)\,d\theta d\Omega \cr
&\quad = -\frac1m\int_{\T \times \R} (\rho u^2)(\tot)g(\Omega)\,d\theta d\Omega + \frac1m\int_{\T \times \R} (\rho u)(\tot)\Omega g(\Omega)\,d\theta d\Omega\cr
&\qquad + \frac Km \int_{\T^2 \times \R^2} \sin(\theta_* - \theta)(\rho u)(\tot)\rho(\tots)g(\Omega_*)g(\Omega)\,d\theta d\theta_* d\Omega  d\Omega_*.
\end{aligned}
\end{align}
By exchanging $(\theta,\Omega)\leftrightarrow(\theta_*,\Omega_*)$, we can rewrite the last term on the right hand side of the above equation as
\begin{align}\label{est_energy2}
\begin{aligned}
&\frac12\frac Km \int_{\T^2 \times \R^2} \sin(\theta_* - \theta)\lt(u(\tot) - u(\tots)\rt)\cr
&\hspace{3cm} \times \rho(\tot)\rho(\tots)g(\Omega_*)g(\Omega)\,d\theta d\theta_* d\Omega  d\Omega_*.
\end{aligned}
\end{align}
Combining \eqref{est_energy1} and \eqref{est_energy2}, we have the desired result.
\end{proof}

\begin{remark}\label{rmk_21} From the continuity equation in \eqref{hydro_Ku}, we easily get
\[
\int_\T \rho(\tot)\,d\theta = \int_\T \rho_0(\theta,\Omega)\,d\theta = 1,
\]
due to \eqref{condi_ir}. Then this together with Lemma \ref{lem_energy} (i) and \eqref{condi_g} yields
\[
\int_{\T \times \R} \rho(\tot)g(\Omega)\,d\theta d\Omega = \int_{\T \times \R} \rho_0(\theta,\Omega)g(\Omega)\,d\theta d\Omega = \int_\R g(\Omega)\,d\Omega = 1.
\]
\end{remark}
\begin{remark}It follows from Lemma \ref{lem_energy} (ii) that 
$$\begin{aligned}
&\int_{\T \times \R} \rho(\tot)\lt(u(\tot) - \Omega\rt) g(\Omega) \,d\theta d\Omega \cr
&\quad = e^{-\frac tm}\int_{\T \times \R} (\rho_0(\theta,\Omega) (u_0(\theta,\Omega) - \Omega)g(\Omega) \,d\theta d\Omega \to 0 \quad \mbox{as} \quad t \to \infty.
\end{aligned}$$
In particular, if we consider the case of identical oscillators, i.e., $g(\Omega) = \delta_0$, upon shifting if necessary, then the above estimate gives the exponential decay of the momentum:
\[
\int_{\T} (\rho u)(\theta,t) \,d\theta = e^{-\frac tm}\int_{\T} (\rho_0 u_0)(\theta) \,d\theta \to 0 \quad \mbox{as} \quad t \to \infty.
\]
\end{remark}

\subsection{Auxiliary lemmas}
In this part, we present several useful lemmas that will be used out later. We first provide the exponential decay estimates for the nonnegative functions satisfying the following second-order differential inequality:
\begin{align} \label{CHY}
\begin{aligned}
&a x''(t) + b x'(t) +cx(t) \leq 0, \quad t > 0, \\
& x(0) = x_0, \quad  x'(0) = x_1,
\end{aligned}
\end{align}
where $a>0, b$ and $c$ are constants. We recall \cite[Lemma 3.1]{CHY1} the following inequalities. 
\begin{lemma} \label{lem_gronwall}
Let $x = x(t)$ be a nonnegative $\mathcal C^2$-function satisfying the differential inequality \eqref{CHY}. 
\begin{itemize}
\item[(i)] If $b^2 -4ac>0$, then we have \\
\[ 
x(t) \leq x_0 e^{-\alpha_1 t} + a\left(x_1 + \alpha_1 x_0\right)  \frac{e^{-\alpha_2 t}-e^{-\alpha_1 t}}{\sqrt{b^2-4ac}},
\]
where decay exponents $\alpha_1$ and $\alpha_2$ are given by
\[
\alpha_1 := \frac{b+\sqrt{b^2-4ac}}{2a} \quad \mbox{and} \quad \alpha_2:= \frac{b-\sqrt{b^2-4ac}}{2a},
\]
respectively.
\item[(ii)] If $b^2 - 4ac \leq 0$, then we have \\
\[
x(t) \leq e^{-\frac{b}{2a}t}\left(x_0 + \left(\frac{b}{2a}x_0 + x_1\right)t\right).
\]
\end{itemize}
\end{lemma}

We next provide the following simple lemma without the proof.
\begin{lemma}  \label{Barbalat}
Suppose that  a real-valued function $f: [0, \infty) \to \R$ is uniformly continuous and satisfies
\[ 
\lim_{t \to \infty} \int_0^t f(s) \,ds \quad \text{exists}. 
\]
Then, $f$ tends to zero as $t \to \infty$:
\[ 
\lim_{t \to \infty} f(t) = 0. 
\]
\end{lemma}
We also present a decay estimate for some differential equation, the proof of which can be found in \cite[Lemma 4.1]{CHM18}.
\begin{lemma}\label{lem_decay} Let $y =y(t)$ be a nonnegative $\mc^1$-function satisfying 
\[
y'(t) + \alpha y(t) = \beta(t) \quad t > 0, \quad y(0) = y_0,
\]
where $\alpha > 0$ and $\beta$ is a bounded continuous function decaying to zero as t goes to infinity. Then $y$ satisfies
\[
y(t) \leq y_0 e^{-\alpha t} + \frac1\alpha \max_{s \in [t/2, t]}|\beta(s)| + \frac{\|\beta\|_{L^\infty}}{\alpha} e^{-\frac{\alpha t}{2}}, \quad t \geq 0.
\]
In particular, $y$ tends to zero as $t$ goes to infinity. 
\end{lemma}
We finally recall from \cite[Lemma 2.4]{Kawa} the following Sobolev inequality.
\begin{lemma}\label{lem_kawa} Let $k\geq 1$. For $f \in (H^k \cap L^\infty)(\T)$, let $p \in [1,\infty]$, and $h \in \mc^k(B(0,\|f\|_{L^\infty}))$ where $B(0,R)$ denotes the ball of radius $R>0$ centered at the origin in $\R$. Then there exists a positive constant $C = C(k,p,h)$ such that
\[
\|\pa_\theta^k h(f)\|_{L^p} \leq C(1 + \|f\|_{L^\infty})^{k-1}\|\pa_\theta^k f\|_{L^p}.
\]
\end{lemma}

%
%
%
%

\section{Local-in-time existence and uniqueness of classical solutions}\label{sec_local_thm}
In this section, we present the local-in-time well-posedness of the Lagrangian system \eqref{lag_Ku}. More precisely, we show the local-in-time existence and uniqueness of $H^s_g$-solutions with $s\geq 1$ to the system \eqref{lag_Ku}.
\begin{theorem} \label{thm_local}
Let $s\geq 1$. Suppose that $(\rho_0,u_0) \in H_g^s(\mathcal S_0) \times H_g^{s+1}(\mathcal S_0)$. Then for any constants $0 < M < M'$ there exists $T_0 > 0$ depending only on $M$ and $M'$, such that if $\|u_0\|_{H_g^{s+1}} < M$, then the system \eqref{lag_Ku}-\eqref{init_lag} has a unique solution 
\[
(h,v) \in \mc([0,T_0];H_g^s(\mathcal S_0)) \times \mc([0,T_0];H_g^{s+1}(\mathcal S_0))
\] 
satisfying
\[
\sup_{0 \leq t \leq T_0}\|v(\cdot,\cdot,t)\|_{H_g^{s+1}} \leq M'.
\]
\end{theorem}
\begin{remark}Almost the same argument as above can be applied to the case of identical oscillators, i.e., $g(\Omega) = \delta_{\Omega_0}(\Omega)$. In this case, the weighted spaces $L^p_g(\T \times \R)$ and $H^s_g(\T \times \R)$ reduce to the spaces $L^p(\T)$ and $H^s(\T)$, respectively.
\end{remark}
\begin{proof}[Proof of Theorem \ref{thm_local}] Although the proof is similar to that of \cite[Theorem A.1]{CCZ}, for the completeness of our paper, we provide the details of it. We first approximate solutions of the system \eqref{lag_Ku} by the sequence $(\eta^n,v^n)$ which is the solution of the following integro-differential system:
\begin{align}\begin{aligned} \label{ap}
&\pa_t \eta^{n+1}(\theta,\Omega,t) = v^n(\theta,\Omega,t), \quad (\theta,\Omega) \in \mathcal S_0, \quad t>0, \\
&\pa_t v^{n+1}(\theta,\Omega,t) = -\frac{1}{m}v^{n+1}(\theta,\Omega,t) + \frac{\Omega}{m} + \frac{K}{m}f^{n+1}(\theta,\Omega,t),
\end{aligned}\end{align}
with the initial data and the first iteration step defined by
\[
(\eta^n(x,\Omega,t),v^n(x,\Omega,t))|_{t=0} = (\theta, u_0(\theta,\Omega)), \quad n \geq 1, \quad (\theta,\Omega) \in \mms_0 
\]
and 
\[
v^0(x,\Omega,t) = u_0, \quad (\theta,\Omega,t) \in \mathcal S_0 \times \R_+. \]
Here the interaction term $f^{n+1}$ is given by
\[
f^{n+1}(\theta,\Omega,t) := \int_{\T \times \R}\sin(\eta^{n+1}(\theta_*,\Omega_*,t)-\eta^{n+1}(\theta,\Omega,t))\rho_0(\theta_*,\Omega_*)g(\Omega_*)\,d\theta_* d\Omega_*.
\]
From now on, for the notational simplicity, we suppress the $\theta$- and $\Omega$-dependences of the variables and domain if the context is clear.\\

\noindent $\bullet$~(Step 1: Uniform bounds): We claim that there exists $T_0 >0$ such that
\[ \sup_{0 \leq t \leq T_0} \|v^n(\cdot,\cdot,t)\|_{H_g^{s+1}} \leq M' \quad \mbox{for}~ n \in \mathbb N \cup \{0\}. \]
\textit{(Proof of claim)}: We use an induction argument. In the first iteration step, we find that 
\[\sup_{0 \leq t \leq T}\|v^0(\cdot,\cdot,t)\|_{H_g^{s+1}}=\|u_0\|_{H_g^{s+1}} \leq M < M'.\]
Let us assume that
\[v^n \in \mc([0,T];H_g^s(\mms_0)) \quad \mbox{and} \quad \sup_{0 \leq t \leq T} \|v^n(\cdot, \cdot, t)\|_{H_g^{s+1}} \leq M',\]
for some $T>0.$ Then, we check that the linear approximations $(\eta^{n+1},v^{n+1})$ from the system $\eqref{ap}$ are well-defined, and they satisfy 
\[
(\eta^{n+1},v^{n+1}) \in \mc([0,T];H_g^{s+1}(\mms_0)) \times \mc([0,T];H_g^{s+1}(\mms_0)).
\]
We begin by estimating $\eta^{n+1}$. It follows from the equation of $\eta^{n+1}$ in \eqref{ap} that 
\begin{align*}\begin{aligned}
\eta^{n+1}(\theta,\Omega,t) &= \theta + \int_0^t v^n(\theta,\Omega,s)\,ds \quad \mbox{and} \quad \pa_\theta^k \eta^{n+1}(\theta,\Omega,t) = \delta_{k,1}+\int_0^t \pa_\theta^k v^n(\theta,\Omega,s)\,ds, 
\end{aligned}\end{align*}
for $k \geq 1$, where $\delta_{k,1}$ denotes Kronecker delta, i.e., $\delta_{k,1}=1$ for $k=1$ and $\delta_{k,1}=0$ otherwise. The equation of $v^{n+1}$ in \eqref{ap} gives
\begin{equation} \label{C-2}
v^{n+1}(\theta,\Omega,t) = u_0(\theta,\Omega)e^{-t/m} + \Omega(1 - e^{-t/m}) + \frac Km e^{-t/m}\int_0^t f^{n+1}(s)e^{s/m}\,ds. 
\end{equation}
Since 
\[
|f^{n+1}| \leq \int_{\T \times \R}\rho_0(\theta_*,\Omega_*)g(\Omega_*)\,d\theta_* d\Omega_* = 1,
\]
we find
\[
|v^{n+1}(\theta,\Omega,t)| \leq |u_0(\theta,\Omega)|e^{-t/m} + (|\Omega| + K)(1 - e^{-t/m}).
\]
Thus we obtain
\bq\label{est_vn}
\|v^{n+1}\|_{L^2_g} \leq \|u_0\|_{L^2_g} e^{-t/m} + C(1 - e^{-t/m}),
\eq
for some $C>0$, due to \eqref{condi_g}. For $1 \leq k \leq s+1$, by taking $\pa_\theta^k$ to \eqref{C-2}, we get
\[
\pa_\theta^k v^{n+1}(\theta,\Omega,t) = \pa_\theta^k u_0(\theta,\Omega)e^{-t/m} +  \frac Km e^{-t/m}\int_0^t \pa_\theta^k f^{n+1}(s)e^{s/m}\,ds. 
\]
We next use the facts $\sin \in \mc^\infty(\T)$ and $\|\sin\|_{L^\infty} \leq 1$ with Lemma \ref{lem_kawa} to estimate
\[
\int_\T (\pa_\theta^k f^{n+1})^2\,d\theta \leq C \int_\T (\pa_\theta^k \eta^{n+1})^2\,d\theta,
\]
for some $C>0$. This yields
$$\begin{aligned}
&\int_{\T \times \R} (\pa_\theta^k v^{n+1})^2 g\,d\theta d\Omega \cr
&\quad \leq Ce^{-2t/m}\int_{\T \times \R} (\pa_\theta^k u_0)^2 g\,d\theta d\Omega + \frac{CK}{m} e^{-2t/m}\int_{\T \times \R} \lt( \int_0^t \pa_\theta^k f^{n+1}(s)e^{s/m}\,ds \rt)^2 g\,d\theta d\Omega\cr
&\quad \leq Ce^{-2t/m} \|\pa_\theta^k u_0\|_{L^2_g}^2 + CK(1 - e^{-2t/m})\int_0^t \int_{\T \times \R} (\pa_\theta^k f^{n+1})^2 g\,d\theta d\Omega ds\cr
&\quad \leq Ce^{-2t/m}\|\pa_\theta^k u_0\|_{L^2_g}^2 + C(1 - e^{-2t/m})\int_0^t \|\pa_\theta^k \eta^{n+1}(\cdot,\cdot,s)\|_{L^2_g} \,ds.
\end{aligned}$$
We then use
\[
\|\pa_\theta^k \eta^{n+1}(\cdot,\cdot,s)\|_{L^2_g} \leq C + \int_0^t \|\pa_\theta^k v^n(\cdot,\cdot,s)\|_{L^2_g}\,ds \leq C+M'T
\]
to have
\[
\|\pa_\theta^k v^{n+1}\|_{L^2_g} \leq Ce^{-t/m}\|\pa_\theta^k u_0\|_{L^2_g} + C(1 - e^{-t/m}),
\]
for some $C>0$. Combing the estimate above with \eqref{est_vn} asserts
\bq\label{C-4}
\|v^{n+1}(\cdot,\cdot,t)\|_{H^{s+1}_g} \leq Ce^{-t/m}\|u_0\|_{H^{s+1}_g} + C(1 - e^{-t/m}).
\eq
Note that at $t=0$, the right hand side of \eqref{C-4} is $\|u_0\|_{H^{s+1}} \leq M < M'$, thus there exists a small time $0<T_0<T$ such that 
\[
\sup_{0 \leq t \leq T_0}\|v^{n+1}(\cdot,\cdot,t)\|_{H^{s+1}_g} \leq M'.
\]
\noindent $\bullet$~(Step 2: Cauchy estimates): For notational simplicity, we set
\[\eta^{n+1,n}:=\eta^{n+1}-\eta^n, \quad v^{n+1,n}:=v^{n+1}-v^n, \quad \mbox{and} \quad f^{n+1,n}:= f^{n+1}-f^n.\]
Then we easily find
\[
\eta^{n+1,n}(\cdot,\cdot,t)= \int_0^t v^{n,n-1}(\cdot,\cdot,s)\,ds,
\]
thus
\begin{equation} \label{C-5}
\|\eta^{n+1,n}(\cdot,\cdot,t)\|_{L^2_g} \leq \int_0^t \|v^{n,n-1}(\cdot,\cdot,s)\|_{L^2_g}\,ds,
\end{equation}
and also we have
\begin{equation} \label{C-5-1}
v^{n+1,n}(\theta,\Omega,t) = \frac Km e^{-\frac{1}{m}t}\int_0^t f^{n+1,n}(\theta,\Omega,s)e^{\frac{1}{m}s}\,ds.
\end{equation}
Note that 
$$\begin{aligned}
|f^{n+1,n}(\theta,\Omega,t)| &\leq \int_{\T \times \R}|\eta^{n+1,n}(\theta_*,\Omega_*,t)-\eta^{n+1,n}(\theta,\Omega,t)|\rho_0(\theta_*,\Omega_*)g(\Omega_*)\,d\theta_*d\Omega_* \\
&\leq \int_{\T \times \R} |\eta^{n+1,n}(\theta_*,\Omega_*,t)|\rho_0(\theta_*,\Omega_*)g(\Omega_*)\,d\theta_* d\Omega_*+\frac{K}{m}|\eta^{n+1,n}(\theta,\Omega,t)|\\
&\leq \|\eta^{n+1,n}(\cdot,\cdot,t)\|_{L^2_g}\|\rho_0\|_{L^2_g} + \frac{K}{m} |\eta^{n+1,n}(\theta,\Omega,t)|,
\end{aligned}$$
where we used H\"older's inequality and $\|\rho_0\|_{L^1_g} = 1$. Thus we get
\begin{equation} \label{C-6}
\|f^{n+1,n}(\cdot,\cdot,t)\|^2_{L^2_g} \leq C\left( 1 + \|\rho_0\|_{L^2_g}^2 \right)\|\eta^{n+1,n}(\cdot,\cdot,t)\|^2_{L^2_g}.  
\end{equation}
Then \eqref{C-5-1} and \eqref{C-6} yield
\begin{align}\begin{aligned} \label{C-5-2}
\|v^{n+1,n}(\cdot,\cdot,t)\|_{L^2_g} &\leq C\left( 1 + \|\rho_0\|_{L^2_g} \right)\int_0^t \|\eta^{n+1,n}(\cdot,\cdot,s)\|_{L^2_g} \,ds.
\end{aligned}\end{align}
By setting $\Delta^{n+1}_{\eta,v}(t):= \|\eta^{n+1,n}(\cdot,\cdot,t)\|_{L^2_g}+\|v^{n+1,n}(\cdot,\cdot,t)\|_{L^2_g}$, and then we combine \eqref{C-5} and \eqref{C-5-2} to get
\[
\Delta^{n+1}_{\eta,v}(t) \lesssim \int_0^t \Delta^{n+1}_{\eta,v}(s)\,ds,
\]
which implies that $\Delta^{n+1}_{\eta,v}(t) \lesssim T_0^{n+1}/(n+1)!$. Thus, we conclude that $(\eta^n(\theta,\Omega,t),v^n(\theta,\Omega,t))$ is a Cauchy sequence in $\mathcal C([0,T_0];L^2_g(\mathcal S_0)) \times \mathcal C([0,T_0];L^2_g(\mathcal S_0))$.\\

\noindent $\bullet$~(Step 3: Regularity of limit functions): It follows from Step 2 that there exist limit functions $\eta$ and $v$ such that \[
(\eta^n,v^n) \to (\eta,v) \quad \mbox{in} \quad \mathcal C([0,T_0];L^2_g(\mathcal S_0)) \times \mathcal C([0,T_0];L^2_g(\mathcal S_0)) \quad \mbox{as}~n \to \infty.\]
Interpolating this with the uniform bound estimates in Step 1, we obtain 
\begin{equation} \label{C-7}
(\eta^n,v^n) \to (\eta,v) \quad \mbox{in} \quad \mathcal C([0,T_0];H^s_g(\mathcal S_0)) \times \mathcal C([0,T_0];H^s_g(\mathcal S_0)) \quad \mbox{as}~n \to \infty.
\end{equation}
Now, we claim that \[(\eta,v) \in \mathcal C([0,T_0];H^{s+1}_g(\mathcal S_0)) \times \mathcal C([0,T_0];H^{s+1}_g(\mathcal S_0)). \]
Note that $v \in \mc([0,T_0];H^{s+1}_g(\mathcal S_0))$ implies $\eta \in \mc^{1}([0,T_0];H^{s+1}_g(\mathcal S_0))$ thanks to \eqref{C-7} and \eqref{ap}, so it suffices to show that $v \in \mc([0,T_0];H^{s+1}_g(\mathcal S_0))$. It follows from Step 1 that there exists a subsequence $(v^{n_k})$ such that 
\[v^{n_k} \rightharpoonup \tilde v \quad \mbox{as}~ k \to \infty, \quad \mbox{thus} \quad \|\tilde v(\cdot,\cdot,t)\|_{H^{s+1}_g} \leq 
\liminf_{k \to \infty} \|v^{n_k}(\cdot,\cdot,t)\|_{H^{s+1}_g} \leq M', \]
for some $\tilde v \in L^\infty([0,T_0];H^{s+1}_g(\mathcal S_0)).$ On the other hand, we also have 
\[v^n(\cdot,\cdot,t) \to v(\cdot,\cdot,t) \quad \mbox{in}~H^s_g(\mathcal S_0),~ t \in [0,T_0].\]
Thus, we conclude that
\[\tilde v(t)=v(t) \quad \mbox{in}~ H^{s+1}_g(\mathcal S_0) ~\mbox{for each}~ t \in [0,T_0],\]
which yields
\[\sup_{0 \leq t \leq T_0} \|v(\cdot,\cdot,t)\|_{H^{s+1}_g} \leq M'.\]
Next, we show that 
\begin{equation} \label{C-8}
v \in \mc_w([0,T_0];H^{s+1}_g(\mathcal S_0)), \quad \mbox{i.e.,} \quad v(t) \rightharpoonup v(t_0) \quad \mbox{in}~H^{s+1}_g(\mathcal S_0) \quad \mbox{as}~t \to t_0,
\end{equation}
for $t_0 \in [0,T_0]$. Without loss of generality, we may assume that $t_0=0$. Since $\|\cdot\|_{H^{s+1}}$ is weakly lower semicontinuous,
\begin{equation} \label{C-9}
\|u_0\|_{H^{s+1}_g} \leq \liminf_{t \to 0+}\|v(t)\|_{H^{s+1}_g}.
\end{equation}
To show the weak continuity, let $\{t_k\}$ be a sequence such that $t_k \to 0$,  $\{t_k\} \subset [0,T]$. For this sequence, we already have $\lim_{k \to \infty}\|v(t_k)-u_0\|_{H^s_g} =0$ and $\|u_0\|_{H^{s+1}_g} \leq M'$.
On the other hand, it is easy to see that \eqref{C-4} yields
\begin{equation} \label{C-10}
\limsup_{t \to 0+} \|v(t)\|_{H^{s+1}_g} \leq \|u_0\|_{H^{s+1}_g}.
\end{equation}
From \eqref{C-9} and \eqref{C-10}, we get
\[\|v(\cdot,\cdot,t)\|_{H^{s+1}_g} \to \|u_0\|_{H^{s+1}_g}, \quad \mbox{as}~t \to 0+.  \]
This together with the weak convergence \eqref{C-8}, we conclude that
\[\lim_{t \to t_0 +}\|v(t)-v_0\|_{H^{s+1}_g} = 0 \quad \mbox{for}~t_0 \in [0,T_0].\]
By considering the time-reversed problem, i.e., $t \mapsto -t$, we can also obtain the left continuity in the same way. \\

\noindent $\bullet$~(Step 4: Existence): In Step 3, we obtained \eqref{C-7}, and this implies that the limit functions $\eta$ and $v$ are the solutions of \eqref{char}-\eqref{lag_Ku} in the sense of distributions.  Moreover, we also have $h \in \mc([0,T_0];H^s_g(\mathcal S_0))$ since $\rho_0 \in H^s_g(\mathcal S_0)$ and $\eta \in \mc([0,T_0];H^{s+1}_g(\mathcal S_0))$. \\

\noindent $\bullet$~(Step 5: Uniqueness): Let $(h, \eta, v)$ and $(\tilde h,\tilde \eta, \tilde v)$ be the two classical solutions with the same initial data $(\rho_0,u_0)$. Let $\eta$ and $\tilde \eta$ be the trajectories with respect to $v$ and $\tilde v$, respectively, or, equivalently, 
\[\pa_t \eta(\theta,\Omega,t)=v(\theta,\Omega,t) \quad \mbox{and} \quad \pa_t \tilde \eta(\theta,\Omega,t)=\tilde v(\theta,\Omega,t).  \]
Then, by the similar arguments as in Step 2, we obtain the Gr\"onwall's inequality:
\[
\|v(\cdot,\cdot,t)-\tilde v(\cdot,\cdot,t)\|_{L^2_g} \leq C \int_0^t \|v(\cdot,\cdot,s)-\tilde v(\cdot,\cdot,s)\|_{L^2_g}\,ds, 
\]
which yields
\[v=\tilde v \quad \mbox{in}~\mc([0,T_0];L^2_g(\mathcal S_0)).  \]
Hence, one can easily check that 
\[v=\tilde v \quad \mbox{in}~\mc([0,T_0];H^{s+1}_g(\mathcal S_0)).  \]
using the similar argument in Step 3. Finally, this yields
\[h(\theta,\Omega,t)=\tilde h(\theta,\Omega,t) \quad \mbox{in}~\mc([0,T_0];H^s_g(\mathcal S_0)).  \]
\end{proof}

\begin{remark}
Theorem \ref{thm_local} gives the local-in-time regularity of solutions for the Cauchy problem in the Lagrangian coordinates \eqref{lag_Ku}. In order to go back to the Cauchy problem in the Eulerian coordinates \eqref{hydro_Ku}, we need to show that the characteristic flow $(\theta, \Omega) \mapsto (\eta(\theta,\Omega,t),\Omega)$ defined in \eqref{char} is a diffeomorphism for all $t \in (0,T_0)$ for some $T_0 > 0$. For this, it suffices to show that $\det(\nabla_{(\theta,\Omega)}(\eta(\theta,\Omega,t),\Omega)) = \pa_\theta \eta(\theta,\Omega,t) > 0$ for all $(\theta,\Omega,t) \in \T \times \R \times (0,T_0)$. However, it is unclear how to show this. On the other hand, for the identical oscillators, i.e., the system \eqref{hydro_iKu}, it follows from \eqref{char} that
\[
\pa_\theta \eta(\theta,t) = 1 + \int_0^t \pa_\theta v(\theta,s)\,ds.
\]
Then, by Theorem \ref{thm_local}, we have
\[
\pa_\theta \eta(\theta,t) \geq 1 - CM'T_0, \quad t \in [0,T_0].
\]
Thus, by choosing small enough $T_0 >0$, we obtain
\[
\pa_\theta \eta(\theta,t) > 0.
\]
This together with Theorem \ref{thm_local} concludes the local-in-time existence and uniqueness of solutions $(\rho, u)$ to the system \eqref{hydro_iKu} such that $(\rho, u) \in \mc([0,T_0]; H^s(\mathcal S_t)) \times \mc([0,T_0]; H^{s+1}(\mathcal S_t))$ for some $T_0 > 0$ with $(\rho_0,u_0) \in H^s(\mathcal S_0) \times H^{s+1}(\mathcal S_0)$. Here $\mathcal S_t = \{ \theta \in \T : \rho(\theta,t) \neq 0\}$.
\end{remark}
\begin{remark}\label{rmk_inic}We can also directly study the existence of classical solutions to the system \eqref{hydro_Ku} without introducing the Lagrangian formulation \eqref{lag_Ku} under the assumption that the initial density $\rho > 0$ in $\T \times \R$, see Theorem \ref{thm_local2}. In this case, however, we cannot use the strategy in Section \ref{sec_pv} for the synchronization estimate for the case of identical oscillators since it requires that the diameter of support of the initial density in phase is less than $\pi$, see Lemma \ref{lem_eta_bdd}. 
\end{remark}

%
%
%
%
\section{Synchronization estimates for identical oscillators }\label{sec_sync}
In this section, we provide synchronization estimates for identical oscillators, i.e., $g(\Omega) = \delta_{\Omega_0}(\Omega)$ for some $\Omega_0 \in \R$. Without loss of generality, upon rotating frame if necessary, we may assume $\Omega_0 = 0$. Note that in this case the system \eqref{hydro_Ku} reduces to
\begin{align*}
\begin{aligned}
&\pa_t \rho + \pa_\theta(\rho u) = 0, \quad \theta \in \T, \quad t >0,\cr
&\pa_t (\rho u) + \pa_\theta(\rho u^2) = -\frac1m \rho u + \frac Km\rho\int_\T \sin(\theta_* - \theta)\rho(\theta_*,t)\,d\theta_*,
\end{aligned}
\end{align*}
In the Lagrangian formulation, it is given by 
\begin{align} \label{lag_iKu}
\begin{aligned}
& h(\theta,t)\pa_\theta \eta(\theta,t) = \rho_0(\theta),\quad \theta \in \T, \quad t >0,\cr
&\pa_t v(\theta,t) + \frac{1}{m}v(\theta,t) = \frac{K}{m}\int_\T \sin(\eta(\theta_*,t)-\eta(\theta,t))\rho_0(\theta_*)\,d\theta_*.
\end{aligned}
\end{align}
As mentioned in Introduction, we propose two different types of strategies for the synchronization estimates in the following two subsections. 
\subsection{Kinetic energy estimate}\label{sec_energy} We introduce the mean velocity and mean phase:
\[
v_c(t) := \int_\T v(\theta,t) \rho_0(\theta)\,d\theta \quad \mbox{and} \quad \eta_c(t) := \int_\T \eta(\theta,t)\rho_0(\theta)\, d\theta,
\]
respectively, and the kinetic and potential energy functions $\me_k$ and $\me_p$:
\begin{align*}
&\me_k(t) := \frac12\int_\T (v(\theta,t) - v_c(t))^2 \rho_0(\theta)\,d\theta  \quad \mbox{and} \\ 
& \me_p(t) := \frac{K}{2m}\int_{\T \times \T}(1 - \cos(\eta(\theta_*,t) - \eta(\theta,t)))\rho_0(\theta_*)\rho_0(\theta)\,d\theta_*d\theta.
\end{align*}
Note that the quantities above can be reformulated in the Eulerian coordinate as follows: the mean velocity and mean phase are given by
\begin{align*}
u_c(t) := \int_\T u(\theta,t)\rho(\theta,t)\,d\theta \quad \mbox{and} \quad \theta_c(t) := \int_\T \theta \rho(\theta,t)\, d\theta,
\end{align*}
and the corresponding kinetic and potential energy functions are
\begin{align*}
&\me_k(t) = \frac12\int_\T (u(\theta,t) - u_c(t))^2\rho(\theta,t)\,d\theta  \quad \mbox{and}\\
&\me_p(t) = \frac{K}{2m}\int_{\T \times \T} (1 - \cos(\theta - \theta_*)) \rho(\theta,t)\rho(\theta_*,t)\,d\theta d\theta_*,
\end{align*}
respectively. 

\begin{lemma}\label{lem_e_id}Let $(\eta,h,v)$ be a global solution to the system \eqref{lag_iKu}. Then we have the following estimates.
\begin{itemize}
\item[(i)] Mean velocity estimate:
\[
v_c(t) = v_c(0) e^{-\frac tm },
\]
\item[(ii)] Mean phase estimate:
\[
\eta_c(t) = \eta_c(0) + m v_c(0)(1-e^{-\frac tm}),
\]
\item[(iii)] Total energy estimate: 
\[
\me_k(t) + \me_p(t) = \me_k(0) + \me_p(0) - \frac 2m \int_0^t \me_k(s)\,ds.
\]
\end{itemize}
\end{lemma}
\begin{proof} (i) 
Note that
\begin{align*}
\frac{d}{dt}v_c(t) &= \int_\T \pa_t v(\theta,t) \rho_0(\theta)\,d\theta \\
&= \int_\T -\frac 1m v(\theta,t)\rho_0(\theta)\,d\theta + \frac Km \int_{\T^2} \sin(\eta(\theta_*,t)-\eta(\theta,t))\rho_0(\theta_*)\rho_0(\theta)\,d\theta_*d\theta.
\end{align*}
Clearly, the second term vanishes and the desired estimate follows.\\

(ii) It directly follows from (i).\\

(iii) It is clear that
$$\begin{aligned}
\frac{d}{dt}\me_p(t) = \frac{K}{2m}\int_{\T \times \T}\sin(\eta(\theta_*,t) - \eta(\theta,t))(v(\theta_*,t) - v(\theta,t))\rho_0(\theta_*)\rho_0(\theta)\,d\theta_*d\theta.
\end{aligned}$$
On the other hand, we use the equation for $v$ in \eqref{lag_iKu} to find
\begin{equation*}
\begin{aligned}
&\frac{d}{dt}\me_k(t) \cr
&\quad = \int_\T (v(\theta,t)-v_c(t))(\pa_t v(\theta,t)-v_c'(t))\rho_0(\theta)\,d\theta \\
&\quad =  \int_\T (v(\theta,t) - v_c(t)) \pa_t v(\theta,t) \rho_0(\theta)\,d\theta \\
&\quad = \int_\T (v(\theta,t) - v_c(t))\lt(-\frac1m v(\theta,t)+ \frac{K}{m}\int_\T \sin(\eta(\theta_*,t) - \eta(\theta,t))\rho_0(\theta_*)\,d\theta_* \rt)  \rho_0(\theta)\,d\theta\cr
&\quad =:I_1 + I_2,
\end{aligned}
\end{equation*}
where $I_i,i=1,2$ can be estimated as follows.
$$\begin{aligned}
I_1 &= -\frac1m \int_\T (v(\theta,t) - v_c(t))^2 \rho_0(\theta)\,d\theta = -\frac 2m \me_k(t),\cr
I_2 &= \frac Km \int_{\T \times \T}\sin(\eta(\theta_*,t) - \eta(\theta,t))(v(\theta,t) - v_c(t))\rho_0(\theta_*)\rho_0(\theta)\,d\theta_* d\theta\cr
&= \frac {K}{2m} \int_{\T \times \T}\sin(\eta(\theta_*,t) - \eta(\theta,t))(v(\theta,t) - v(\theta_*,t))\rho_0(\theta_*)\rho_0(\theta)\,d\theta_* d\theta\cr
&= -\frac{d}{dt}\me_p(t).
\end{aligned}$$
Combining the above estimates concludes the desired result.
\end{proof}
\begin{remark}\label{rmk_ep}
The function $\me_p$ can be rewritten in terms of the order parameter defined in \eqref{order_pa}:
$$\begin{aligned}
\me_p(t) &= \frac{K}{2m} - \frac{K}{2m}\int_{\T \times \T}\cos(\eta(\theta_*,t)-\eta(\theta,t))\rho_0(\theta_*)\rho_0(\theta)\,d\theta_* d\theta \\
&= \frac{K}{2m} - \frac{K}{2m}r(t) \int_\T \cos(\varphi(t)-\eta(\theta,t))\rho_0(\theta)\,d\theta \\
&= \frac{K}{2m}(1 - r(t)^2).
\end{aligned}$$
\end{remark}
We now state our main results on the decay of kinetic energy $\me_k$ and the convergence of order parameter $r$.
\begin{theorem}\label{thm_gene}
Let $(\eta,h,v)$ be a global solution to the system \eqref{lag_iKu}. Then  we have
\[
\lim_{t \to \infty}\me_k(t) = 0 \quad \mbox{and} \quad \lim_{t \to \infty} r(t) = \sqrt{r_0^2 - \frac{2m}{K}\me_k(0) + \frac{4}{K}\int_0^\infty \me_k(s)\,ds}.
\]
\end{theorem}
\begin{proof} $\bullet$ (Decay of kinetic energy): In view of Lemmas \ref{Barbalat} and \ref{lem_e_id}, it suffices to show that the kinetic energy function $\me_k(t)$ is uniformly continuous since $\me_k \in L^1(0,\infty)$. Note that the kinetic energy $\me_k$ satisfies
$$\begin{aligned}
\frac{d}{dt}\me_k(t) + \frac2m\me_k(t) &= \frac Km \int_{\T \times \T}\sin(\eta(\theta_*,t) - \eta(\theta,t))(v(\theta,t) - v_c(t))\rho_0(\theta_*)\rho_0(\theta)\,d\theta_* d\theta\cr
&= I_2.
\end{aligned}$$
We then estimate $I_2$ as 
$$\begin{aligned}
I_2 &\leq \frac Km\int_\T |v(\theta,t) - v_c(t)|\rho_0(\theta)\,d\theta \cr
&\leq \frac{K}{m}\lt(\int_\T |v(\theta,t) - v_c(t)|^2\rho_0(\theta)\,d\theta\rt)^{1/2} \cr
&= \frac {\sqrt{2}K}{m} \me_k(t)^{1/2}.
\end{aligned}$$
This yields
\[
\frac{d}{dt}\me_k(t) + \frac2m\me_k(t) \leq \frac {\sqrt{2}K}{m}\me_k(t)^{1/2}.
\]
Since Lemma \ref{lem_e_id} (iii) implies  $\me_k(t) \leq \me_k(0) + \me_p(0)$, we obtain
$$\begin{aligned}
\lt|\frac{d}{dt}\me_k(t) \rt|& \leq \frac2m\me_k(t) + \frac {\sqrt{2}K}{m}\me_k(t)^{1/2} \cr
&\leq \frac 2m ( \me_k(0) + \me_p(0)) +  \frac {\sqrt{2}K}{m}(\me_k(0) + \me_p(0))^{1/2}.
\end{aligned}$$
$\bullet$ (Convergence of order parameter):  It follows from Lemma \ref{lem_e_id} that
\[
\me_k(t) + \frac{K}{2m}\lt(1 - r(t)^2\rt) + \frac{2}{m}\int_0^t \me_k(s)\,ds = \me_k(0) + \frac{K}{2m}\lt(1 - r_0^2\rt).
\]
Since $t \mapsto \int_0^t \me_k(s)\,ds$ is increasing and bounded by $(m/2)\me_k(0) + (K/4)(1 - r_0^2)$ from above, it converges. On the other hand, $\me_k(t)$ decays to zero as $t \to \infty$, and thus we get
$$\begin{aligned}
\lim_{t \to \infty} r(t) &= \lim_{t \to \infty}\sqrt{\frac{2m}{K}\me_k(t) + r_0^2 - \frac{2m}{K}\me_k(0) + \frac{4}{K}\int_0^t \me_k(s)\,ds} \cr
&= \sqrt{r_0^2 - \frac{2m}{K}\me_k(0) + \frac{4}{K}\int_0^\infty \me_k(s)\,ds},
\end{aligned}$$
due to $r \geq 0$. This completes the proof.
\end{proof}
In the rest of this subsection, we further study the time evolution of solutions to the system  \eqref{lag_iKu}. For this, we set 
\[
\ml(t) = \frac12 \int_\T \lt( v(\theta,t) + Kr(t)\sin(\eta(\theta,t) - \varphi(t))\rt)^2 \rho_0(\theta)\,d\theta.
\]
We then show the convergence of $\ml(t)$ to zero as $t$ goes to infinity in the proposition below.
\begin{proposition}Let $(\eta,h,v)$ be a global solution to the system \eqref{lag_iKu}. Then we have the following assertions.
\begin{itemize}
\item[(i)] If 
\bq\label{new_as}
r_0 > \sqrt{\frac{2m}{K}\me_k(0)},
\eq
then we have
\[
\inf_{t \geq 0} r(t) > 0.
\]
\item[(ii)] The function $\ml(t)$ decays to zero as $t \to \infty$:
\[
\ml(t) \to 0 \quad \mbox{as} \quad t \to \infty.
\]
In particular, if \eqref{new_as} holds, we have
\[
\int_\T \sin^2(\eta(\theta,t) - \varphi(t)) \rho_0(\theta)\,d\theta \to 0 \quad \mbox{as} \quad t \to \infty.
\] 
\end{itemize}
\end{proposition}
\begin{proof} (i) It is clear from Theorem \ref{thm_gene} that $\lim_{t \to \infty} r(t) > 0$.  Suppose that there exists $t_0 \in (0,\infty)$ such that 
\[
\lim_{t \to t_0-}r(t) = 0.
\]
On the other hand, it follows from Lemma \ref{lem_e_id} (iii) and Remark \ref{rmk_ep} that
\[
\me_k(t) - \frac{K}{2m} r(t)^2 + \frac2m \int_0^t \me_k(s)\,ds = \me_k(0) - \frac{K}{2m} r_0^2,
\]
and taking the limit $t \to t_0-$ gives the following contradiction.
\[
0 \leq \me_k(t_0) + \frac2m \int_0^{t_0} \me_k(s)\,ds = \me_k(0) - \frac{K}{2m} r_0^2 < 0.
\]
Thus we have
\[
\inf_{t \geq 0} r(t) > 0.
\]

(ii) First, we notice that the following identity holds:
\bq \label{order_reform}
r(t) \sin(\varphi(t)-\eta(\theta,t)) = \int_\T \sin(\eta(\theta_*,t)-\eta(\theta,t))\rho_0(\theta_*)\,d\theta_*.
\eq
Taking the time derivative to $\ml$, we find
$$\begin{aligned}
\frac{d}{dt}\ml(t) &= \int_\T \lt(v(\theta,t) + Kr(t)\sin(\eta(\theta,t) - \varphi(t))\rt)\cr
&\hspace{2cm} \times \lt(\pa_t v(\theta,t) + K \frac{\pa}{\pa t}\Big(r(t)\sin(\eta(\theta,t) - \varphi(t))\Big) \rt) \rho_0(\theta)\,d\theta\cr
&=: J_1 + J_2.
\end{aligned}$$
Using \eqref{lag_order_form}, we can easily estimate $J_1$ as 
\begin{align*}
J_1 &= \int_\T \lt(v(\theta,t) + Kr(t)\sin(\eta(\theta,t) - \varphi(t))\rt) \pa_t v(\theta,t) \rho_0(\theta)\,d\theta \\
&=- \frac 1m\int_\T \lt( v(\theta,t) + Kr(t)\sin(\eta(\theta,t) - \varphi(t))\rt)^2 \rho_0(\theta)\,d\theta \\
&= -\frac{2}{m}\ml(t).
\end{align*}
For the estimate of $J_2$, we use \eqref{order_reform} to find
$$\begin{aligned}
J_2 &= -K\int_{\T \times \T} \lt(v(\theta,t) + Kr(t)\sin(\eta(\theta,t) - \varphi(t))\rt)\cr
&\hspace{3cm} \times \cos(\eta(\theta_
*,t) - \eta(\theta,t))(v(\theta_*,t) - v(\theta,t))\rho_0(\theta_*)\rho_0(\theta)\,d\theta_* d\theta\cr
&\leq K\lt(\int_\T \lt(v(\theta,t) + Kr(t)\sin(\eta(\theta,t) - \varphi(t))\rt)^2 \rho_0(\theta)\,d\theta\rt)^{1/2}\cr
&\hspace{2cm} \times \lt(\int_{\T \times \T}\cos^2(\eta(\theta_
*,t) - \eta(\theta,t))(v(\theta_*,t) - v(\theta,t))^2\rho_0(\theta_*)\rho_0(\theta)\,d\theta_* d\theta\rt)^{1/2}\cr
&\leq K(2\ml(t))^{1/2}\lt(\int_{\T \times \T}(v(\theta_*,t) - v(\theta,t))^2\rho_0(\theta_*)\rho_0(\theta)\,d\theta_* d\theta\rt)^{1/2}.
\end{aligned}$$
Note that 
$$\begin{aligned}
\int_{\T \times \T}(v(\theta_*,t) - v(\theta,t))^2\rho_0(\theta_*)\rho_0(\theta)\,d\theta_* d\theta &= 2\int_{\T \times \T}(v(\theta_*,t) - v_c(t))^2\rho_0(\theta_*)\rho_0(\theta)\,d\theta_* d\theta\cr
&=4\me_k(t).
\end{aligned}$$
Thus we obtain
\[
J_2 \leq 2\sqrt2 K (\ml(t))^{1/2} (\me_k(t))^{1/2} \leq 2\sqrt2 K\lt(\e \ml(t) + \frac{1}{4\e}\me_k(t) \rt),
\]
for some $\e > 0$ which will be determined later.

Combining all of the above estimates, we have
\bq\label{est_ml}
\frac{d}{dt} \ml(t) + \frac{2}{m}\lt(1 - \sqrt2 K\e \rt)\ml(t) \leq \frac{\sqrt2 K}{2\e} \me_k(t).
\eq
We now choose $\e > 0$ small enough such that $1 - \sqrt2 K\e > 0$ and use Lemma \ref{lem_decay} and Thoerem \ref{thm_gene} to obtain
\bq\label{decay_ml}
\ml(t) \to 0 \quad \mbox{as} \quad t \to \infty.
\eq
Finally, since 
$$\begin{aligned}
&\int_\T K^2 r(t)^2 \sin^2(\eta(\theta,t) - \varphi(t))\rho_0(\theta)\,d\theta \cr
&\quad = \int_\T \lt( v(\theta,t) - v(\theta,t) + Kr(t)\sin(\eta(\theta,t) - \varphi(t))\rt)^2 \rho_0(\theta)\,d\theta\cr
&\quad \leq 4\ml(t) + 2\int_\T (v - v_c(t))^2 \rho_0(\theta)\,d\theta + 2v_c(t)^2,\cr
&\quad = 4\ml(t) + 4\me_k(t) + 2v_c(t)^2,
\end{aligned}$$
it follows from \eqref{decay_ml}, Thoerem \ref{thm_gene}, and Lemma \ref{lem_e_id} (i) that
\[
\int_\T \lt( r(t)\sin(\eta(\theta,t) - \varphi(t))\rt)^2 \rho_0(\theta)\,d\theta \to 0 \quad \mbox{as} \quad t \to \infty.
\]
This together with the result (i) concludes our desired result.
\end{proof}
\begin{remark} Since $\me_k \in L^1(0,\infty)$, we find from \eqref{est_ml} that $\ml \in L^1(0,\infty)$. This further yields 
\[
\int_0^\infty\int_\T \lt(r(t)\sin(\eta(\theta,t) - \varphi(t))\rt)^2 \rho_0(\theta)\,d\theta\,dt < \infty.
\]
Moreover, if the initial data satisfy \eqref{new_as}, we have
\[
\int_0^\infty\int_\T \lt(\sin(\eta(\theta,t) - \varphi(t))\rt)^2 \rho_0(\theta)\,d\theta\,dt < \infty.
\]
\end{remark}
\subsection{Phase \& velocity diameter estimates}\label{sec_pv}
In this part, we provide phase and velocity diameter estimates showing the exponential synchronization behavior under certain assumptions on the initial configurations. For this, we introduce the phase and velocity diameter functions as follows.
\[
d_\eta(t) := \max_{\theta, \theta_* \in \overline\mms_t}\,|\eta(\theta,t) - \eta(\theta_*,t)| \quad \mbox{and} \quad d_v(t) := \max_{\theta, \theta_* \in \overline\mms_t}\,|v(\theta,t) - v(\theta_*,t)|.
\]
For the synchronization estimates, we derive Gr\"onwall-type differential inequalities for $d_\eta$ and $d_v$. We first show differentiability of these functions in the lemma below.
\begin{lemma} \label{lem_diff}Let $T>0$ and $(\eta,h,v)$ be a solution to the system \eqref{lag_iKu} on the interval $[0,T]$.
Then, there exists a $T^* \in (0,T]$ such that 
\begin{itemize}
\item[(i)] $\pa_\theta \eta(\theta,t) > 0~\mbox{for all}~ (\theta,t) \in \T \times [0,T^*),$
\item[(ii)] $d_\eta'(t) ~\mbox{and}~ d_\eta''(t) ~\mbox{exist on the interval}~ [0,T^*).$
\end{itemize}
\end{lemma}

\begin{proof}
(i) Note that $\eta$ defined in \eqref{char} satisfies
\[
\eta(\theta,t) = \theta + \int_0^t v(\theta,s)\,ds \quad \mbox{and}\quad \pa_\theta \eta(\theta,t) = 1 + \int_0^t \pa_\theta v(\theta,s)\,ds.
\]
Thus, we get
\[
\pa_\theta \eta(\theta,t) > 1 - \sup_{0 \leq t \leq T}\| \pa_\theta v(\cdot,t)\|_{L^\infty} t \quad \mbox{for all}\quad t>0.
\]
Thus, the set defined by
\[
\mathcal T = \{t>0~|~\pa_\theta \eta(\theta,s) > 0, \quad 0 \leq s \leq t \}
\]
is nonempty, and thus the assertion (i) is obtained for $T^* := \sup \mathcal T$.\\

(ii) In view of (i), we find that there is no intersection between the characteristic curves starting from different points $\theta$ and $\theta_*$ on $[0,T^*)$. Accordingly, indices $M(t)$ and $m(t)$ which give $d_\eta(t) = \eta(\theta_{M(t)},t) - \eta(\theta_{m(t)},t)$ stay fixed on that time interval, i.e., the indices $M(t)$ and $m(t)$ are constants on $[0,T^*)$. This yields that $d_\eta$ and $d_\eta'$ are differentiable. 
\end{proof}

We now set $C_0$ and $D_0$ as follows.
\begin{equation*}
\begin{aligned}
C_0 := \max\lt\{d_\eta(0), d_\eta(0) + m d_v(0)\rt\}, \quad D_0 := \frac{\sin C_0}{C_0}.
\end{aligned}
\end{equation*}

\begin{lemma} \label{lem_eta_bdd} (Uniform boundedness of phase diameter) Let $T>0$ and $(\eta,h,v)$ be a solution to the system \eqref{lag_iKu} on the interval $[0,T]$. Suppose that the initial data satisfy $0 < C_0 < \pi.$
Then, we have
\[
d_\eta(t) \leq C_0, \quad 0 \leq t < T^*,
\]
where $T^*$  appeared in Lemma \ref{lem_diff} and $C_0 > 0$ is independent of $t$. 
\end{lemma}

\begin{proof}
Suppose that the phase diameter satisfies $d_\eta(t) < \pi$ on the interval $[0,T_0)$ for some $T_0  \leq T^*$ and we choose $\theta_{M(t)}$ and $\theta_{m(t)}$ such that $d_\eta(t) = \eta(\theta_{M(t)},t) - \eta(\theta_{m(t)},t)$ for $t \in [0,T_0)$. Then, by Lemma \ref{lem_diff}, we get $M(t) \equiv M$ and $m(t) \equiv m$ for some $M, m \in \overline\mms_t$ and thus $d_\eta(t) \in \mc^2((0,T_0))$. Thus we obtain from \eqref{lag_iKu} that 
$$\begin{aligned}
d_\eta''(t) + \frac{1}{m}d_\eta' (t) &= \frac Km \int_\T \lt(\sin(\eta(\theta_*,t)-\eta(\theta_M,t))-\sin(\eta(\theta_*,t)-\eta(\theta_m,t))\rt)\rho_0(\theta_*)\,d\theta_*  \\
&\leq 0,
\end{aligned}$$
for $0 < t < T_0$, which implies
\begin{equation} \label{E-1}
d_\eta(t) \leq d_\eta(0) + m\lt(1-e^{-\frac tm}\rt)d_\eta'(0) \leq C_0, \qquad 0 \leq t <T_0,
\end{equation}
due to $|d_\eta'(0)| \leq d_v(0)$. We then define a set
\[
\mathcal T_1 := \{t > 0~| ~ d_\eta(s) < C_0, ~ \forall s \in [0,t) \} \cap [0,T^*).
\]
Due to the assumption on the initial condition, $\mathcal T_1$ is nonempty. We now claim that $T_1^* := \sup \mathcal T_1 = T^*$. Suppose, contrary to our claim, that $T_1^* < T^*$. Then, the definition of $T_1^*$ gives
\begin{equation} \label{ctrd}
\lim_{t \to T_1^*-}d_\eta(t) = C_0.
\end{equation}
The relation \eqref{E-1}, however, yields 
\[
\lim_{t \to T_1^*-}d_\eta(t) \leq d_\eta(0) + m\lt(1-e^{-\frac {T_1^*}{m}}\rt)d_v(0) < C_0,
\]
which is contradictory to \eqref{ctrd}. Thus, we have $T_1^* = T^*$ and the conclusion readily follows. 
\end{proof}
Using the above uniform boundedenss of the phase diameter, we show the exponential decay of the phase diameter function on the time interval $[0,T^*)$. 
\begin{proposition} \label{prop_eta_decay} 
(Exponential decay of phase diameter)  Let $T>0$ and $(\eta,h,v)$ be a solution to the system \eqref{lag_iKu} on the interval $[0,T]$. Suppose that the initial data satisfy $0 < C_0 < \pi$.
\begin{itemize}
 \item [(i)] If $0 < 4mkD_0 < 1$, then we have
\[
d_\eta(t) \leq d_\eta(0)e^{-\nu_1 t} + m \frac{e^{-\nu_2 t}-e^{-\nu_1 t}}{\sqrt{1-4mK D_0}}(d_v(0) + \nu_1 d_\eta(0)), \quad t \in [0,T^*),
\]
where
\[
\nu_1 := \frac{1+\sqrt{1-4mK D_0}}{2m} \quad \mbox{and} \quad \nu_2 := \frac{1-\sqrt{1-4mK D_0}}{2m},
\]
 \item [(ii)] If $1 \leq 4mk D_0$, then we have
\begin{align*}
d_\eta(t) \leq e^{-\frac{1}{2m}t}\left(d_\eta(0) + \left(\frac{d_\eta(0)}{2m} + d_v(0)\right)t\right), \quad t \in [0,T^*).
\end{align*}
 \end{itemize}
 Here $T^*>0$ appeared in Lemma \ref{lem_diff}.
\end{proposition}
\begin{proof} Similarly as before, we choose $\theta_M \in \overline\mms_t$ and $\theta_m \in \overline\mms_t$ such that $d_\eta(t) = \eta(\theta_M,t) - \eta(\theta_m, t)$ on $[0,T^*)$. Then, by definition, we have
\[
d_\eta'(t) = v(\theta_M,t) - v(\theta_m,t), \quad d_\eta''(t) = v'(\theta_M,t) - v'(\theta_m,t), \quad t \in (0,T^*).
\]
Then we find from \eqref{lag_iKu} that $d_\eta$ satisfies
$$\begin{aligned}
d_\eta''(t) + \frac 1m d_\eta'(t) &= \frac Km \int_\T \lt(\sin(\eta(\theta_*,t)-\eta(\theta_M,t))-\sin(\eta(\theta_*,t)-\eta(\theta_m,t))\rt)\rho_0(\theta_*)\,d\theta_* \\
&=-\frac{2K}{m}\int_\T \cos\left(\eta(\theta_*,t) - \frac{\eta(\theta_M,t)+\eta(\theta_m,t)}{2} \right)\sin \lt(\frac{d_\eta(t)}{2}\rt)\rho_0(\theta_*)\,d\theta_* \\
&\leq -\frac{K}{m} \sin d_\eta(t), 
\end{aligned}$$
for $t \in (0,T^*)$, where we used Lemma \ref{lem_eta_bdd} and 
\[
\cos\left(\eta(\theta_*,t) - \frac{\eta(\theta_M,t)+\eta(\theta_m,t)}{2} \right) \geq \cos \lt(\frac{d_\eta(t)}{2}\rt).
\]
We now use the relation 
\[
\frac{\sin d_\eta(t)}{d_\eta(t)} > \frac{\sin C_0}{C_0} = D_0 \quad \mbox{for} \quad 0\leq d_\eta(t) < C_0 < \pi
\] 
to find 
\begin{equation} \label{E-2}
md_\eta''(t) + d_\eta'(t) + K D_0 d_\eta(t) \leq 0, \quad t \in (0,T^*).
\end{equation}
Using Lemma \ref{lem_gronwall} (i), we obtain
\[
d_\eta(t) \leq d_\eta(0)e^{-\nu_1 t} + m \frac{e^{-\nu_2 t}-e^{-\nu_1 t}}{\sqrt{1-4mK D_0}}(d_v(0) + \nu_1 d_\eta(0)), \quad t \in [0,T^*),
\]
due to $\nu_1 > \nu_2$ and $|d_\eta'(0)| \leq d_v(0)$. This proves $(i)$. The inequality (ii) can also be obtained by applying Lemma \ref{lem_gronwall} (ii) to \eqref{E-2}. This concludes the desired results.
\end{proof}

\begin{proposition} \label{prop_v_decay}
(Exponential decay of velocity diameter)  Suppose that $0 < C_0 < \pi$. Then, the following assertions hold.
\begin{itemize}
 \item [(i)] If $0 < 4mkD_0 < 1$, then we have
$$\begin{aligned}
d_v(t) &\leq \lt(d_v(0) + \frac{K(C_1 - d_\eta(0))}{1-m \nu_1}  - \frac{KC_1}{1 - m \nu_2}\rt)e^{-\frac tm}  \cr
&\quad + \frac{KC_1}{1 - m\nu_2 } e^{-\nu t} - \frac{K(C_1 - d_\eta(0))}{1 - m\nu_1} e^{-\nu_1 t}, 
\end{aligned}$$
for $t \in [0,T^*)$, where $C_1 > 0$ is given by
\[
C_1 :=\frac{m(d_v(0) + \nu_1 d_\eta(0))}{\sqrt{1-4mK D_0}}.
\]
 \item [(ii)] If $1 \leq 4mk D_0$, then we have
\[
d_v(t) \leq (1 + 4Km)d_v(0) e^{-\frac tm} + \lt( 2KC_2 t - 4Kmd_v(0)\rt)e^{-\frac{t}{2m}},
\]
for $t \in [0,T^*)$, where $C_2 > 0$ is given by
\[
C_2 := \frac{d_\eta(0)}{2m} + d_v(0).
\]
 \end{itemize}
\end{proposition}
\begin{proof}
In a similar fashion as before, we choose $\theta_{\widetilde M} \in \overline\mms_t$ and $\theta_{\widetilde m} \in \overline\mms_t$ such that $d_v(t) = v(\theta_{\widetilde M},t) - v(\theta_{\widetilde m}, t)$ on the time interval $[0,T^*)$. Then it follows from \eqref{lag_iKu} that $d_v$ satisfies 
$$\begin{aligned}
md_v'(t) + d_v(t) &\leq K \int_\T \lt(\sin(\eta(\theta_*,t)-\eta(\theta_{\widetilde M},t))- \sin(\eta(\theta_*,t)-\eta(\theta_{\widetilde m},t)) \rt)\rho_0(\theta_*)\,d\theta_*\\
&\leq Kd_\eta(t),  
\end{aligned}$$
for $t \in (0,T^*)$.
Applying Gr\"onwall's lemma to the inequality above gives
\[
d_v(t) \leq d_v(0) e^{-\frac tm} + \frac Km e^{-\frac tm} \int_0^t d_\eta(s) e^{\frac sm}\,ds.
\]
We then now use the upper bounds for $d_\eta$ obtained in Proposition \ref{prop_eta_decay} to have
$$\begin{aligned}
d_v(t) &\leq \lt(d_v(0) + \frac{K(C_1 - d_\eta(0))}{1-m \nu_1}  - \frac{KC_1}{1 - m \nu_2}\rt)e^{-\frac tm}  \cr
&\quad + \frac{KC_1}{1 - m\nu_2 } e^{-\nu t} - \frac{K(C_1 - d_\eta(0))}{1 - m\nu_1} e^{-\nu_1 t},
\end{aligned}$$
for $0 < 4mkD_0 < 1$, where 
\[
C_1 =\frac{m(d_v(0) + \nu_1 d_\eta(0))}{\sqrt{1-4mK D_0}}.
\]
For $1 \leq 4mkD_0$, we find
$$\begin{aligned}
d_v(t) &\leq \lt( d_v(0) + 4Km C_2 - 2Kd_\eta(0)\rt) e^{-\frac tm} \cr
&\quad + \lt( 2Kd_\eta(0) -4Km C_2+ 2KC_2 t\rt)e^{-\frac{t}{2m}},
\end{aligned}$$
where
\[
C_2 = \frac{d_\eta(0)}{2m} + d_v(0).
\]
\end{proof}
We are now in a position to state the exponential synchronization estimates for the system \eqref{lag_iKu} under an appropriate regularity assumptions on the solutions and smallness assumptions on the initial phase and velocity diameters. 
\begin{theorem} \label{thm_diam_decay}
Let $T>0$ and $(\eta,h,v)$ be a solution to the system \eqref{lag_iKu} on the interval $[0,T]$ with initial data $(h_0, v_0) \in H^2(\mms_0) \times H^3(\mms_0)$.  Suppose that  there exists $M > 0$ independent of $t$ such that 
\bq\label{ass_m}
\sup_{0 \leq t \leq T}\|v(\cdot,t)\|_{H^3} \leq M.
\eq
If the initial phase and velocity diameters $d_\eta(0)$, $d_v(0)$ are small enough, then there exist positive constants $c_1,c_2, \Lambda_1, \Lambda_2$, which are independent of $t$, such that
\[
d_\eta(t) \leq c_1 e^{-\Lambda_1 t}, \quad d_v(t) \leq c_2 e^{-\Lambda_2 t}, 
\]
for $t \in [0,T)$.
\end{theorem}
\begin{remark}Theorem \ref{thm_diam_decay} implies that as long as there exists a solution satisfying a certain regularity, the exponential decay of phase and velocity diameters can be obtained under some smallness assumptions on the initial phase and velocity diameters. It is worth noticing that we only require the smallness assumptions on the initial phase and velocity diameters, not initial data $(h_0, v_0) = (\rho_0, u_0)$.
\end{remark}
\begin{proof}[Proof of Theorem \ref{thm_diam_decay}]
It suffices to prove that $T^* = T$ in view of Lemma \ref{lem_diff}, Propositions \ref{prop_eta_decay} and \ref{prop_v_decay}.
 Assume to the contrary that $T^* < T$. Then, we have
\begin{equation} \label{D-2}
\lim_{t \to T^*-}\pa_\theta \eta(\theta,t) = 0.
\end{equation} 
Recall the notation $v_c(t) = \int_\T v(\theta,t)\rho_0(\theta)\,d\theta$, and note that
\begin{align}\label{D-2-1}
\begin{aligned}
|v(\theta,t)-v_c(t)| & = \left| \int_\T v(\theta,t)\rho_0(\theta_*)\,d\theta_* - \int_\T v(\theta_*,t)\rho_0(\theta_*)\,d\theta_* \right|   \leq d_v(t),
\end{aligned}
\end{align}
for $t \in [0,T^*)$.
Using the Sobolev embedding, we get 
\[
\| \pa_\theta v(\cdot,t) \|_{L^\infty} \leq C \|  \pa_\theta v(\cdot,t)\|_{H^1} = C \| \pa_\theta(v(\cdot,t)-v_c(t))\|_{H^1},
\]
and furthermore, by Gagliardo-Nirenberg interpolation inequality, we estimate
$$\begin{aligned}
\| \pa_\theta(v(\cdot,t)-v_c(t))\|_{L^2} &\leq C\|v(\cdot,t)-v_c(t)\|_{L^\infty}^{2/3}\|\pa_\theta^2 (v(\cdot,t)-v_c(t))\|_{L^2}^{1/3} \quad \mbox{and}\cr
\| \pa_\theta^2(v(\cdot,t)-v_c(t))\|_{L^2} &\leq C\|v(\cdot,t)-v_c(t)\|_{L^\infty}^{2/5}\|\pa_\theta^3 (v(\cdot,t)-v_c(t))\|_{L^2}^{3/5},
\end{aligned}$$
for $t \in [0,T^*)$. This together with \eqref{ass_m} and \eqref{D-2-1} asserts
\[
\| \pa_\theta v(\cdot,t) \|_{L^\infty} \leq C \|  \pa_\theta v(\cdot,t)\|_{H^1} \leq CM^{1/3}d_v(t)^{2/3} + CM^{3/5}d_v(t)^{2/5}.
\]
Note that it is clear from the smallness assumption on $d_\eta(0)$ and $d_v(0)$, and the estimates for $d_v(t)$ in Proposition \ref{prop_v_decay} that $d_v^{2/3}$ and $d_v^{2/5}$ are integrable in $(0,\infty)$. More precisely, we obtain
\[
\int_0^{T^*} (d_v(s)^{2/3} + d_v(s)^{2/5})\,ds \leq C(d_\eta(0), d_v(0)),
\]
where $C(d_\eta(0), d_v(0)) > 0$ is independent of $t$, and it satisfies
\[
\lim_{d_\eta(0), d_v(0) \to 0}C(d_\eta(0), d_v(0)) = 0.
\]
Thus for sufficiently small $d_\eta(0)$ and $d_v(0)$, we have
$$\begin{aligned}
\lim_{t \to T^*-}\pa_\theta \eta(\theta,t) &= 1 + \int_0^{T^*}\pa_\theta v(\theta,t)\,dt \cr
&\geq 1- \int_0^{T^*}\| \pa_\theta v(\cdot,t)\|_{L^\infty}\,dt \cr
&\geq 1 - C(d_\eta(0), d_v(0)).
\end{aligned}$$
This is a contradiction to \eqref{D-2} and completes the proof.
\end{proof}

\begin{remark}\label{rmk_dirac} Theorem \ref{thm_diam_decay} implies 
\[
d_{BL}(\rho(\theta,t)d\theta, \delta_{\eta_\infty}d\theta) \leq c_3 e^{-\Lambda_3 t}, \quad t \in [0,T),
\]
for some positive constants $c_3$ and $\Lambda_3$, where $d_{BL}$ denotes the bounded Lipschitz distance. Indeed, if we set 
\[
\mathbb D:= \lt\{\phi: \T \to \R~|~\|\phi\|_{\infty} \leq 1, ~ \mbox{Lip}(\phi):= \sup_{\theta \neq \theta_*} \frac{\phi(\theta) - \phi(\theta_*)}{| \theta - \theta_*|} \leq 1 \rt\}
\]
and $\eta_\infty:= \eta_c(0) + mv_c(0)$, then we have
$$\begin{aligned}
d_{BL}(\rho(\theta,t)d\theta, \delta_{\eta_\infty}d\theta) & = \sup_{\phi \in \mathbb D} \lt| \int_{\T}\phi(\theta)\rho(\theta,t)\,d\theta - \int_{\T}\phi(\theta)\delta_{\eta_{\infty}}\,d\theta \rt| \\
&= \sup_{\phi \in \mathbb D} \lt| \int_{\T} (\phi(\eta(\theta,t)) - \phi(\eta_{\infty}))\rho_0(\theta)\,d\theta \rt|\cr
&\leq \int_\T |\eta(\theta,t)-\eta_{\infty}| \rho_0(\theta)\,d\theta \\
&\leq \int_\T (|\eta(\theta,t)-\eta_c(t)| + |\eta_c(t)-\eta_\infty|) \rho_0(\theta)\,d\theta,
\end{aligned}$$
for $\phi \in \mathbb D$, where $\eta_c$ is the mean phase given by
\[
\eta_c(t) = \int_{\T}\eta(\theta,t)\rho_0(\theta)\,d\theta.
\]
The last integral can be estimated as follows. First, we get
\begin{align*}
|\eta(\theta,t)-\eta_c(t)| = \lt| \int_\T (\eta(\theta,t)-\eta(\theta_*,t))\rho_0(\theta_*)\,d\theta_* \rt| \leq d_\eta(t).
\end{align*}
For the second term, we use Lemma \ref{lem_e_id} (i) and the relation $\eta_c'(t) = v_c(t)$ to obtain
\[
\eta_c(t) = \eta_c(0) + \int_0^t v_c(0)e^{-\frac 1m s}\,ds = \eta_c(0) + mv_c(0)(1-e^{-\frac 1m t}) = \eta_\infty -mv_c(0)e^{-\frac 1m t}.
\]
Thus we have
\[
d_{BL}(\rho(\theta,t)d\theta, \delta_{\eta_\infty}d\theta) \leq d_\eta(t) + mv_c(0)e^{-\frac 1m t}, \quad t \in [0,T).
\]
Finally, we use Theorem \ref{thm_diam_decay} to conclude the desired result.
\end{remark}

%
%
%
%
\section{Critical thresholds phenomena}\label{sec_cri}
\setcounter{equation}{0}
In this section, we study critical thresholds phenomena in the system \eqref{hydro_Ku}. We first provide the local-in-time existence and uniqueness of solutions to the system \eqref{hydro_Ku}. 
\begin{theorem}\label{thm_local2} For any $0< N < M$, there is a positive $T_0 > 0$ such that if
\bq\label{asp_thm2}
\|\rho_0\|_{H^2_g} + \|\rho_0\|_{L^\infty} + \|u_0\|_{H^3_g} + \|\pa_\theta u_0\|_{L^\infty} < N \quad \mbox{and} \quad \rho > 0 \mbox{ in } \T \times \R,
\eq
then the Cauchy problem \eqref{hydro_Ku}-\eqref{ini_hydro_Ku} has a unique strong solution 
$$\begin{aligned}
\rho &\in \mc([0,T_0]; H^2_g(\T \times \R)) \cap L^\infty(\T \times \R \times (0,T_0)),\cr
u &\in  \mc([0,T_0]; H^3_g(\T \times \R)), \quad \mbox{and} \quad \pa_\theta u \in L^\infty(\T \times \R \times (0,T_0)) 
\end{aligned}$$
satisfying
\[
\sup_{0 \leq t \leq T_0} \lt(\|\rho(\cdot,\cdot,t)\|_{H^2_g} + \|\rho(\cdot,\cdot,t)\|_{L^\infty} + \|u(\cdot,\cdot,t)\|_{H^3_g} + \|\pa_\theta u(\cdot,\cdot,t)\|_{L^\infty} \rt) < M.
\]
\end{theorem}
\begin{proof} We notice that the local-in-time existence theory is well developed by now, however our solution space is weighted by the distribution function $g$ for natural frequencies. Compared to Theorem \ref{thm_local}, we need to be more careful because of the convection term in \eqref{hydro_Ku} which is nonlinear, and it does not appear in the Lagrangian system \eqref{lag_Ku}. For these reasons, we provide some details of the proof in Appendix \ref{app_local}. 
\end{proof}

Differentiating the momentum equation in \eqref{hydro_Ku} with respect to $\theta$ and letting $d := \pa_\theta u$, we rewrite the equation \eqref{hydro_Ku} as follows:
\begin{align}\label{cri_Ku}
\begin{aligned}
&D_t \rho + \rho d = 0,\quad (\theta,\Omega) \in \T \times \R, \quad t > 0,\cr
&D_t d + d^2 = -\frac dm - \frac Km \int_{\T \times \R} \cos(\theta_* - \theta)\rho(\theta_*,\Omega_*,t)g(\Omega_*)\,d\theta_* d\Omega_*,
\end{aligned}
\end{align}
where $D_t$ denotes the time derivative along the characteristic flow $\eta(\theta,\Omega,t)$, i.e., $D_t = \pa_t + u \pa_\theta $. 

\begin{proposition}\label{prop_cri} Consider the system \eqref{cri_Ku}. Then the following assertions hold.
\begin{itemize} 
\item[(i)] {\bf (Subcritical region)} If $1 \geq 4Km$ and 
\[
d_0(\theta,\Omega) \geq \frac{-1 - \sqrt{1 - 4Km}}{2m},
\]
then $d(\eta(\tot),\Omega,t)$ remains bounded from below for $(\theta,\Omega) \in \T \times \R$ and $t \geq 0$.
\item[(ii)] {\bf (Supercritical region)} If 
\[
d_0(\theta,\Omega) < \frac{-1 - \sqrt{1 + 4Km}}{2m},
\] 
then $d(\eta(\tot),\Omega,t) \to - \infty$ in finite time.
\end{itemize}
\end{proposition}
\begin{proof}
(i)~ Note that the interaction term in \eqref{cri_Ku} is easily bounded by
\[
\lt|\int_{\T \times \R}\cos(\theta_* - \theta)\rho(\theta_*,\Omega_*,t)g(\Omega_*)\,d\theta_* d\Omega_*\rt| \leq \int_{\T \times \R}\rho(\theta_*,\Omega_*,t)g(\Omega_*)\,d\theta_* d\Omega_* = 1,
\]
due to Lemma \ref{lem_energy} (i), \eqref{condi_g}, and \eqref{condi_ir}. This yields
\bq\label{est_cri}
-\frac Km \leq D_t d + d^2 + \frac dm \leq \frac Km.
\eq
Let us consider the first inequality in \eqref{est_cri}. It can be rewritten as
\[
D_t d \geq -\lt(d^2 + \frac dm + \frac Km\rt) = -(d - d_-)(d-d_+),
\]
where $d_\pm$ is given by
\[
d_\pm := \frac{-1 \pm \sqrt{1 - 4Km}}{2m} \quad \mbox{for} \quad 1 \geq 4Km.
\]
For $d_- \leq d_0 <d_+$, $d(t) \geq d_0$ by continuity argument. We now let $q$ solve the following Riccati's equation:
\[
D_tq = -(q-d_-)(q-d_+), \quad q(0)=d_0.
\]
The solution of this equation is explicitly given as follows.
\[
q(t) = \frac{d_+(d_0-d_-)e^{(d_+ - d_-)t}+d_-(d_+-d_0)}{(d_0-d_-)e^{(d_+ - d_-)t}+d_+-d_0},
\]
and it is easy to see that $q(t) \geq d_+$ if $d_0 \geq d_+.$ The comparison principle yields  $d(t) \geq q(t) \geq d_+$ for $d_0 \geq d_+$. Thus, we have
\[
\left\{ \begin{array}{ll}
d(t) \geq d_0 & \textrm{if $d_0 \in [d_-, d_+)$}, \\[2mm]
 d(t) \geq d_+ & \textrm{if $d_0 \in [d_+,\infty)$},
  \end{array} \right.
\]
for $t \geq 0$. \newline

(ii)~In a similar fashion as above, for the second inequality in \eqref{est_cri}, we get
\[
D_t d \leq -\lt(d^2 + \frac dm - \frac Km\rt) \leq -(d - d_-^*)(d-d_+^*) \quad \mbox{with} \quad d_\pm^* := \frac{-1 \pm \sqrt{1 + 4Km}}{2m}.
\]
This together with the continuity argument implies that if $d_0 < d_-^*$, then $d(t) \leq d_0 < d_{-}^*$. This readily gives
\[
D_t d \leq -(d - d_-^*)^2, \quad \mbox{i.e.,} \quad D_t (d - d_-^*) \leq -(d - d_-^*)^2,
\]
and, subsequently, solving the above differential inequality yields
\[
d(t) \leq \frac{1}{(d_0 - d_-^*)^{-1} + t} + d_-^*.
\]
Therefore, $d(t)$ will diverge to $-\infty$ until the time $T < (d_-^* - d_0)^{-1}$. This completes the proof.
\end{proof}

\begin{remark} \label{rem_density}
In the subcritical case, it follows from the continuity equation in $\eqref{cri_Ku}$ that
\[
\rho(\eta(\theta,\Omega,t),\Omega, t) = \rho_0(\theta,\Omega)e^{-\int_0^t d(\eta(\theta,\Omega,s),\Omega, s)ds} \leq \rho_0(\theta,\Omega)e^{-t d_-}.
\]
Thus, $\rho$ cannot attain $+\infty$ in a finite time. On the other hand, for the supercritical case, we see
\begin{align*}
\int_0^t d(s)\, ds \leq \int_0^t \left(\frac{1}{(d_0-d_-^*)^{-1}+s} + d_-^*\right) ds = \ln | 1- (d_-^*-d_0)t| + d_-^*,
\end{align*}
and thus $\rho$ can be estimated as
\[
\rho(\eta(\theta,\Omega,t),\Omega, t) = \rho_0(\theta,\Omega)e^{-\int_0^t d(\eta(\theta,\Omega,s),\Omega, s)\,ds} \geq \frac{\rho_0(\theta,\Omega)e^{-d_-^* t}}{|1-(d_-^*-d_0)t|}.
\]
This implies $\rho$ diverges to $+\infty$ until the time $T < (d_-^* - d_0)^{-1}$.
\end{remark}

\begin{remark} In the case of no interactions between oscillators, i.e., $K = 0$, the momentum equation in \eqref{cri_Ku} reduces to the damped pressureless Euler system:
\[
D_t d = -d^2 - \frac dm = -d\lt(d + \frac1m\rt).
\]
Thus we obtain a sharp critical thresholds:
\begin{itemize}
\item If $d_0(\theta,\Omega) < -1/m$, then $d(\eta(\tot),\Omega,t) \to -\infty$ in finite time.
\item If $d_0(\theta,\Omega) \geq -1/m$, then $d(\eta(\tot),\Omega,t)$ remains bounded for $(\theta,\Omega) \in \T \times \R$ and $t \geq 0$.
\end{itemize}
\end{remark}

We next provide {\it a priori} estimates of solutions $(\rho,u)$ to the system \eqref{hydro_Ku}. In the proposition below, we show that $\|\rho\|_{H^2_g}$ and $\|u \|_{H^3_g}$ can be controlled by $\|\rho\|_{L^\infty}$ and $\|\pa_\theta u\|_{L^\infty}$. \begin{proposition}\label{prop_ap}Let $s \geq 0$ be an integer and consider the system \eqref{hydro_Ku}. Then, for any $T > 0$, we have
$$\begin{aligned}
&\sup_{0 \leq t \leq T}\lt(\|\rho(\cdot,\cdot,t)\|_{H^2_g} + \|u(\cdot,\cdot,t)\|_{H^3_g} \rt) \cr
&\qquad \leq \lt(\|\rho_0\|_{H^2_g} + \|u_0\|_{H^3_g}\rt) \exp\lt(C\int_0^T \lt( 1 + \|\rho(\cdot,\cdot,s)\|_{L^\infty} + \|\pa_\theta u(\cdot,\cdot,s)\|_{L^\infty}\rt)dt \rt).
\end{aligned}$$
\end{proposition}
\begin{proof} Replacing $\bar u$ in the proof of Theorem \ref{thm_local2} by $u$, we obtain from \eqref{est_rho3} and \eqref{est_u3} that
\bq\label{est_rho3}
\frac{d}{dt} \|\rho\|_{H^2_g} \leq C\|\pa_\theta u\|_{L^\infty}\|\rho\|_{H^2_g} + C\|\rho\|_{L^\infty}\|\pa_\theta u\|_{H^2_g}
\eq
and
\bq\label{est_u3}
\frac{d}{dt}\|u\|_{H^3_g} \leq C\|\pa_\theta u \|_{L^\infty}\|u\|_{H^3_g} + C.
\eq
On the other hand, we find
\[
1 = \int_{\T \times \R} \rho(\theta_*,\Omega_*,t)g(\Omega_*)\,d\theta_* d\Omega_* \lesssim \|\rho\|_{L^2_g} \leq \|\rho\|_{H^2_g},
\]
due to \eqref{condi_g}. This together with combining \eqref{est_rho3} and \eqref{est_u3} gives
\[
\frac{d}{dt}\lt(\|\rho\|_{H^2_g} + \|u\|_{H^3_g} \rt) \leq C\lt(1 + \|\rho\|_{L^\infty} + \|\pa_\theta u\|_{L^\infty} \rt)\lt(\|\rho\|_{H^2_g} + \|u\|_{H^3_g} \rt).
\]
Applying Gr\"onwall's lemma to the inequality above, we conclude the desired result.
\end{proof}
As a direct consequence of Theorem \ref{thm_local2}, Propositions \ref{prop_cri} and \ref{prop_ap}, we have the following results for the critical thresholds phenomena in \eqref{hydro_Ku}. 
\begin{theorem}\label{thm_cri} Let $T>0$ and consider the system \eqref{hydro_Ku}. Under the assumptions of Theorem \ref{thm_local2}, we have the following assertions.
\begin{itemize} 
\item[(i)] {\bf (Subcritical region)} If $1 \geq 4Km$ and 
\[
\pa_\theta u_0(\theta,\Omega) \geq \frac{-1 - \sqrt{1 - 4Km}}{2m} \quad \mbox{for all} \quad (\theta, \Omega) \in \T \times \R,
\] 
then the system has a global solution, $(\rho,u) \in \mc([0,T); H^2_g(\T \times \R)) \times \mc([0,T); H^3_g(\T \times \R))$.
\item[(ii)]
{\bf (Supercritical region)} If there exists $(\theta_*,\Omega_*)$ such that 
\[
\pa_\theta u_0(\theta_*,\Omega_*) < \frac{-1 - \sqrt{1 + 4Km}}{2m},
\]
then the solution blows up in finite time.
\end{itemize}
\end{theorem}
\begin{remark}For reasons mentioned before, almost the same argument as above can be applied to the case of identical oscillators, i.e., $g(\Omega) = \delta_{\Omega_0}(\Omega)$ by replacing the weighted spaces $L^p_g(\T \times \R)$ and $H^s_g(\T \times \R)$ by $L^p(\T)$ and $H^s(\T)$, respectively.
\end{remark}
\begin{remark}\label{rem_comm_cri} The results of Theorem \ref{thm_cri} (ii) and Remark \ref{rem_density} (ii) give some possible finite-time synchronization. It is very hard to expect the finite-time synchronization phenomena in the classical Kuramoto models, for instances \eqref{particle_Ku} and \eqref{kinetic_Ku}, with a smooth distribution function $g(\Omega)$ for natural frequencies. However, as mentioned in Introduction, our hydrodynamic model \eqref{hydro_Ku} is the pressureless Euler-type system, and thus it may form singularities in finite time. It is unclear though how to rigorously justify that this finite-time blow-up of solutions implies the finite-time synchronization. With regard to this matter, we will numerically examine the time evolution of solutions to the system \eqref{hydro_Ku} in the next section.
\end{remark}
%
%
%
%
\section{Numerical experiments}\label{sec_numer}\setcounter{equation}{0}
In this section, we present several numerical experiments validating our theoretical results for the system \eqref{hydro_Ku}. We also numerically examine that our system \eqref{hydro_Ku} exhibits the phase transitions and hysteresis phenomena like the particle system \eqref{particle_Ku}. For the numerical integration of the system \eqref{hydro_Ku}, the finite volume method is used, in particular, we employ Kurganov-Tadmor central scheme proposed in \cite{KT} for the evaluation of numerical fluxes. A brief description of the scheme is provided below.
\subsection{Numerical scheme} Note that the system \eqref{hydro_Ku} can be written in the following form.
\begin{align*}
\begin{aligned}
\pa_t q + \pa_\theta F(q) = G(q).
\end{aligned}
\end{align*}
Here we set
\[
q = \begin{pmatrix} \rho \\ u \end{pmatrix} , \quad F(q) = \begin{pmatrix} \rho u \\ u^2/2 \end{pmatrix},
\]
and the source term is given by
\[ G(q)= \left(0, \frac 1m\big(-u + \Omega + K\int_{\T \times \R}\sin(\theta_* - \theta)\rho(\theta_*,\Omega_*,t)g(\Omega_*)\,d\theta_* d\Omega_* \big) \right)^T.
\]
The cell average $Q_j$ over the grid cell $\mathcal C_{j} = (\theta_{j-1/2},\theta_{j+1/2})$ at time $t$ is given by
\[
Q_j = \frac{1}{\Delta \theta}\int_{\mathcal C_j}q(\theta,t)\,d\theta,
\]
where $\Delta \theta = \theta_{j+1/2} - \theta_{j-1/2}$. Then, for given $\Omega$, the finite volume method is formulated as follows:
\[
\frac{d}{dt}Q_{j}(t) = - \frac{F^*_{j+1/2}(t)-F^*_{j-1/2}(t)}{\Delta \theta} + G^*_j(t).
\]
Here, $F^*_{j \pm 1/2}$ denotes the numerical flux through the cell interface at $\theta_{j \pm 1/2}$, which will be given later, and $G_j^*$ is an associated source term evaluated at $\theta = \theta_j$ where the integration is performed using the midpoint rule. The {\it reconstruction} first requires a piecewise linear function
\begin{equation} \label{piecewise}
\tilde q^n(x,t_n) = Q_j^n + \sigma_j^n(\theta-\theta_j), \qquad \theta_{j-1/2} \leq \theta < \theta_{j+1/2},
\end{equation}
where $\theta_j$ is the center of $\mathcal C_j$, and $\sigma_j^n$ denotes an approximation to the spatial derivative on $\mathcal C_j$. In order to prevent nonphysical oscillations, we use the slope limiter method which was introduced by van Leer  \cite{van}. In particular, we use {\it minmod slope} here:
\[
\sigma_j^n = \mbox{minmod}\left(\frac{Q_j^n - Q_{j-1}^n}{\Delta \theta}, \frac{Q^n_{j+1}-Q^n_j}{\Delta \theta} \right),
\]
where the minmod function is defined by 
\begin{align*}
\mbox{minmod}(a,b) = \begin{cases}
a \quad  &\mbox{if} ~|a| < |b| ~\mbox{and}~ ab>0,\\
b \quad &\mbox{if} ~|b| < |a| ~\mbox{and}~ ab>0,\\
0 \quad &\mbox{if}~ ab \leq 0.
\end{cases}
\end{align*}
The numerical fluxes are now evaluated as
\begin{align*}
F^*_{j+1/2}:= \frac{a^+_{j+1/2}F(q_j^E)-a^-_{j+1/2}F(q^W_{j+1})}{a^+_{j+1/2}-a^-_{j+1/2}} + \frac{a^+_{j+1/2}a^-_{j+1/2}}{a^+_{j+1/2}-a^-_{j+1/2}}(q^W_{j+1} - q^E_j).
\end{align*}
Here, $a^+_{j+1/2}$ and $a^-_{j+1/2}$ denote the largest and the smallest speed of characteristic at the cell interfaces. The reconstructed values at the cell interface $\theta_{j+1/2}$ using \eqref{piecewise} are given by
\[
q_j^E = Q_j^n +\sigma _j^{n} \frac{\Delta \theta}{2}, \qquad q_{j+1}^W = Q_{j+1}^n - \sigma_{j+1}^n \frac{\Delta \theta}{2}.
\]
Finally, the second-order Runge--Kutta method is employed for time integrations.

\subsection{Time evolutions of density and velocity}
In this subsection, we present the time evolutions of the density and the velocity profiles at different time $t$'s for both identical and nonidentical oscillators  case.  For the identical case, the $\Omega$-dependences of $\rho$ and $u$ are simply neglected, see Section \ref{sec_sync}. We identify $\T$, the $\theta$-domain, as $[-\pi,\pi]$ for the numerical computation domain and set the initial density $\rho_0$ and the distribution function $g(\Omega)$ the standard normal distribution. The numbers of $\theta$-grid and $\Omega$-grid are $1000$ and $600$, respectively.

\begin{figure}[ht]
\centering
\subfigure[Density - subcritical case]{  
\includegraphics[width=2.8in]{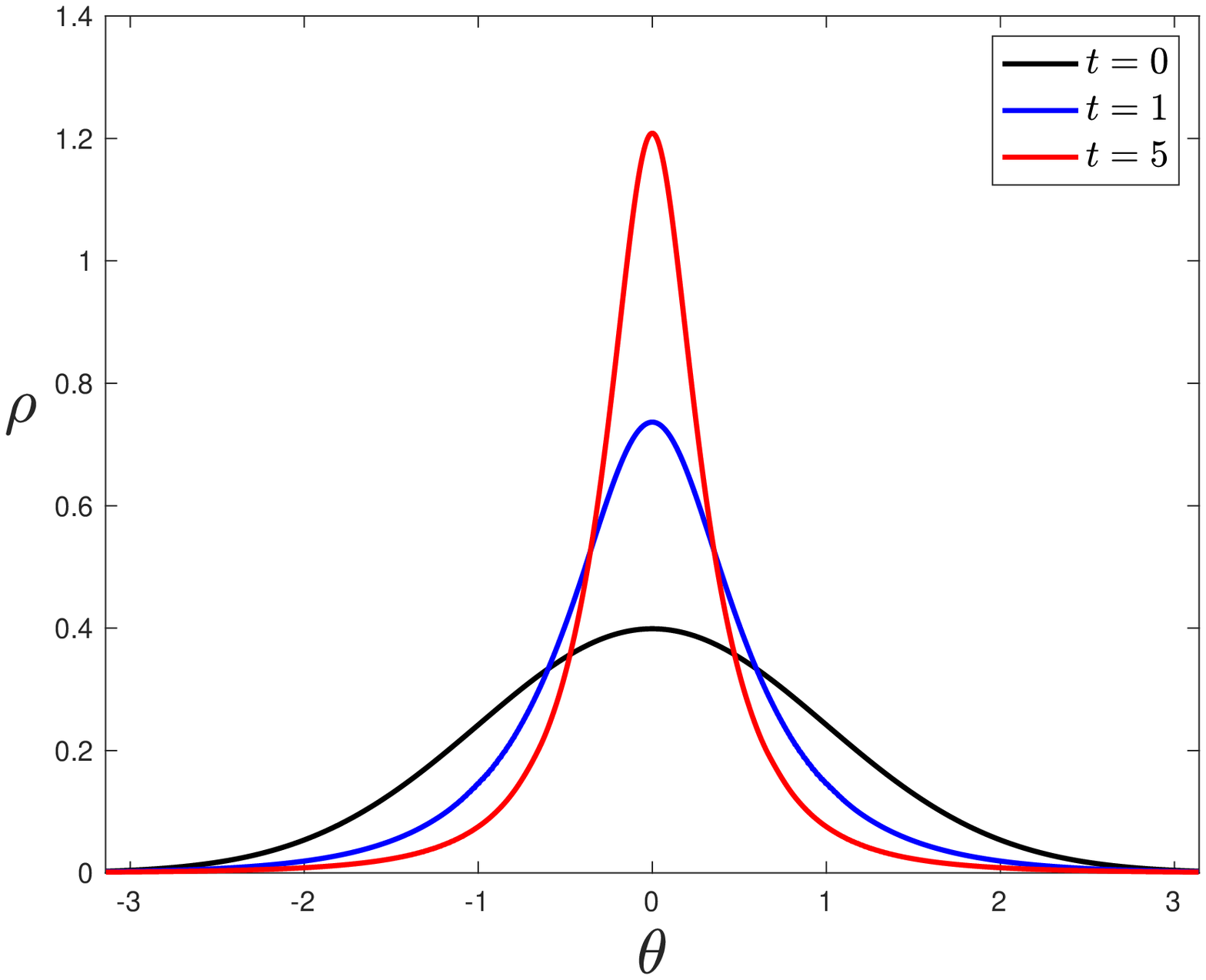}} 
\subfigure[Velocity - subcritical case]{  
\includegraphics[width=2.8in]{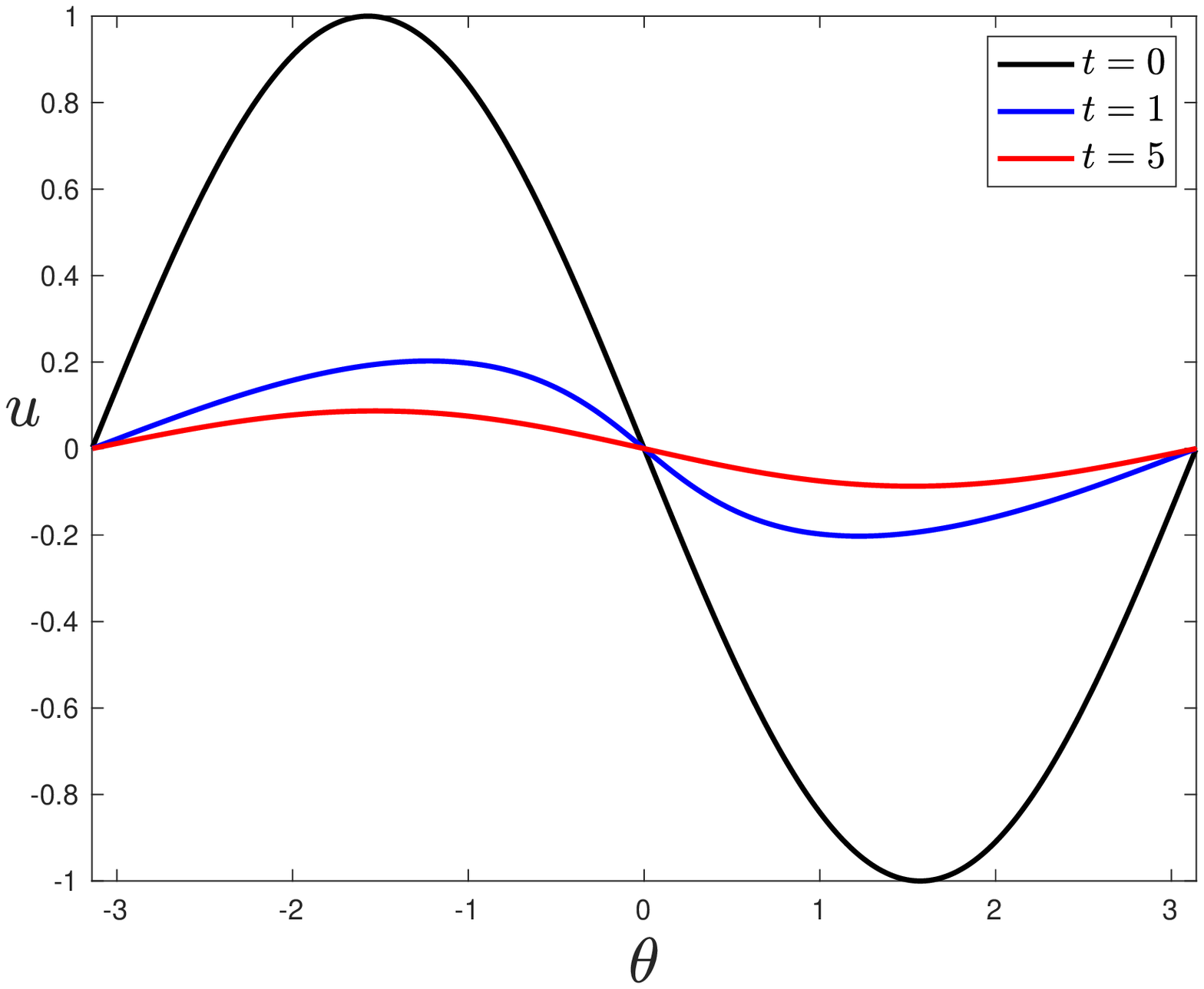}} 
\subfigure[Density - supercritical case] { 
\includegraphics[width=2.8in]{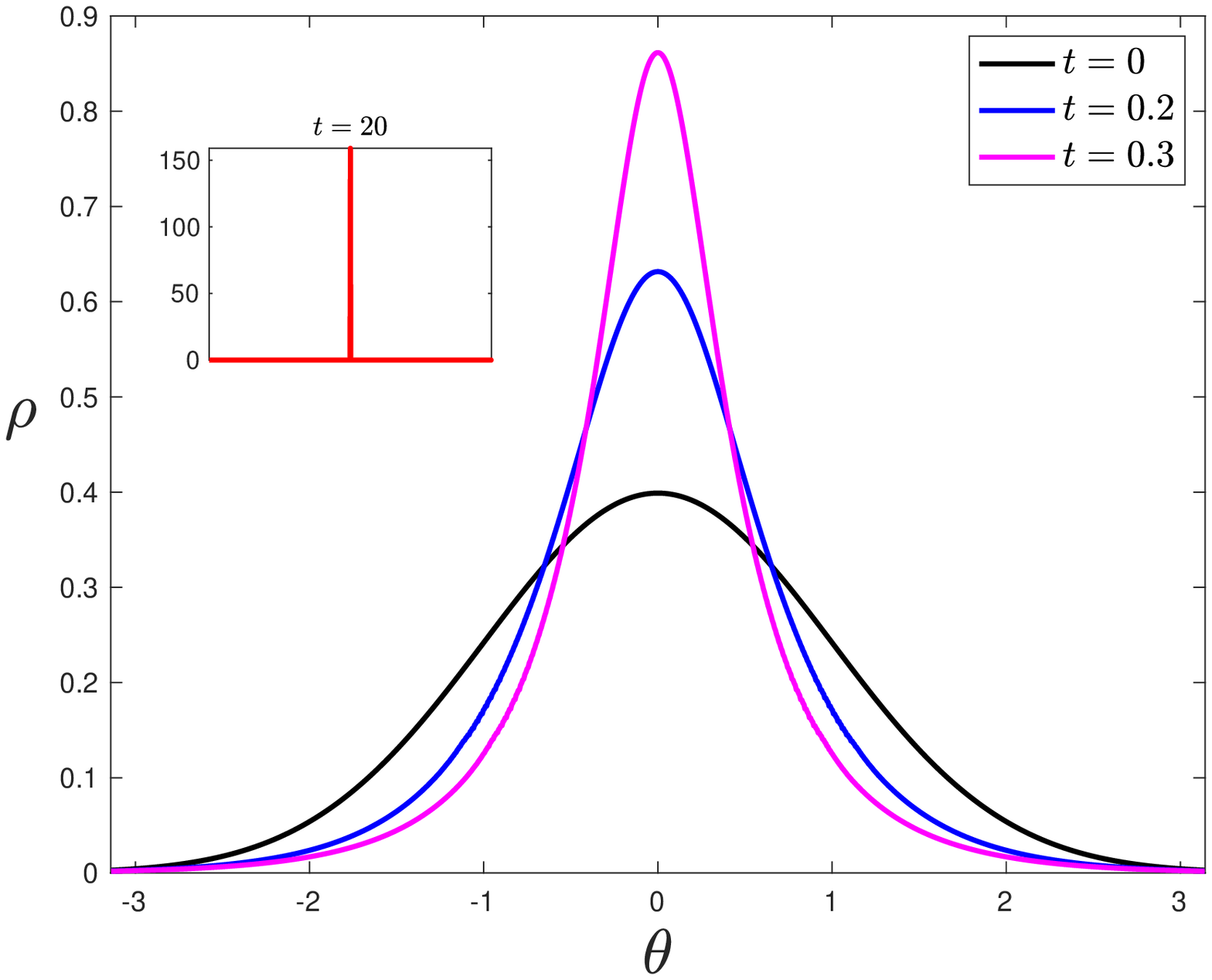}} 
\subfigure[Velocity - supercritical case]{  
\includegraphics[width=2.8in]{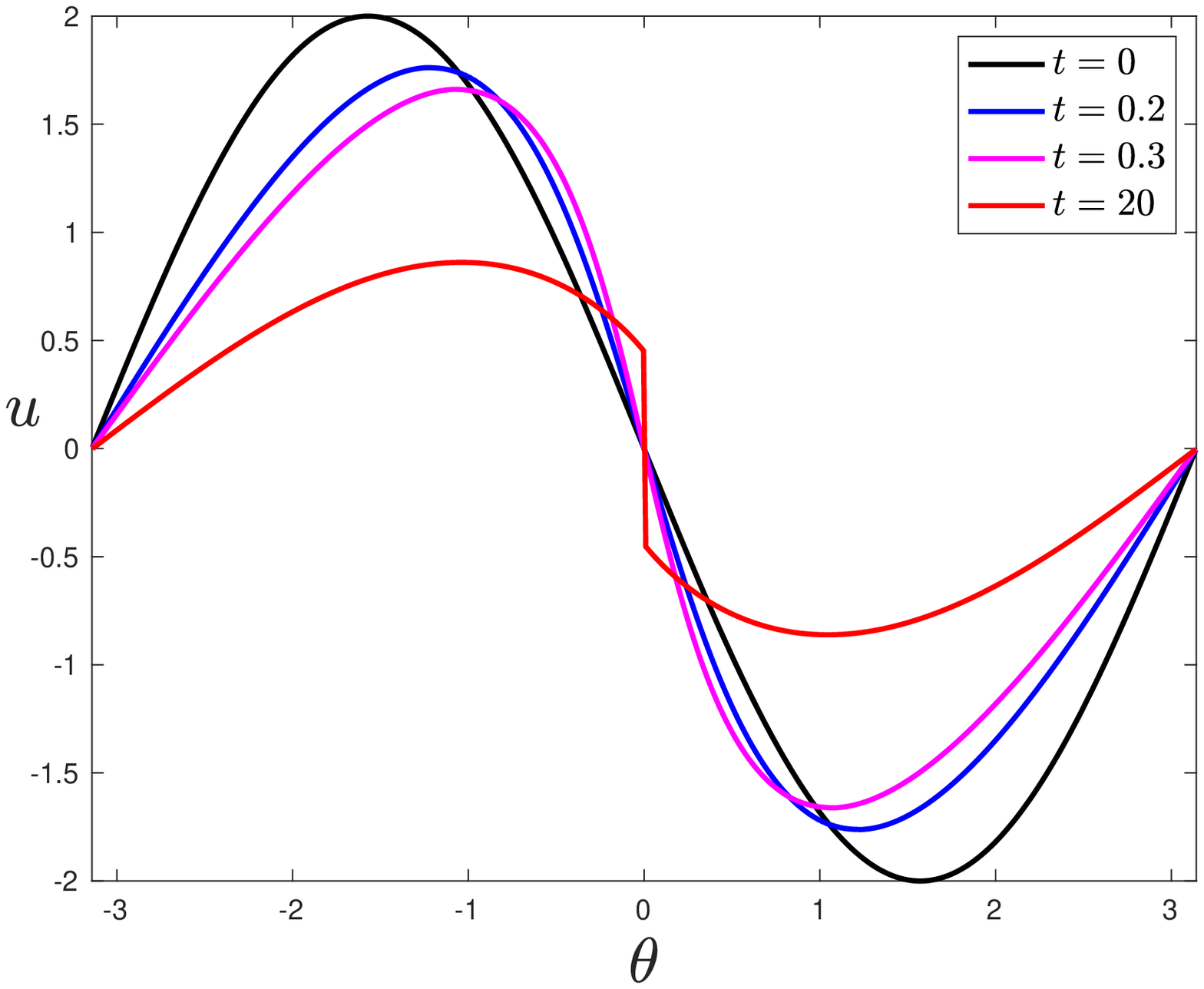}} 
\subfigure[Density - supercritical case] { 
\includegraphics[width=2.8in]{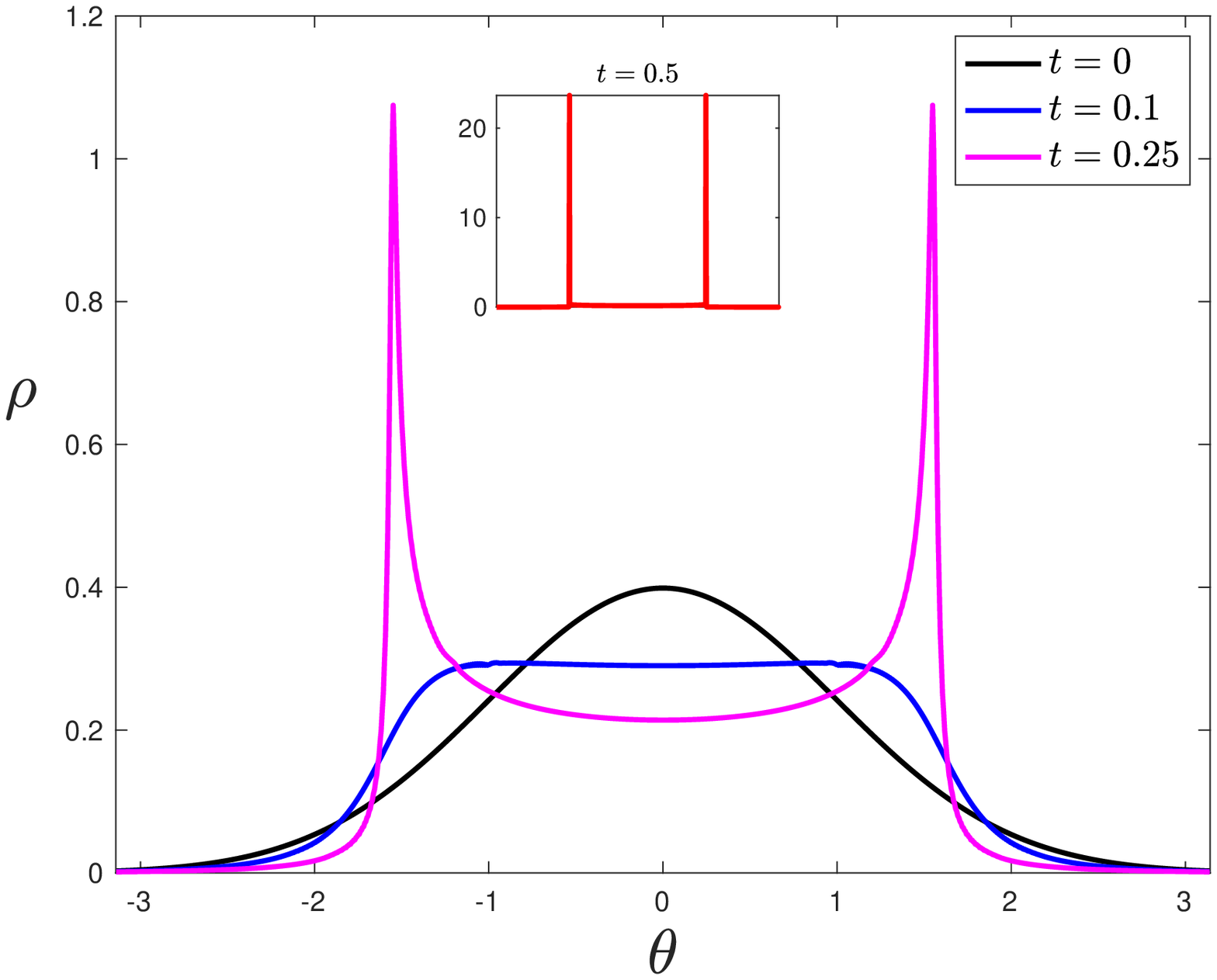}} 
\subfigure[Velocity - supercritical case]{  
\includegraphics[width=2.8in]{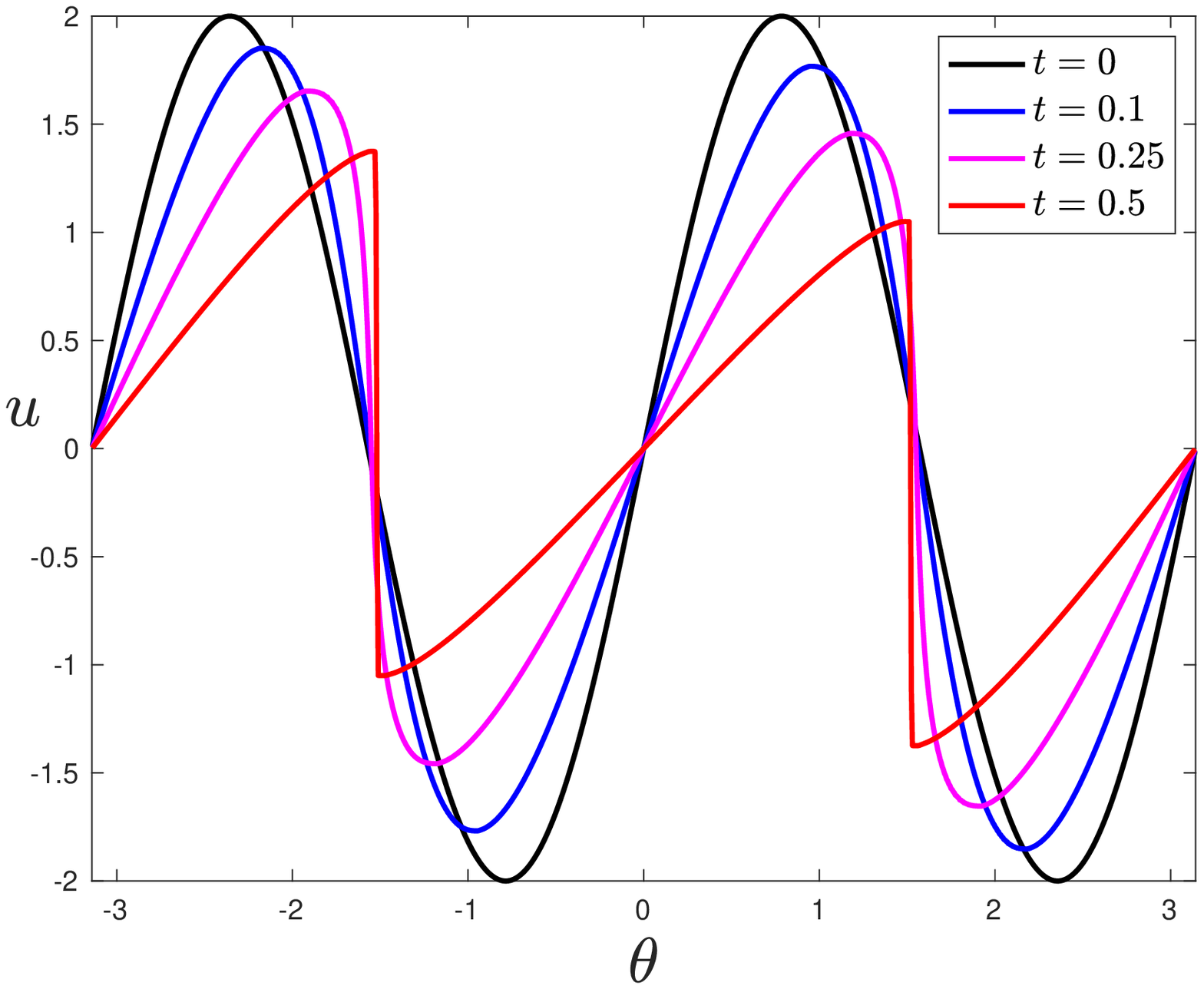}} 
\caption{Identical natural frequencies}
\label{fig_id} 
\end{figure}

Figure \ref{fig_id} exhibits the time evolutions of $\rho$ and $u$ for the identical cases. The parameters and the initial data for Figures \ref{fig_id} (a) and (b) are selected as
\[
m=0.5, \quad  K=0.1, \quad \mbox{and} \quad u_0(\theta) = -\sin \theta
\]
so that they lie in subcritical region. We can observe the contractivity of support of $\rho$, and it is consistent with our theoretical result, Remark \ref{rmk_dirac}. Nevertheless, $\rho$ does not blow-up in finite time, see Remark \ref{rem_density}. We also note that the velocity diameter $d_v$ decreases. On the other hand, Figures \ref{fig_id} (c) and (d) show the profiles of $\rho$ and $u$ for supercritical region around $\theta=0$ with the following initial velocity and $m, K$:
\[
m=1, \quad K=1, \quad \mbox{and} \quad u_0(\theta) = -2\sin \theta.
\]
As proved in Proposition \ref{prop_cri} and Remark \ref{rem_density}, we can observe the finite-time blow-up of solutions. Furthermore, if we set 
\[
m=1, \quad K=1, \quad \mbox{and} \quad u_0(\theta) = 2\sin 2\theta
\]
so that we have a supercritical case around $\theta= \pm \pi/2$, the profiles show the finite-time blow-up around these points as can be seen in Figures \ref{fig_id} (e) and (f). It is remarkable that the density forms Dirac measures in finite time in the supercritical case, and this supports that our model exhibits the finite-time synchronization phenomena. 

We next consider the nonidentical case, that is, the distribution function $g$ for natural frequencies is not the form of Dirac measure centered on some fixed point $\Omega_0 \in \R$. 

\begin{figure}[ht]
\centering
\subfigure[Density - subcritical region]{  
\includegraphics[width=2.8in]{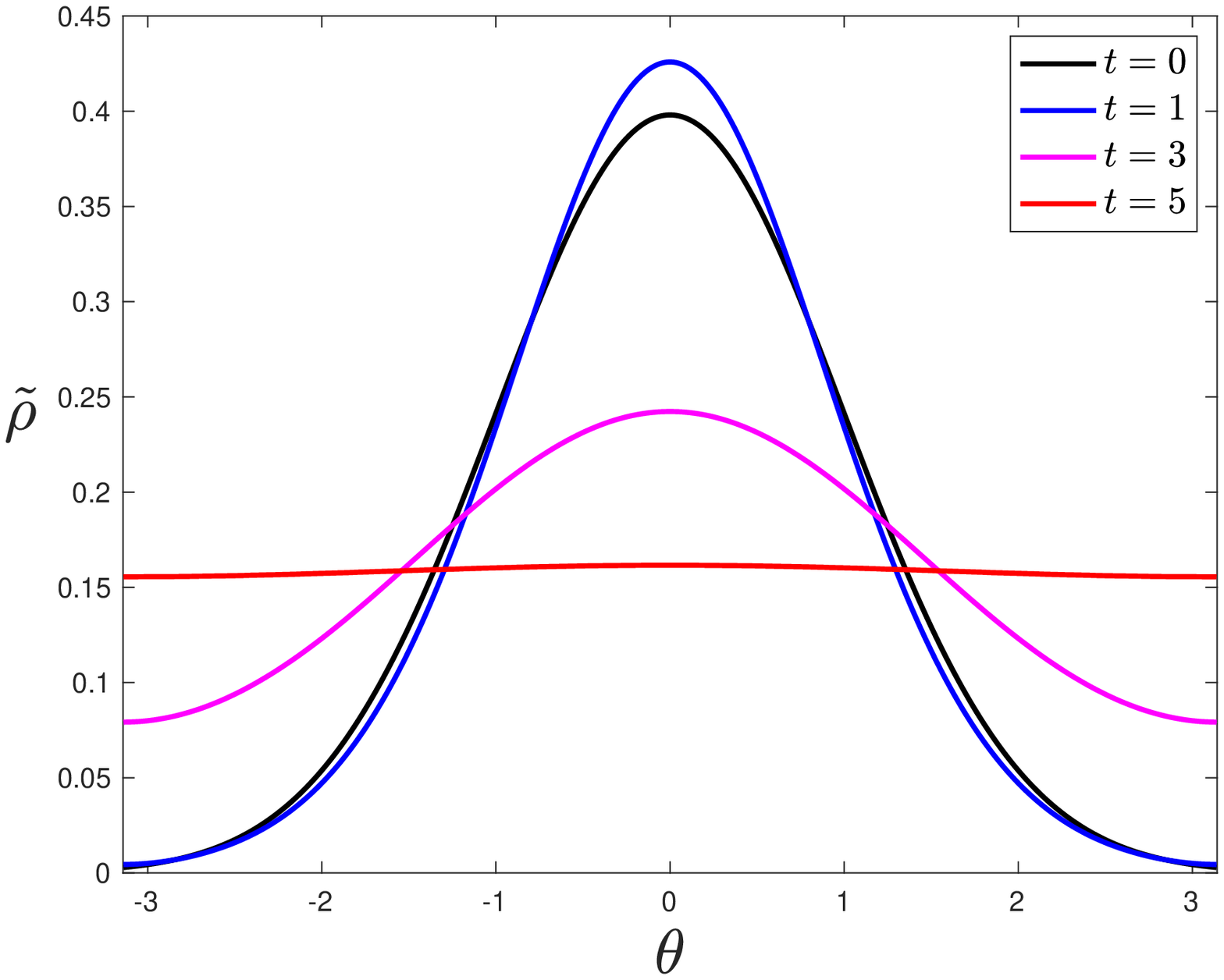}} 
\subfigure[Velocity - subcritical region]{  
\includegraphics[width=2.8in]{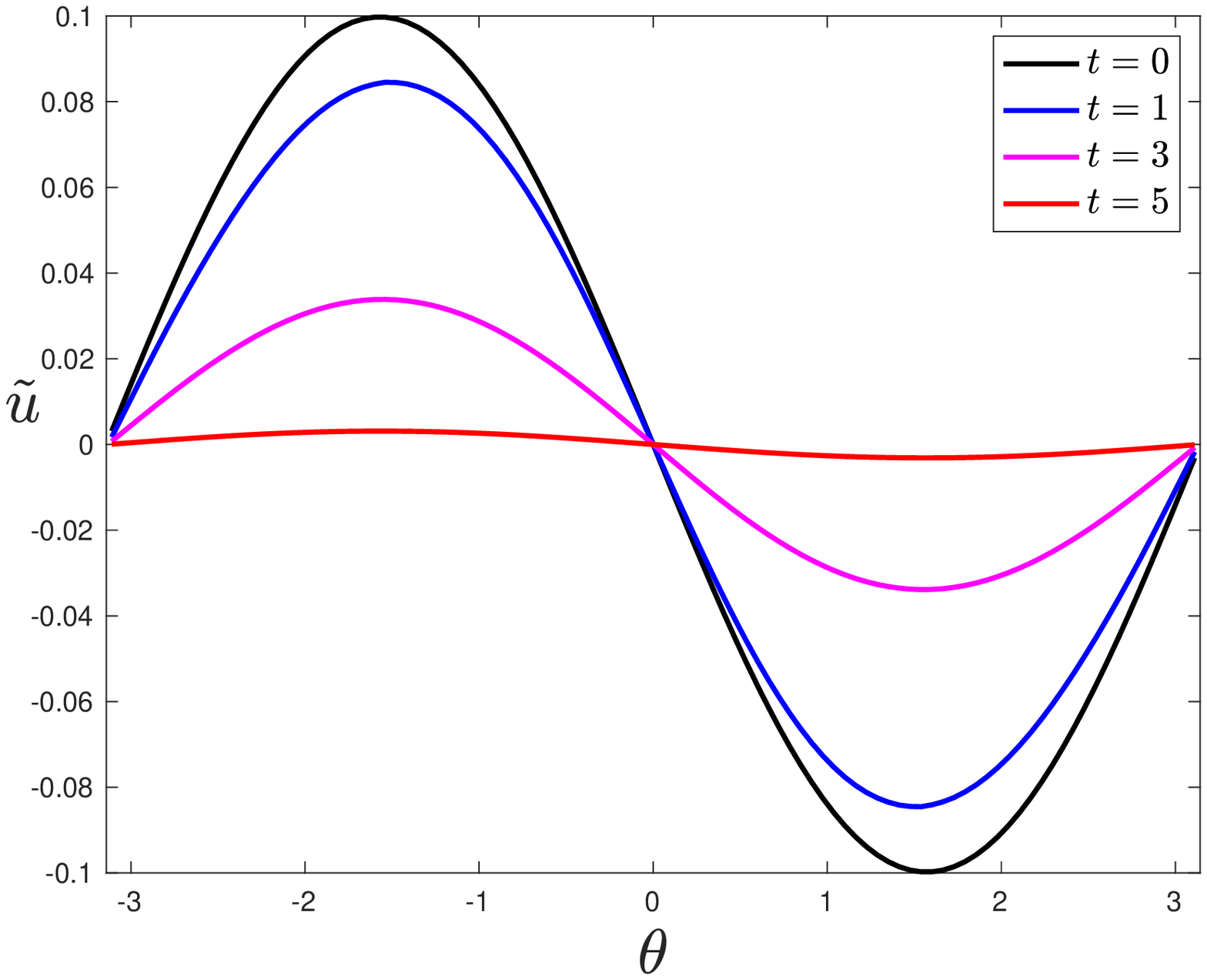}} 
\subfigure[Density - supercritical region] {  
\includegraphics[width=2.8in]{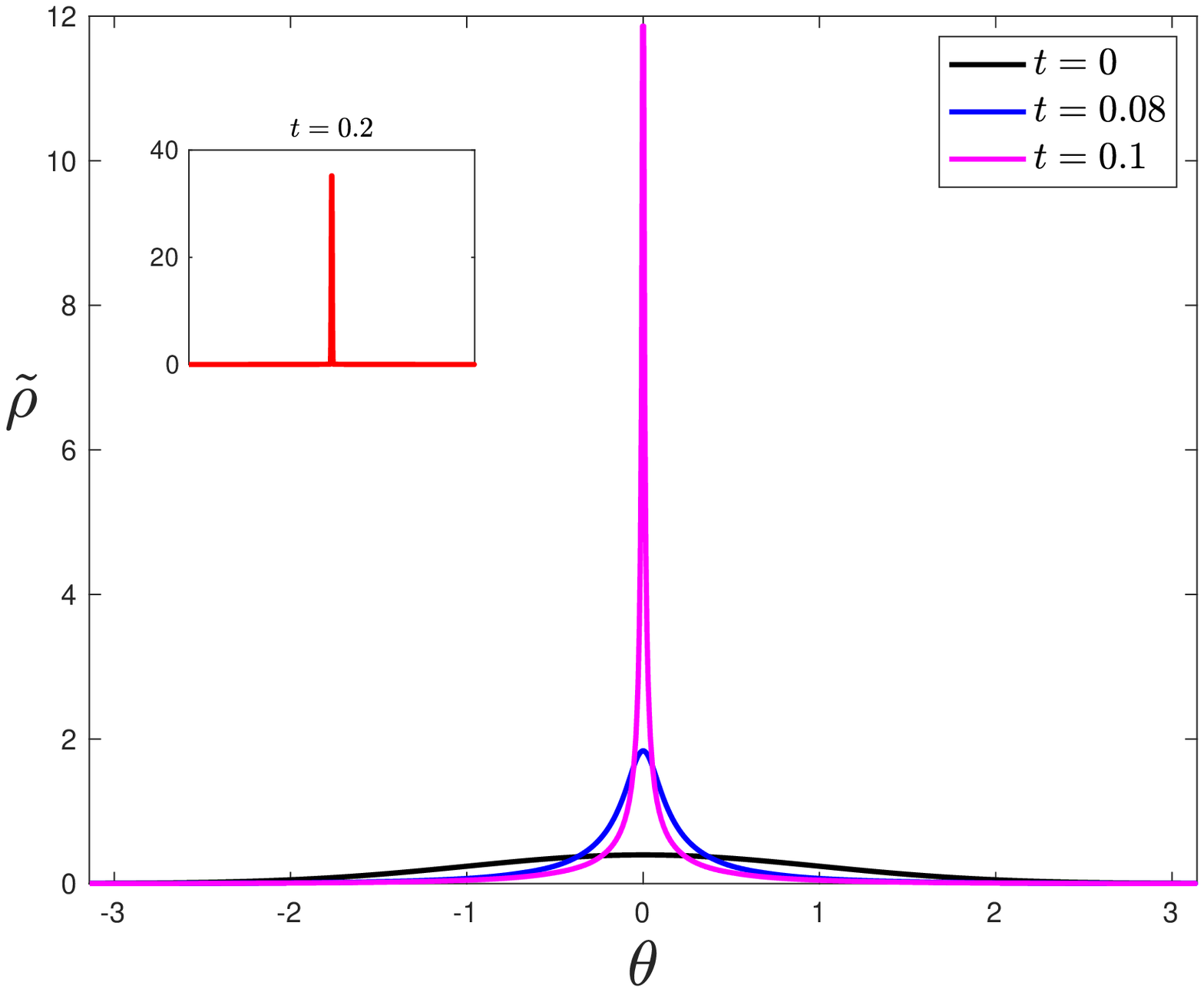}} 
\subfigure[Velocity - supercritical region]{  
\includegraphics[width=2.8in]{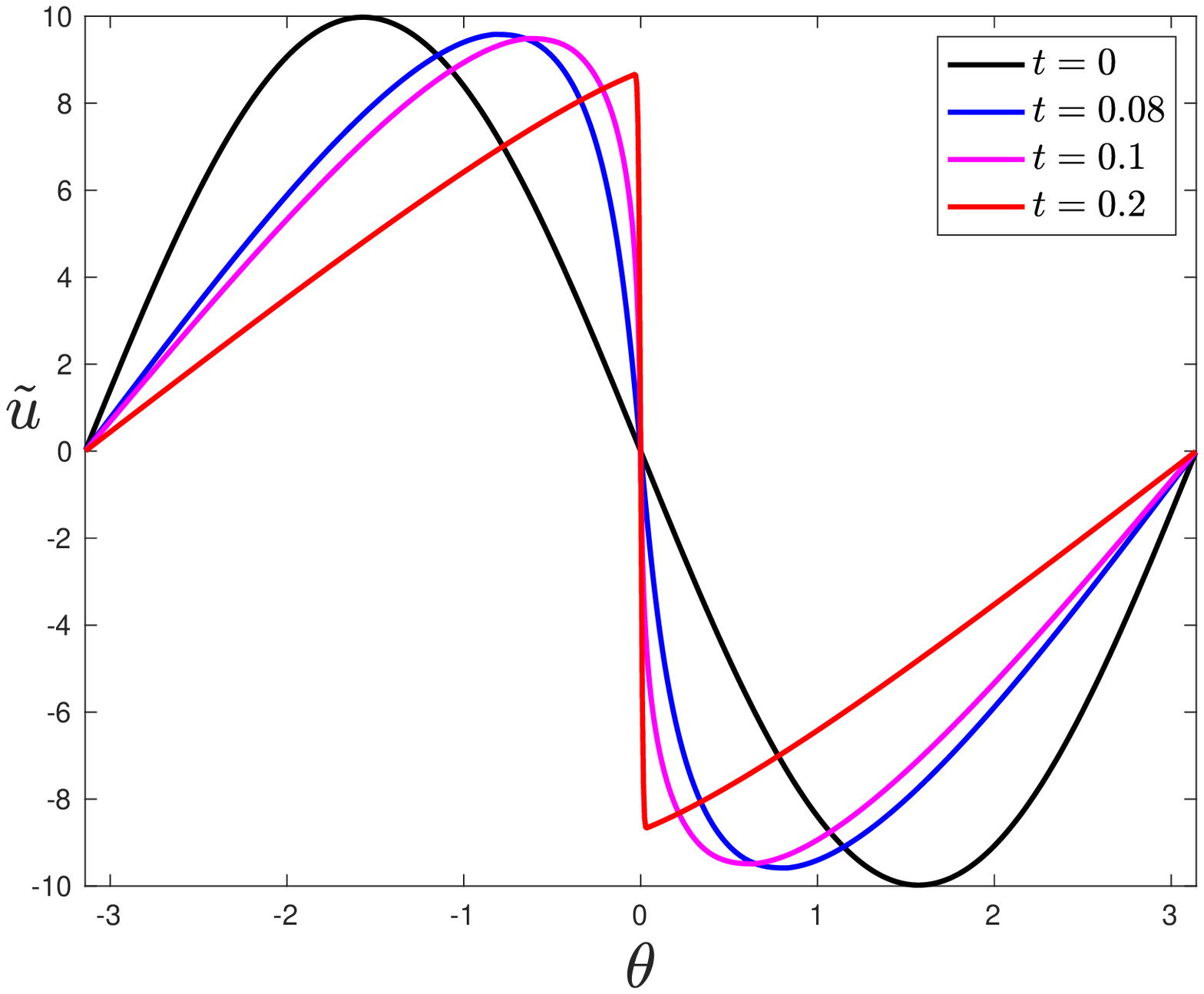}} 
\subfigure[Density - supercritical region] { 
\includegraphics[width=2.8in]{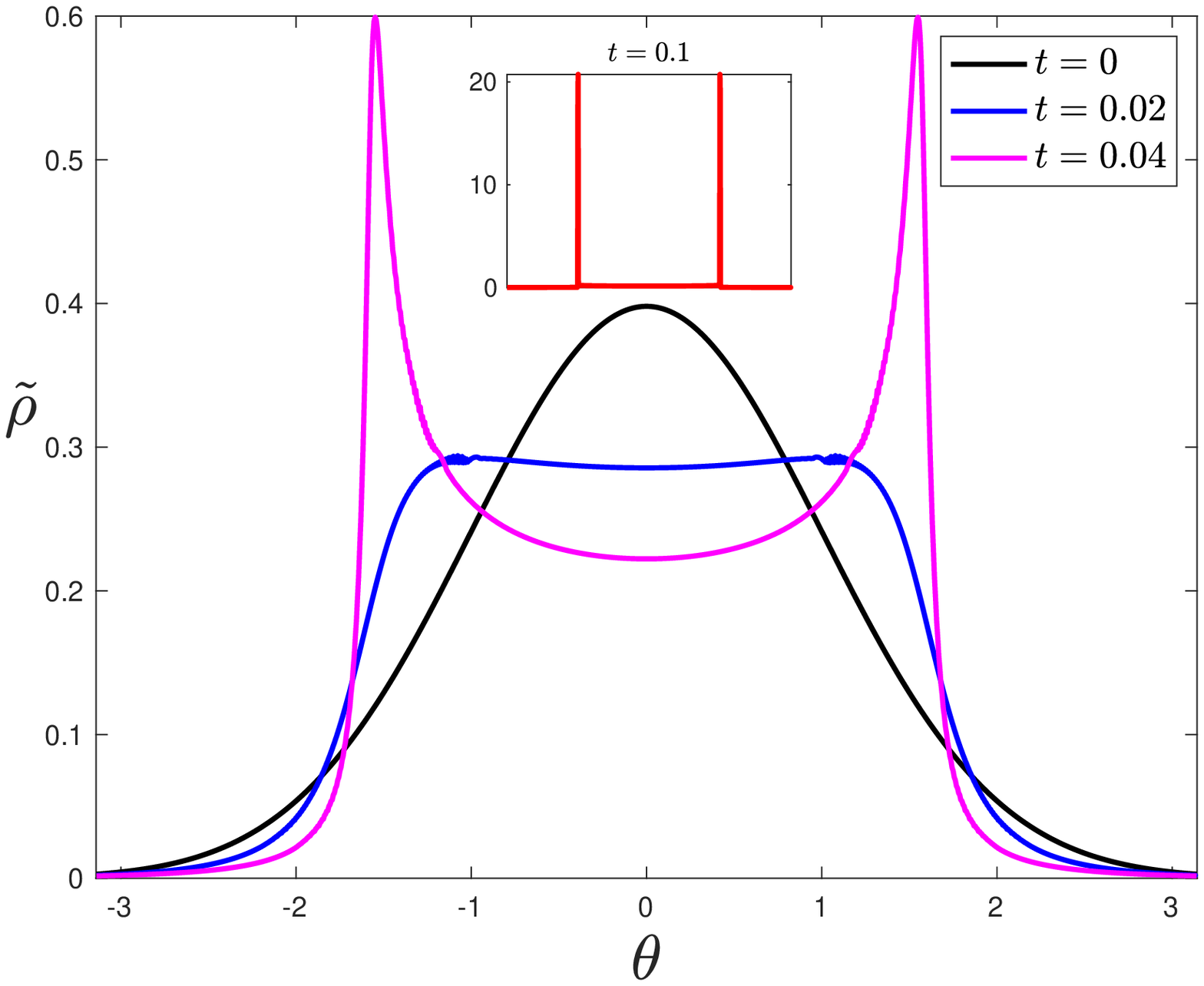}} 
\subfigure[Velocity - supercritical region]{  
\includegraphics[width=2.8in]{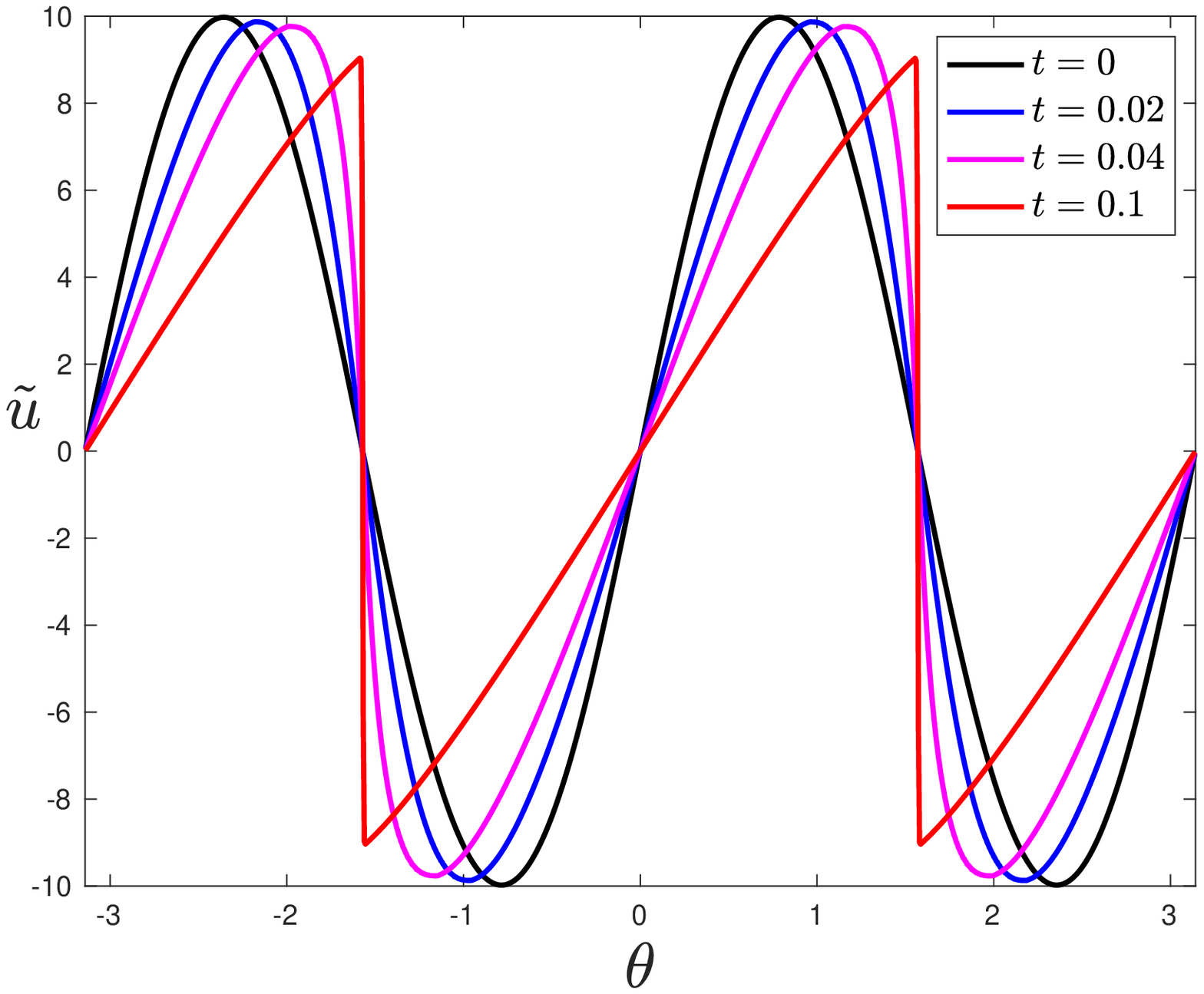}} 
\caption{Nonidentical natural frequencies }
\label{fig_nid}
\end{figure}

\begin{figure}[ht]
\centering
\subfigure[Density at $t=0$]{  
\includegraphics[width=2.8in,height=1.5in]{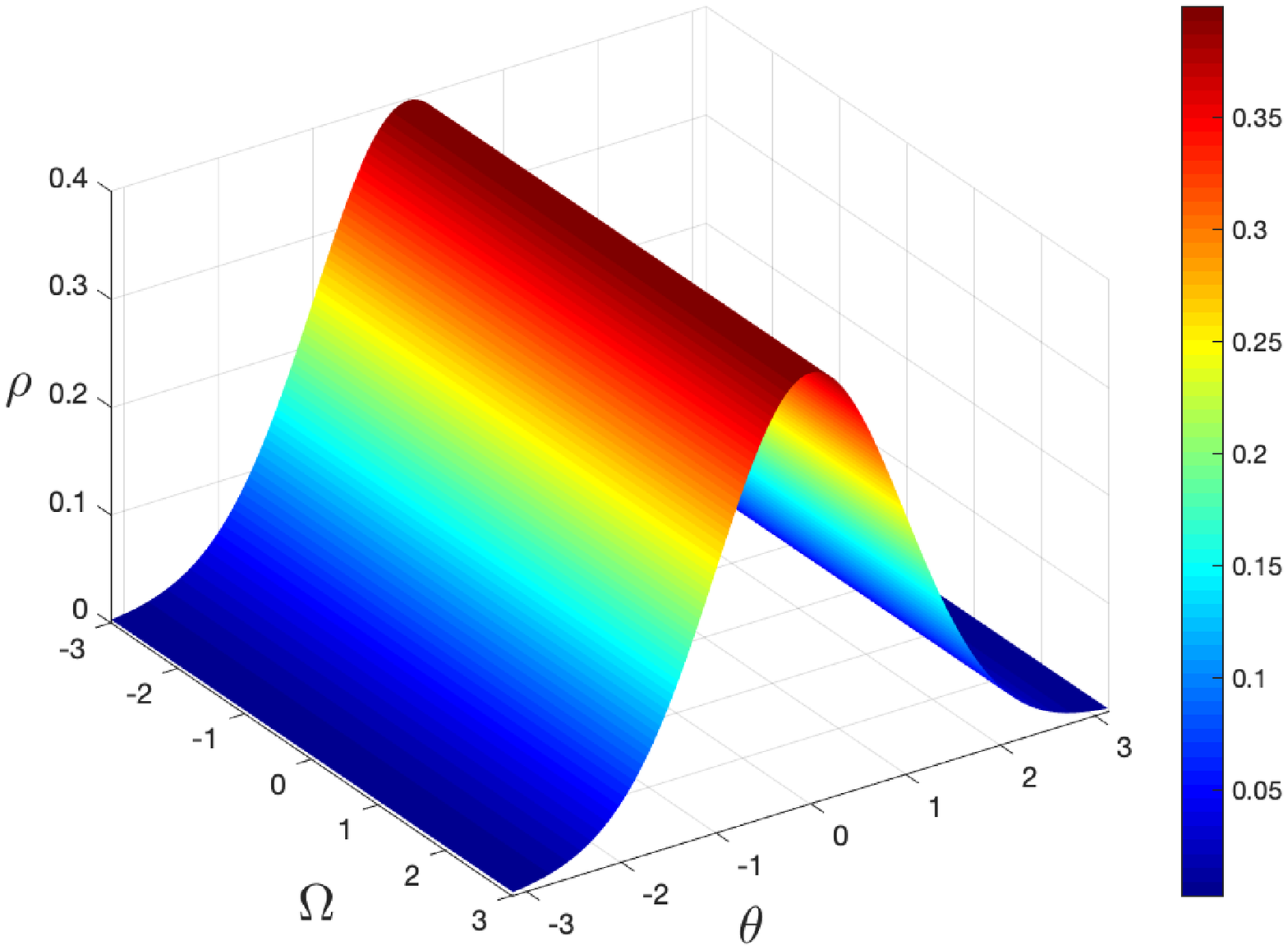}} 
\subfigure[Velocity at $t=0$]{  
\includegraphics[width=2.8in,height=1.5in]{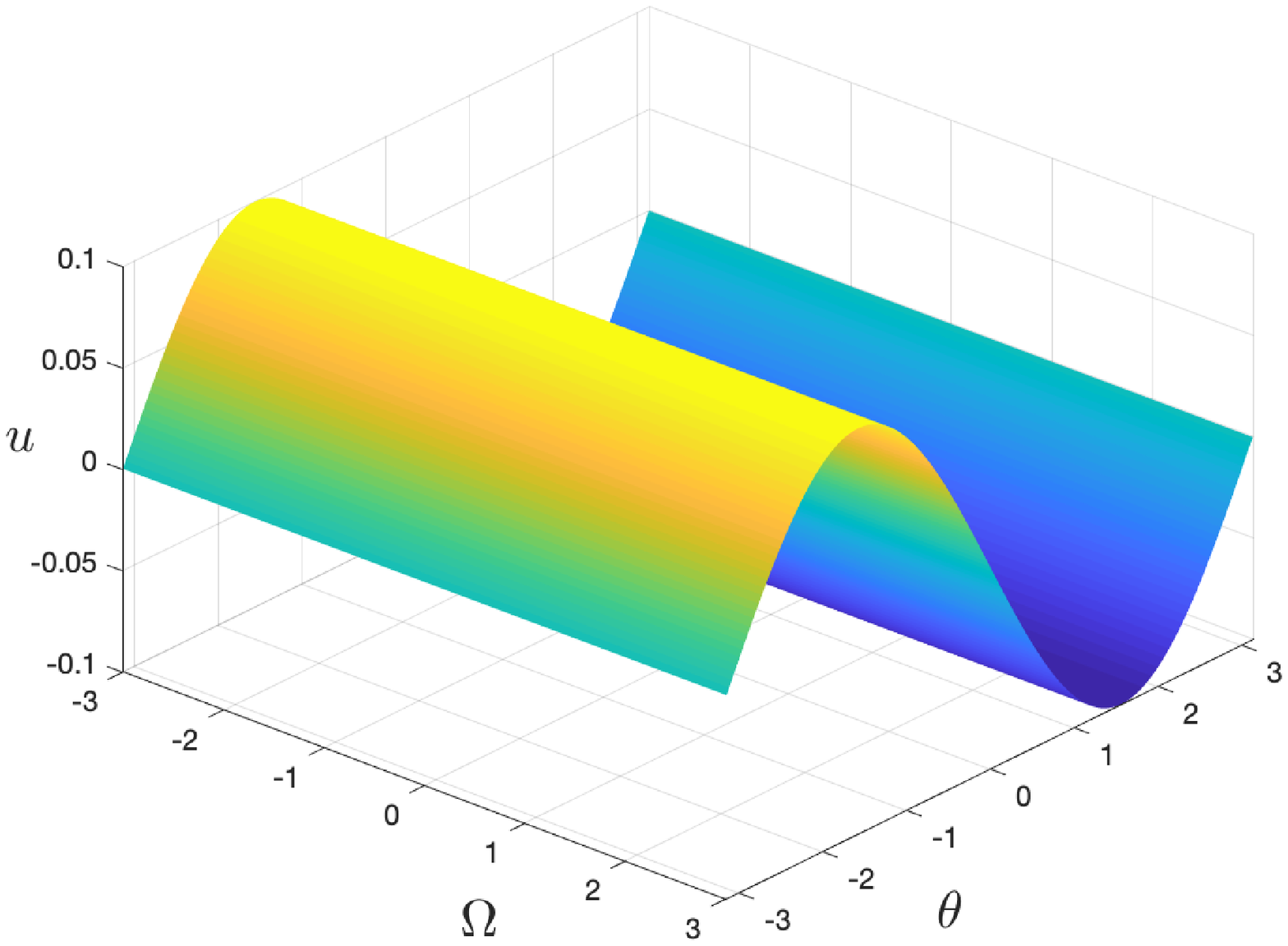}} 
\subfigure[Density at $t=1$] { 
\includegraphics[width=2.8in,height=1.5in]{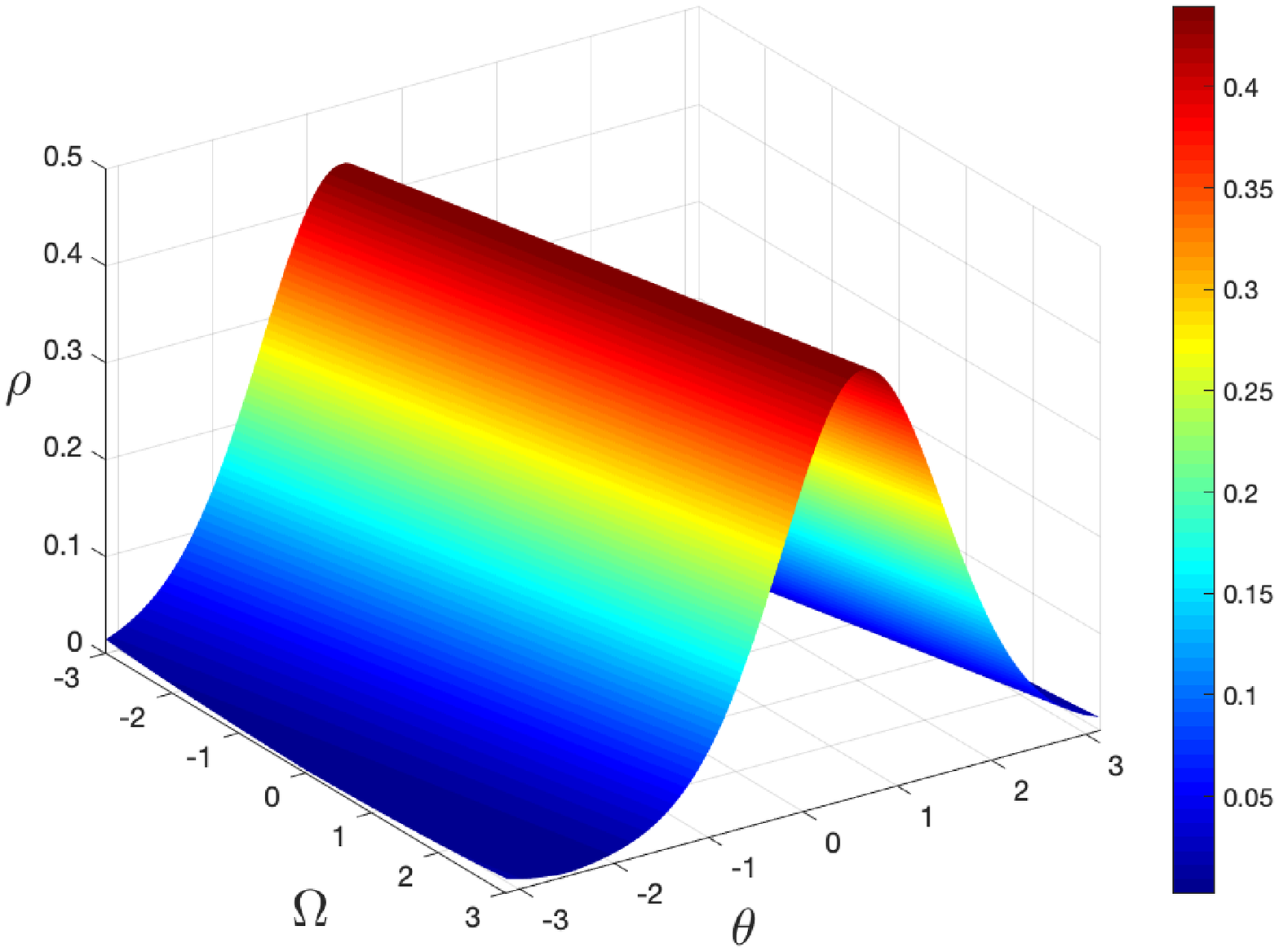}} 
\subfigure[Velocity at $t=1$]{  
\includegraphics[width=2.8in,height=1.5in]{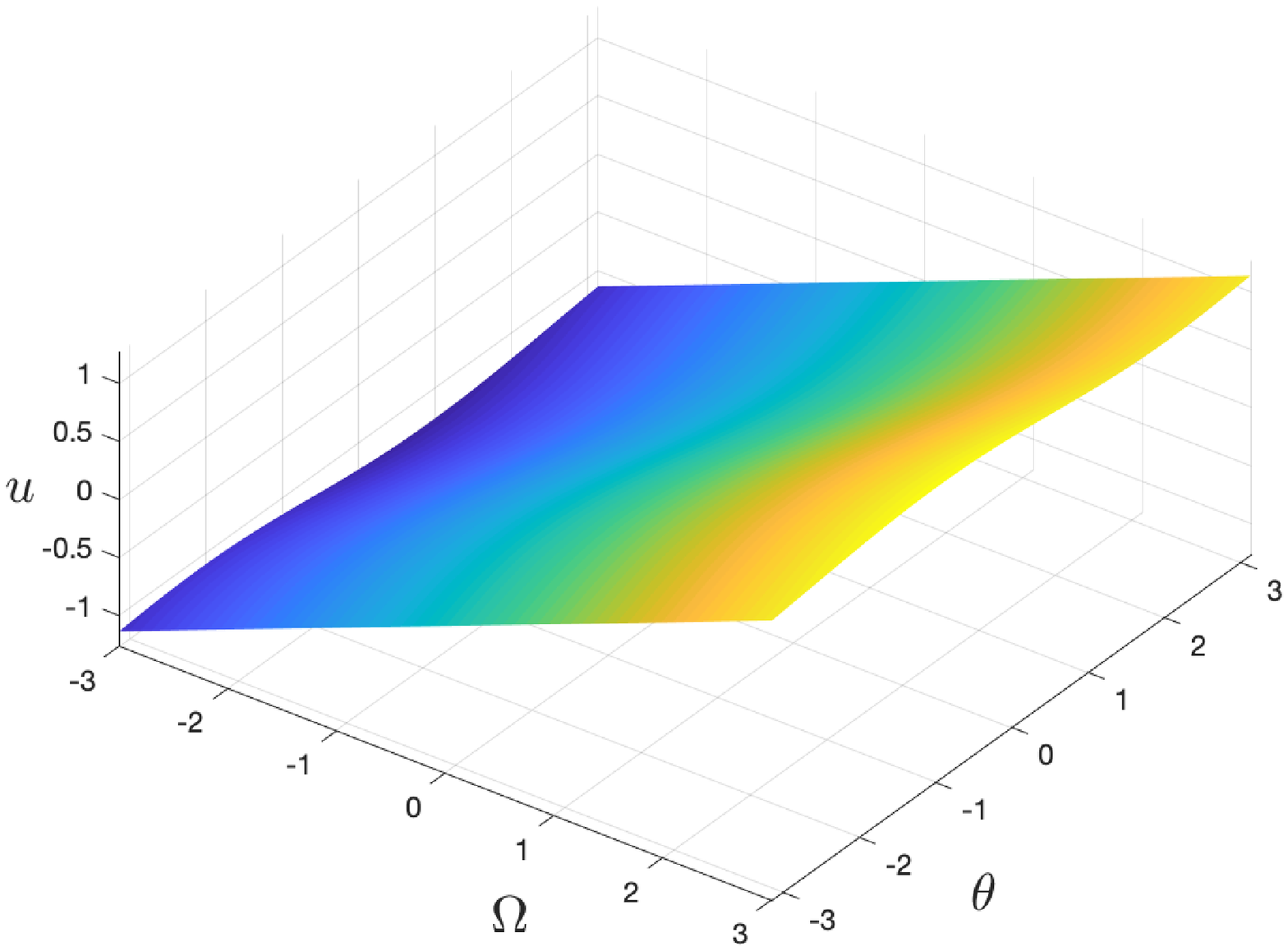}} 
\subfigure[Density at $t=3$] { 
\includegraphics[width=2.8in,height=1.5in]{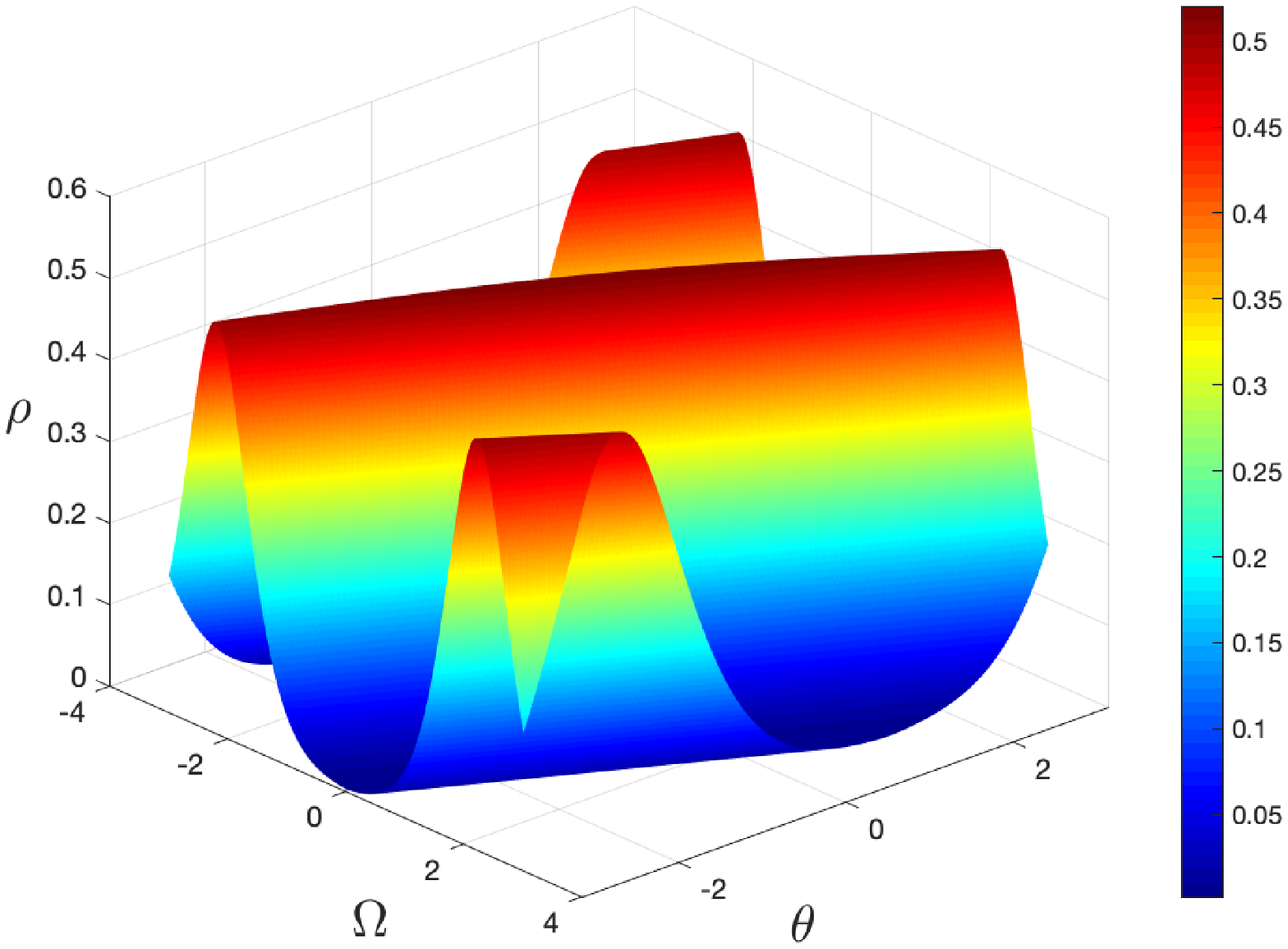}} 
\subfigure[Velocity at $t=3$]{  
\includegraphics[width=2.8in,height=1.5in]{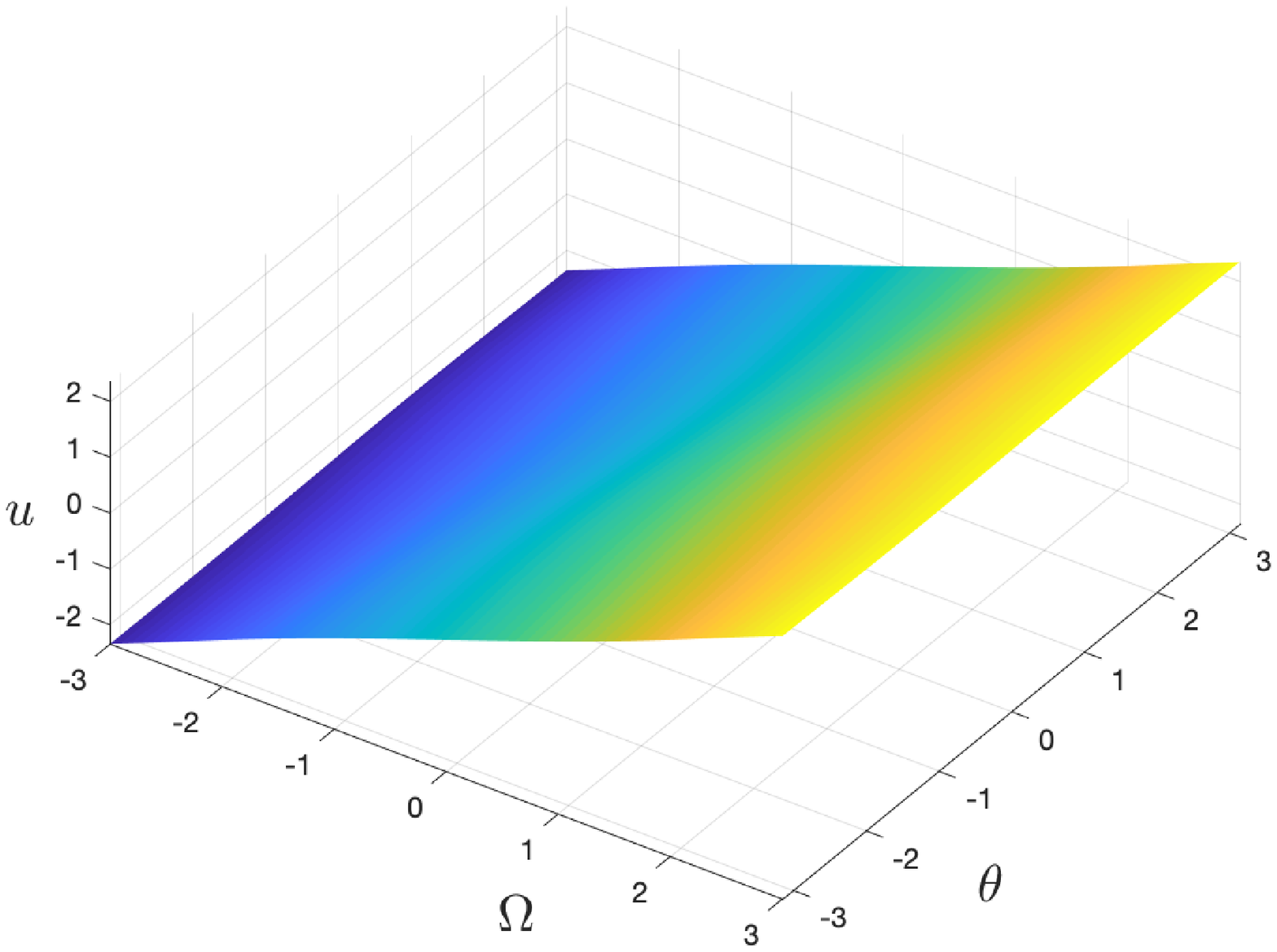}} 
\subfigure[Density at $t=5$] { 
\includegraphics[width=2.8in,height=1.5in]{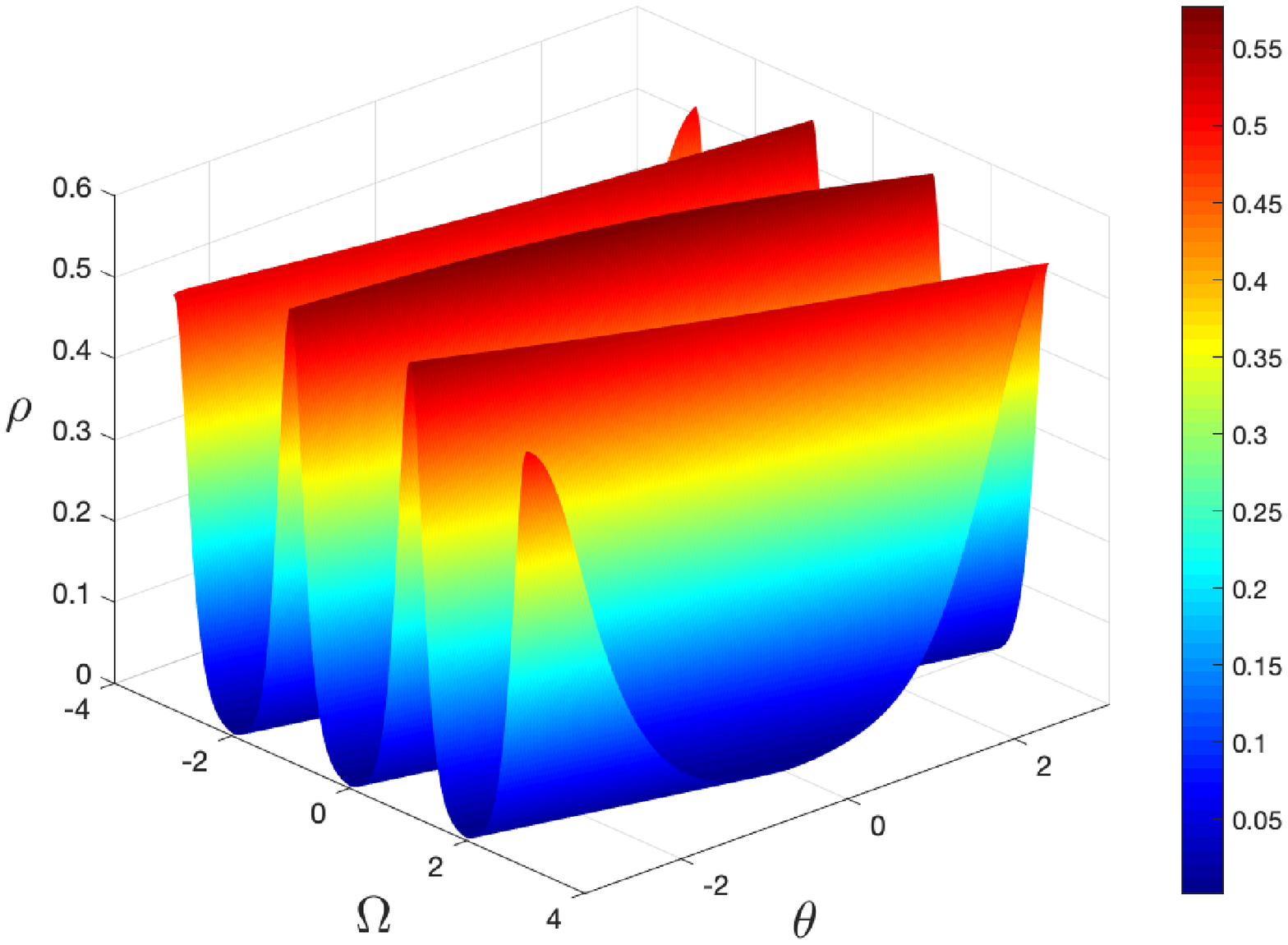}} 
\subfigure[Velocity at $t=5$]{  
\includegraphics[width=2.8in,height=1.5in]{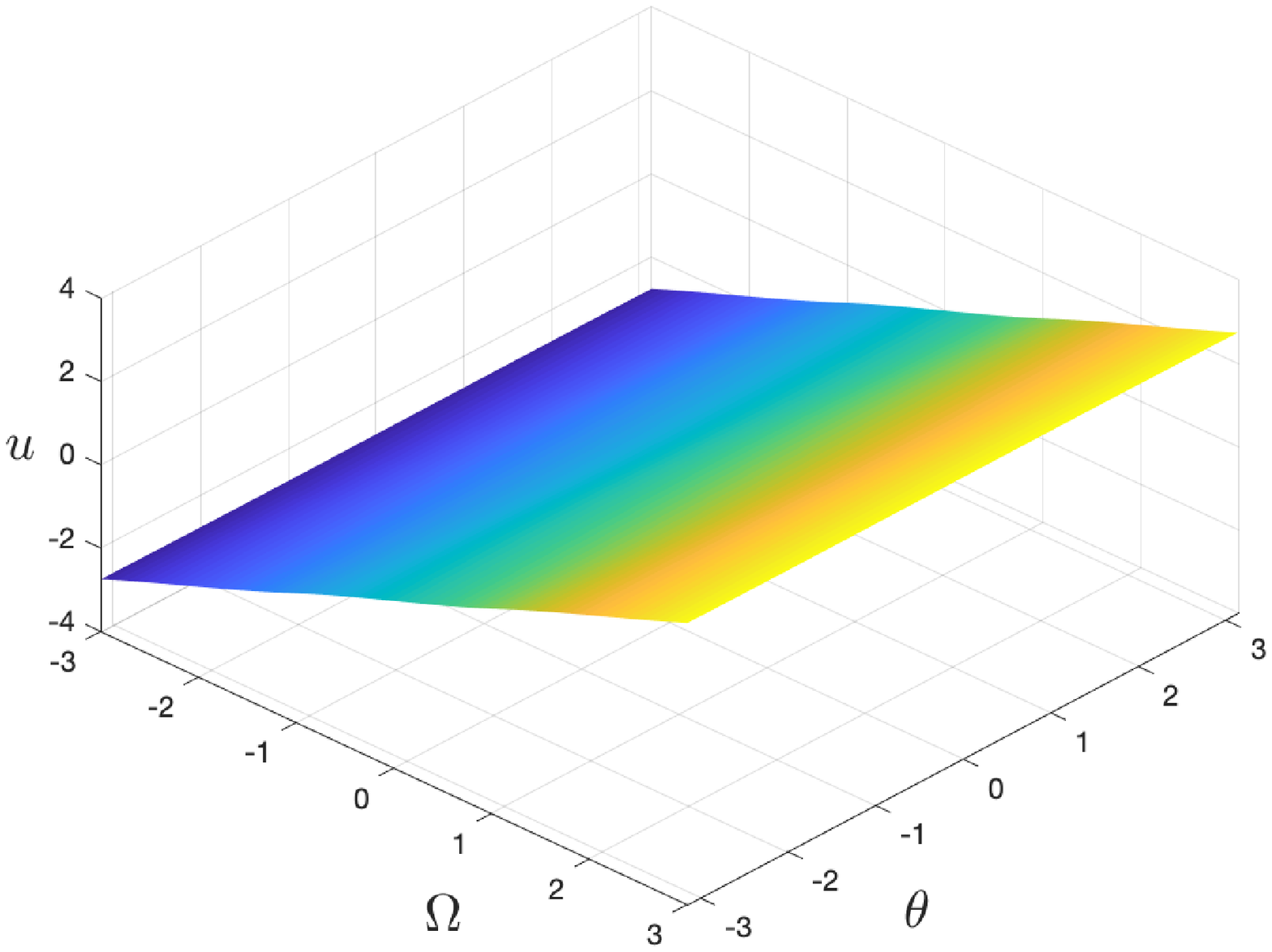}} 
\caption{Subcritical case}
\label{fig_nidsub}
\end{figure}

\begin{figure}[ht]
\centering
\subfigure[Density at $t=0$]{  
\includegraphics[width=2.8in,height=1.5in]{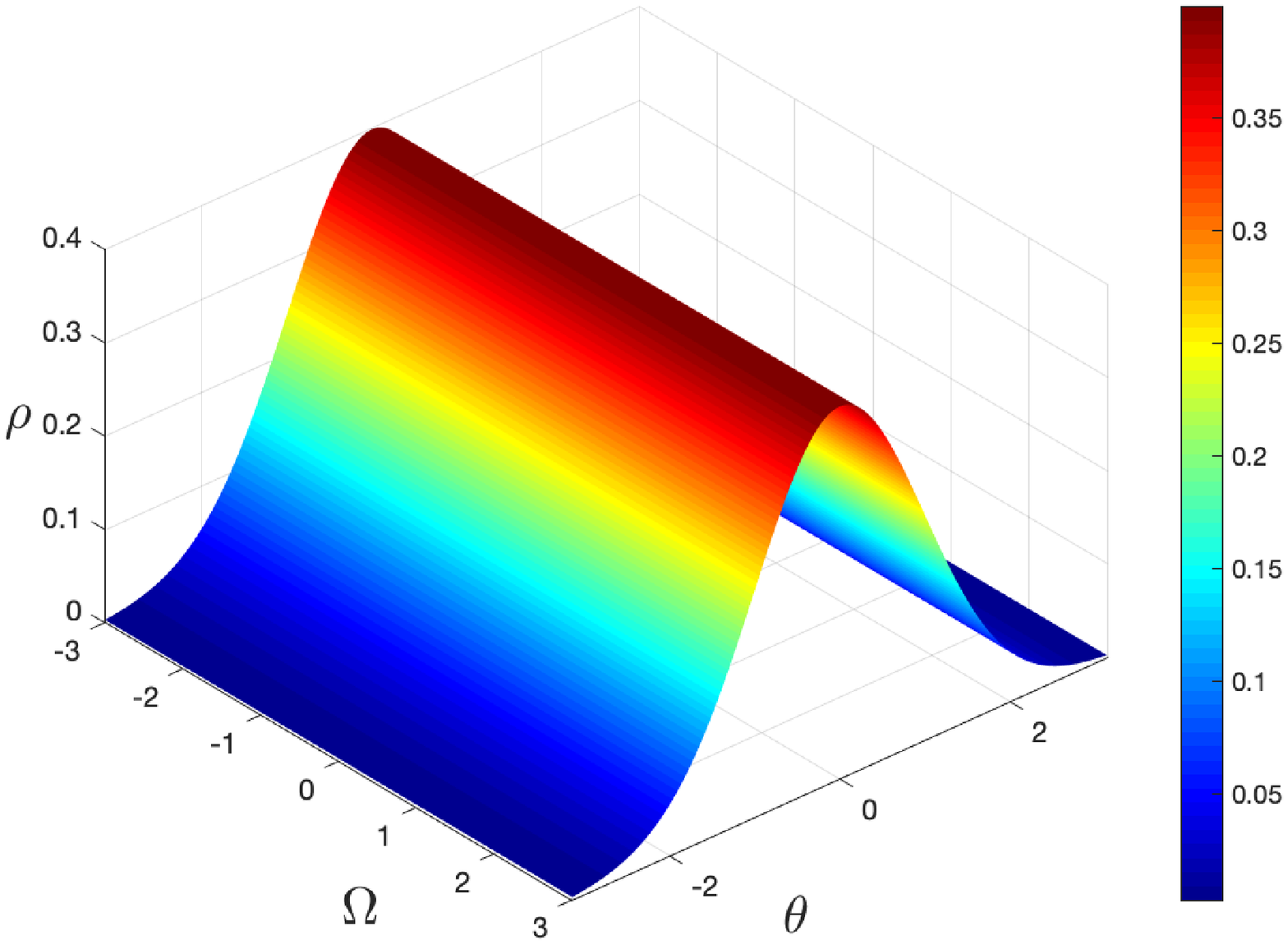}} 
\subfigure[Velocity at $t=0$]{  
\includegraphics[width=2.8in,height=1.5in]{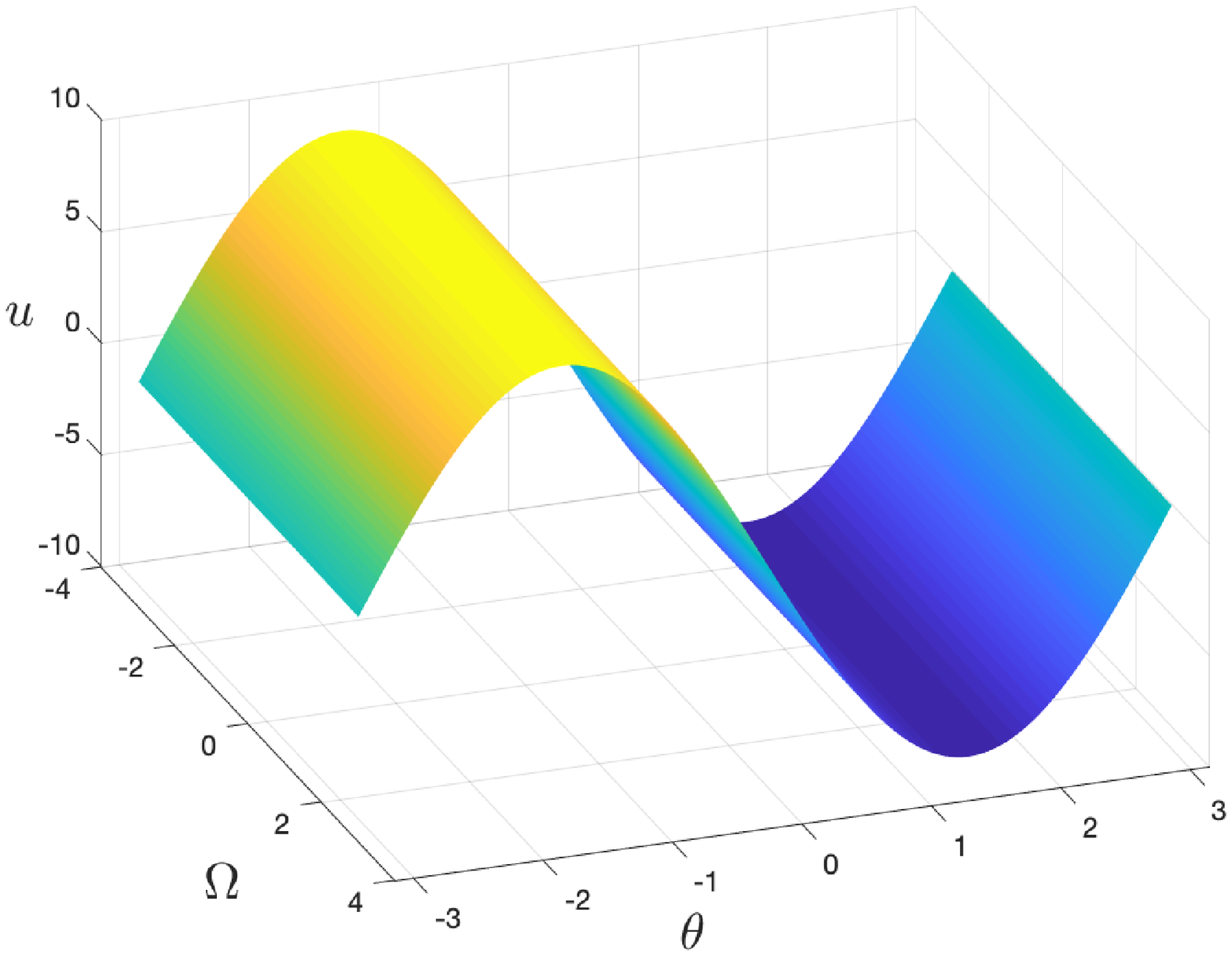}} 
\subfigure[Density at $t=0.08$] { 
\includegraphics[width=2.8in,height=1.5in]{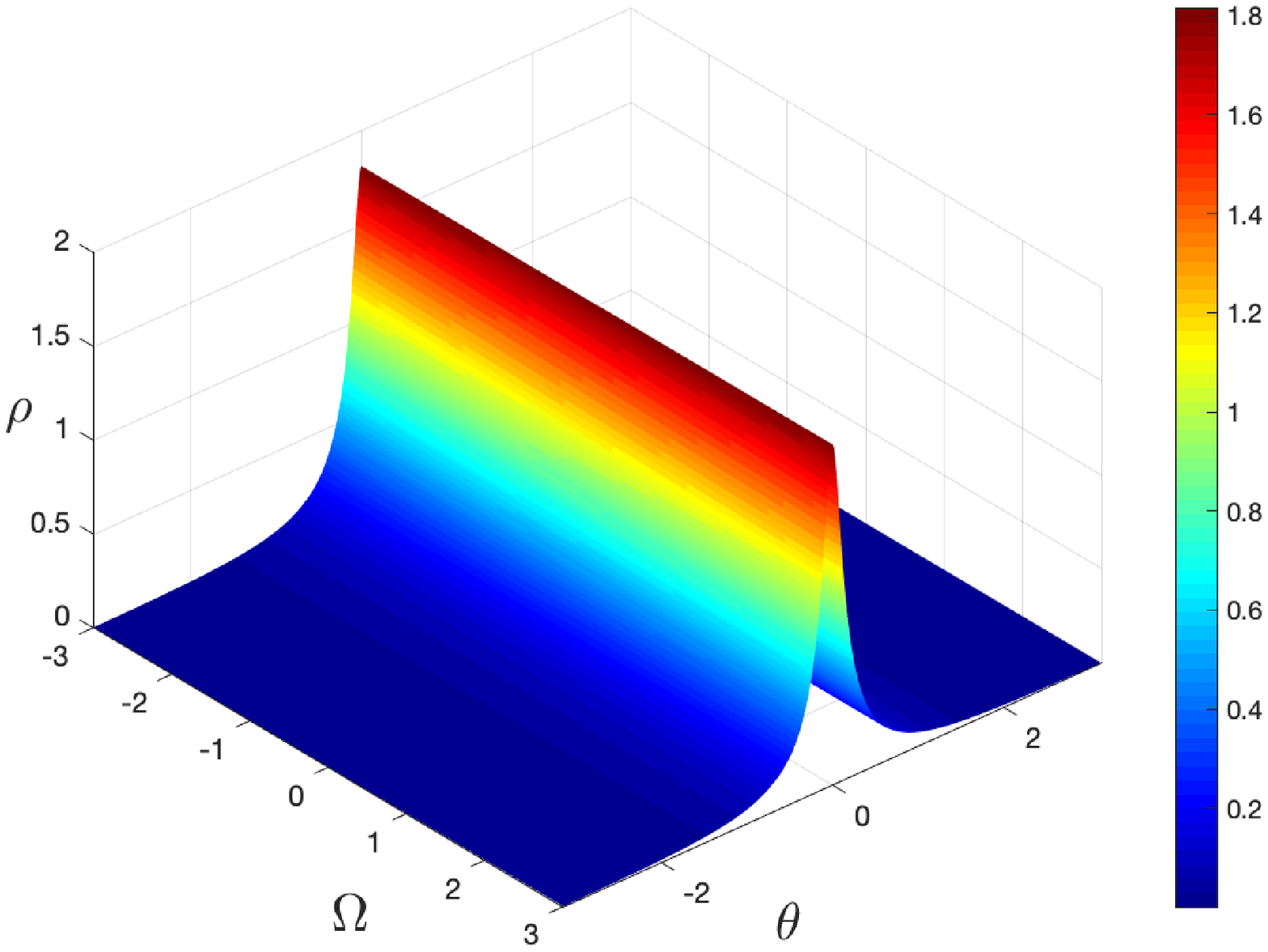}} 
\subfigure[Velocity at $t=0.08$]{  
\includegraphics[width=2.8in,height=1.5in]{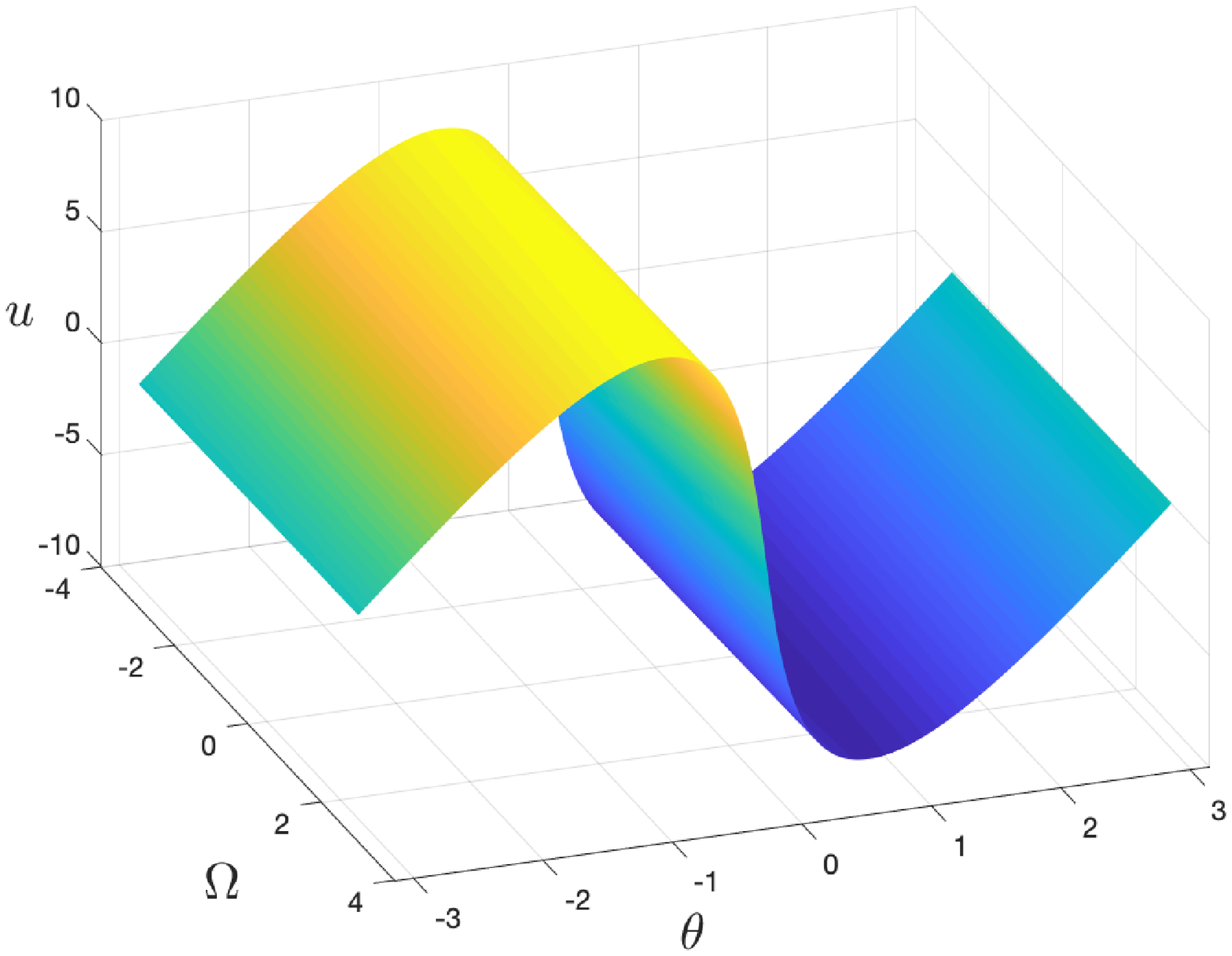}} 
\subfigure[Density at $t=0.1$] {  
\includegraphics[width=2.8in,height=1.5in]{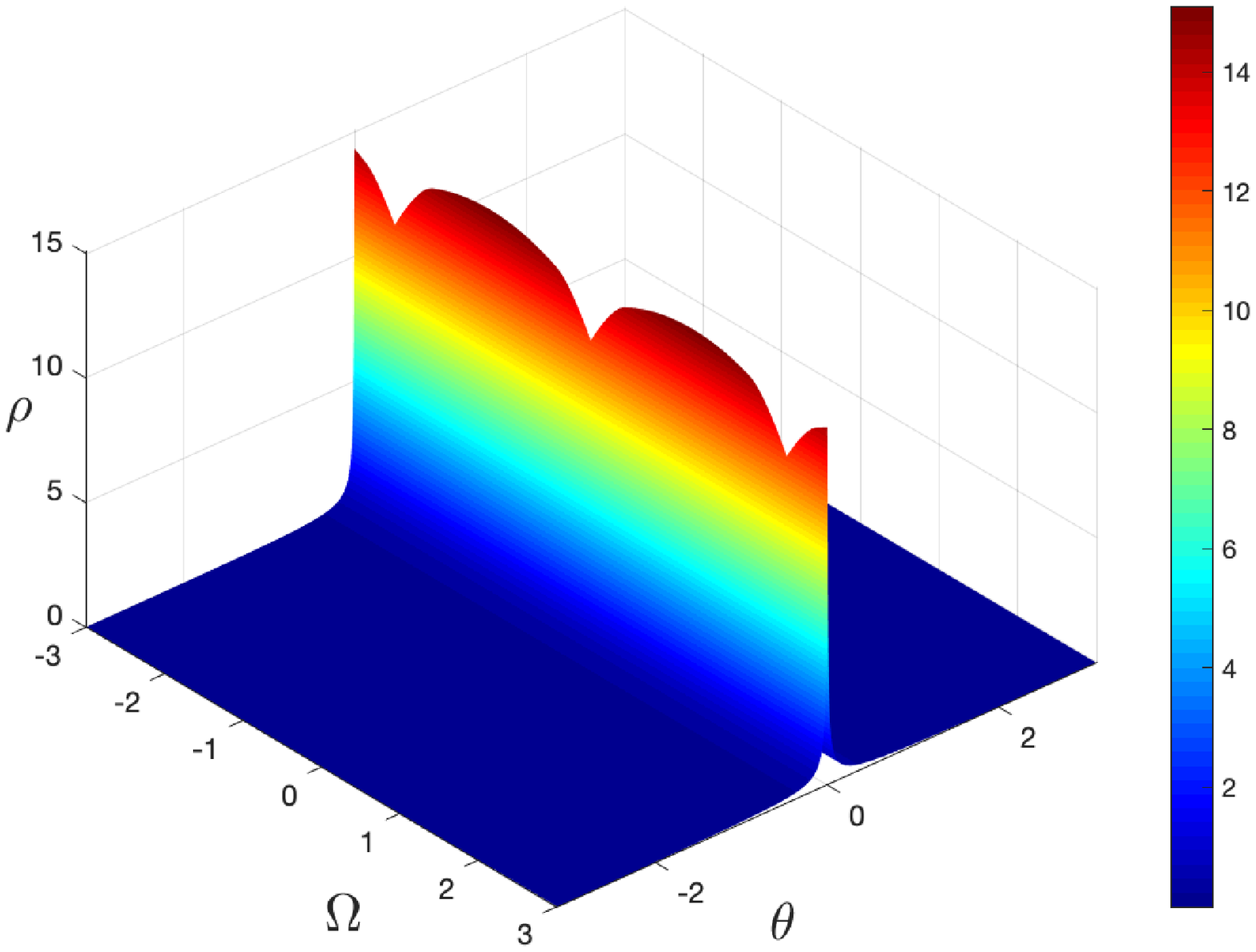}} 
\subfigure[Velocity at $t=0.1$]{  
\includegraphics[width=2.8in,height=1.5in]{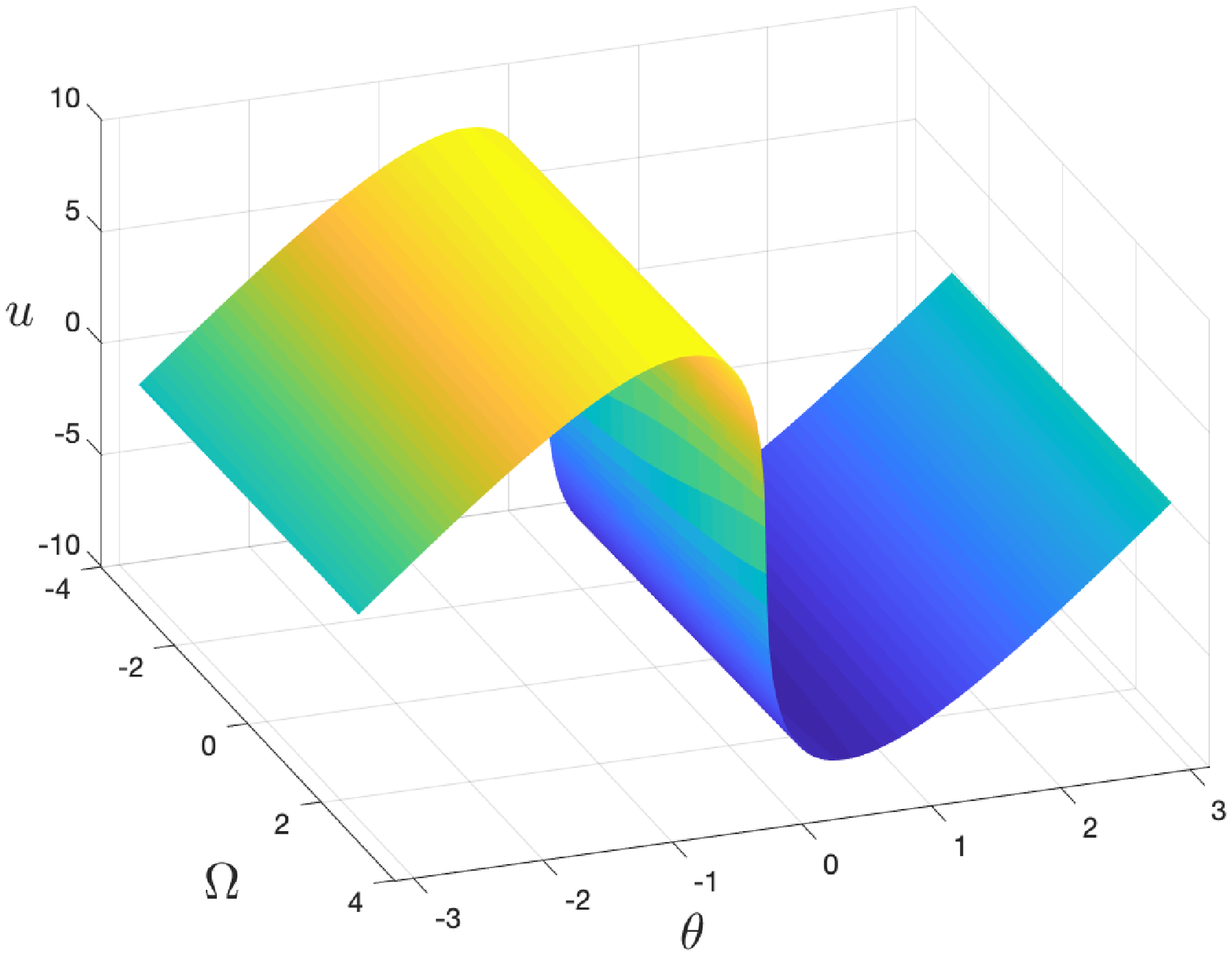}} 
\subfigure[Density at $t=0.2$] { 
\includegraphics[width=2.8in,height=1.5in]{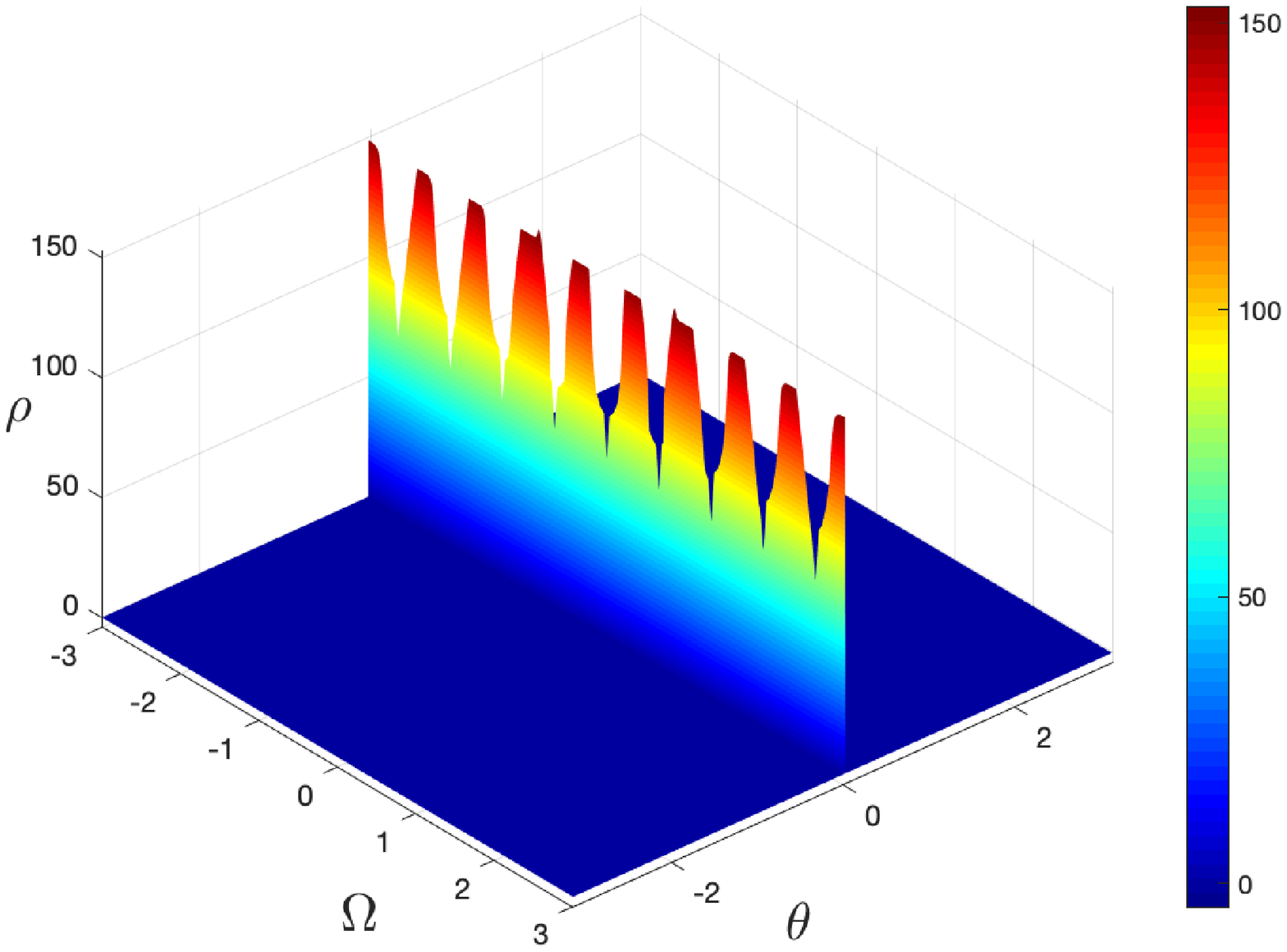}} 
\subfigure[Velocity at $t=0.2$]{  
\includegraphics[width=2.8in,height=1.5in]{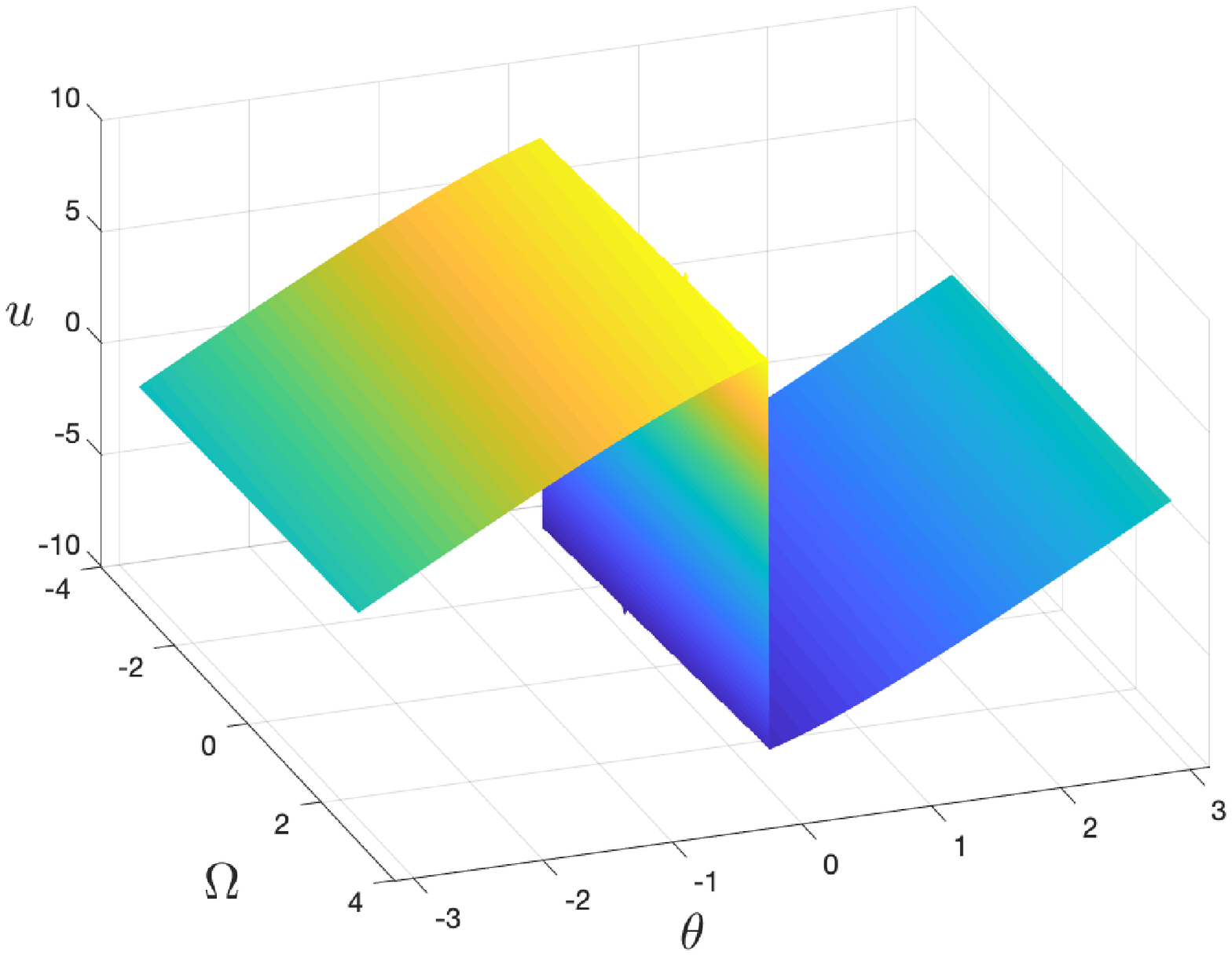}} 
\caption{Supercritical case 1}
\label{fig_nidsup}
\end{figure}

\begin{figure}[ht]
\centering
\subfigure[Density at $t=0$]{ 
\includegraphics[width=2.8in,height=1.5in]{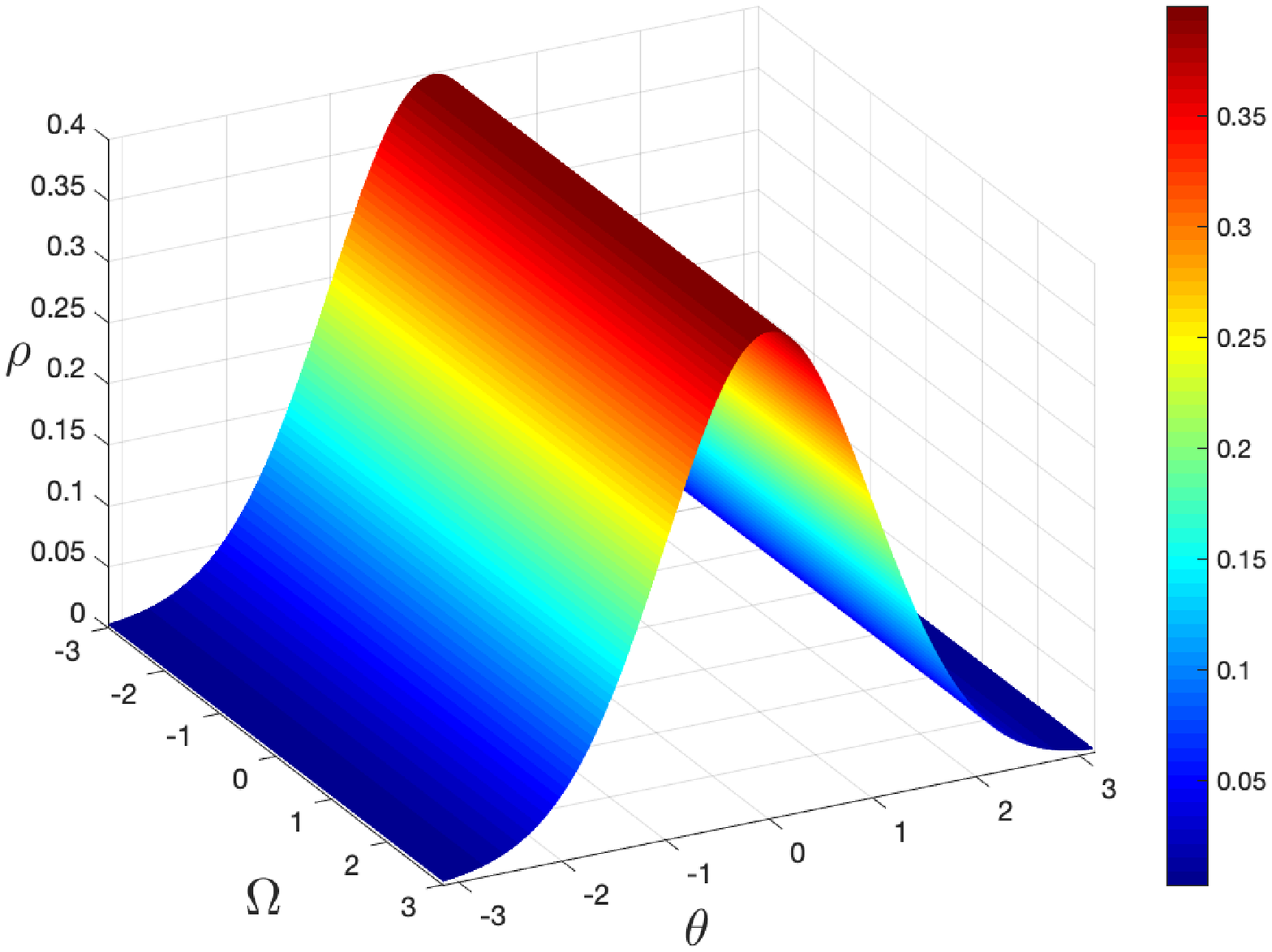}} 
\subfigure[Velocity at $t=0$]{  
\includegraphics[width=2.8in,height=1.5in]{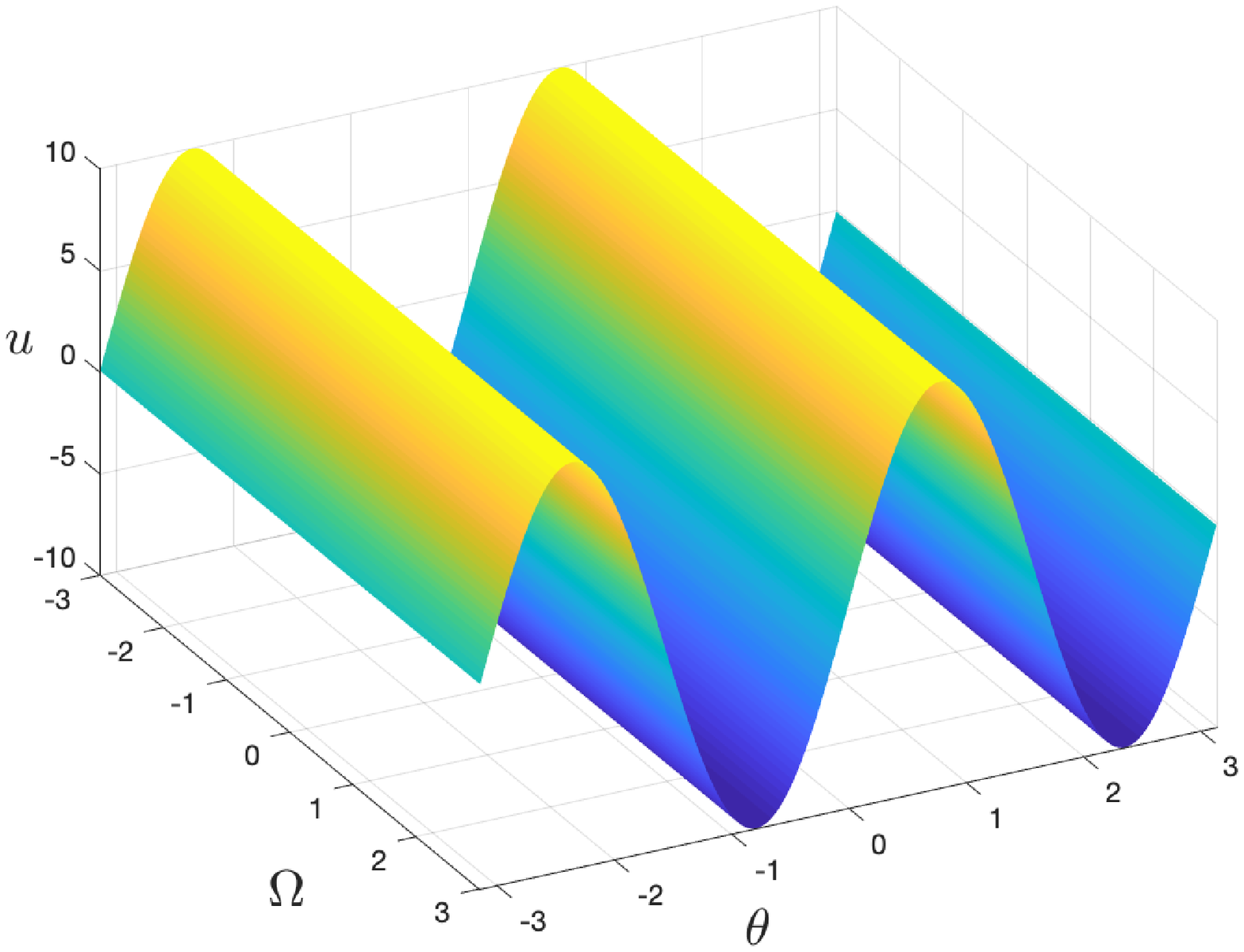}} 
\subfigure[Density at $t=0.02$] {  
\includegraphics[width=2.8in,height=1.5in]{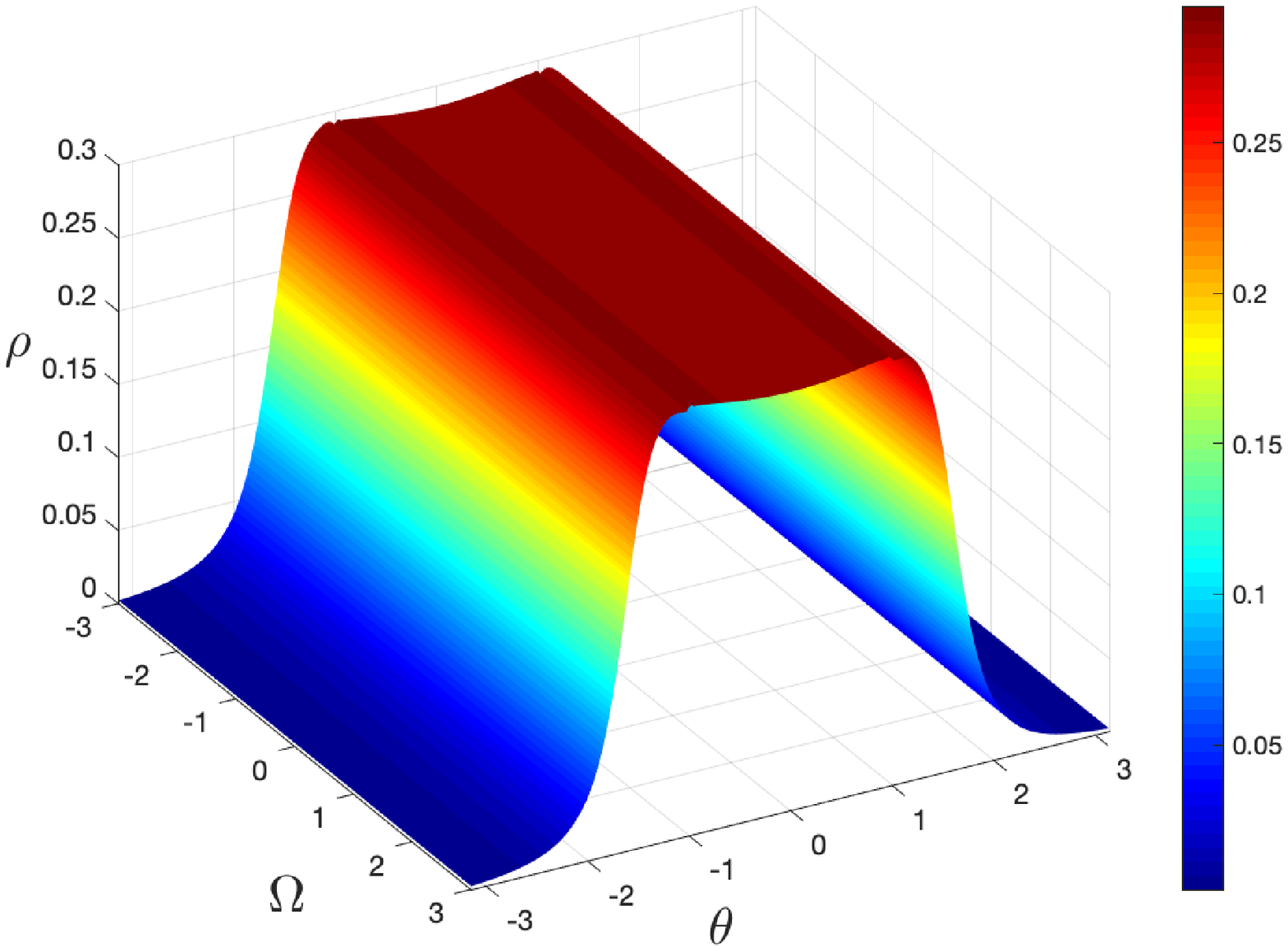}} 
\subfigure[Velocity at $t=0.02$]{  
\includegraphics[width=2.8in,height=1.5in]{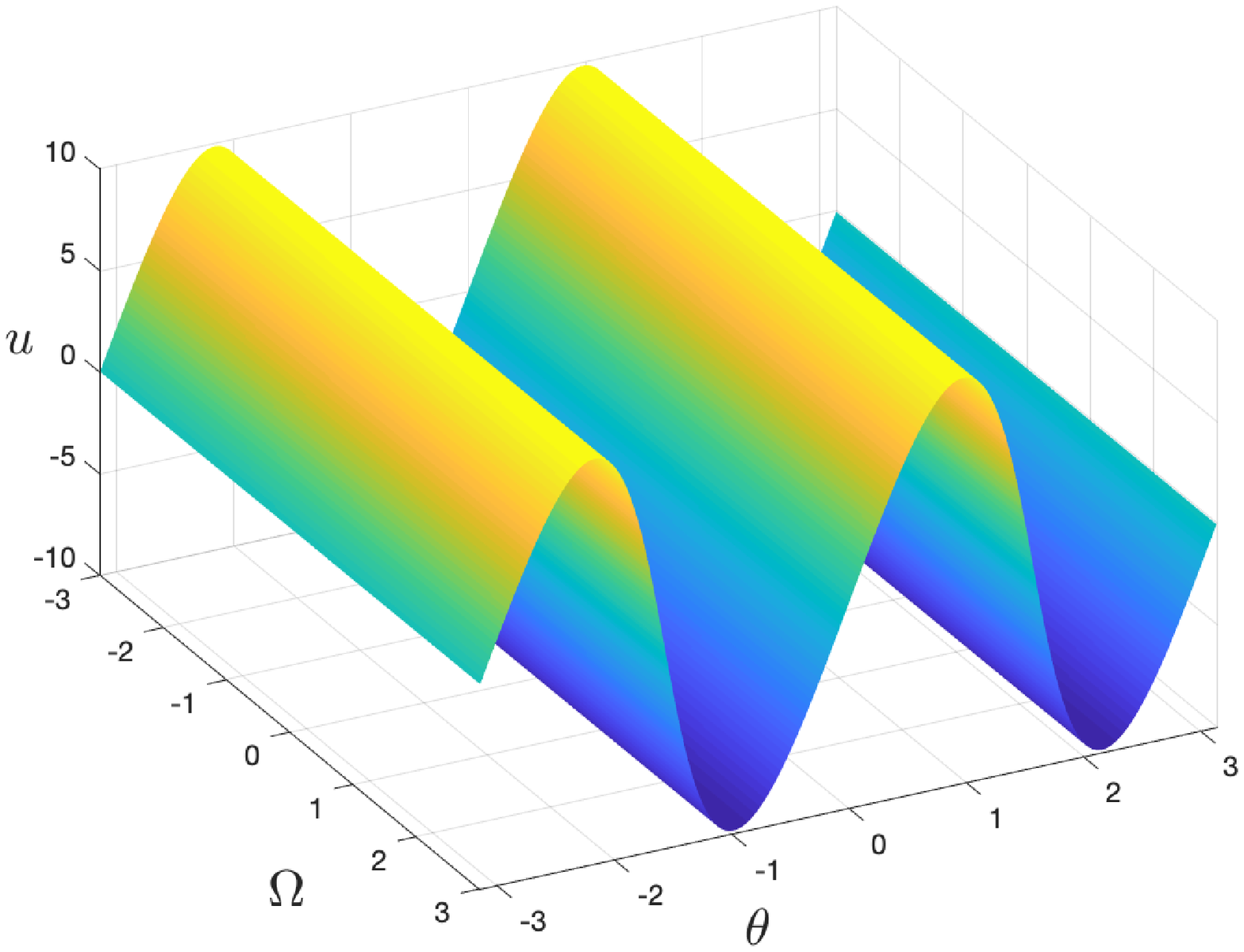}} 
\subfigure[Density at $t=0.04$] { 
\includegraphics[width=2.8in,height=1.5in]{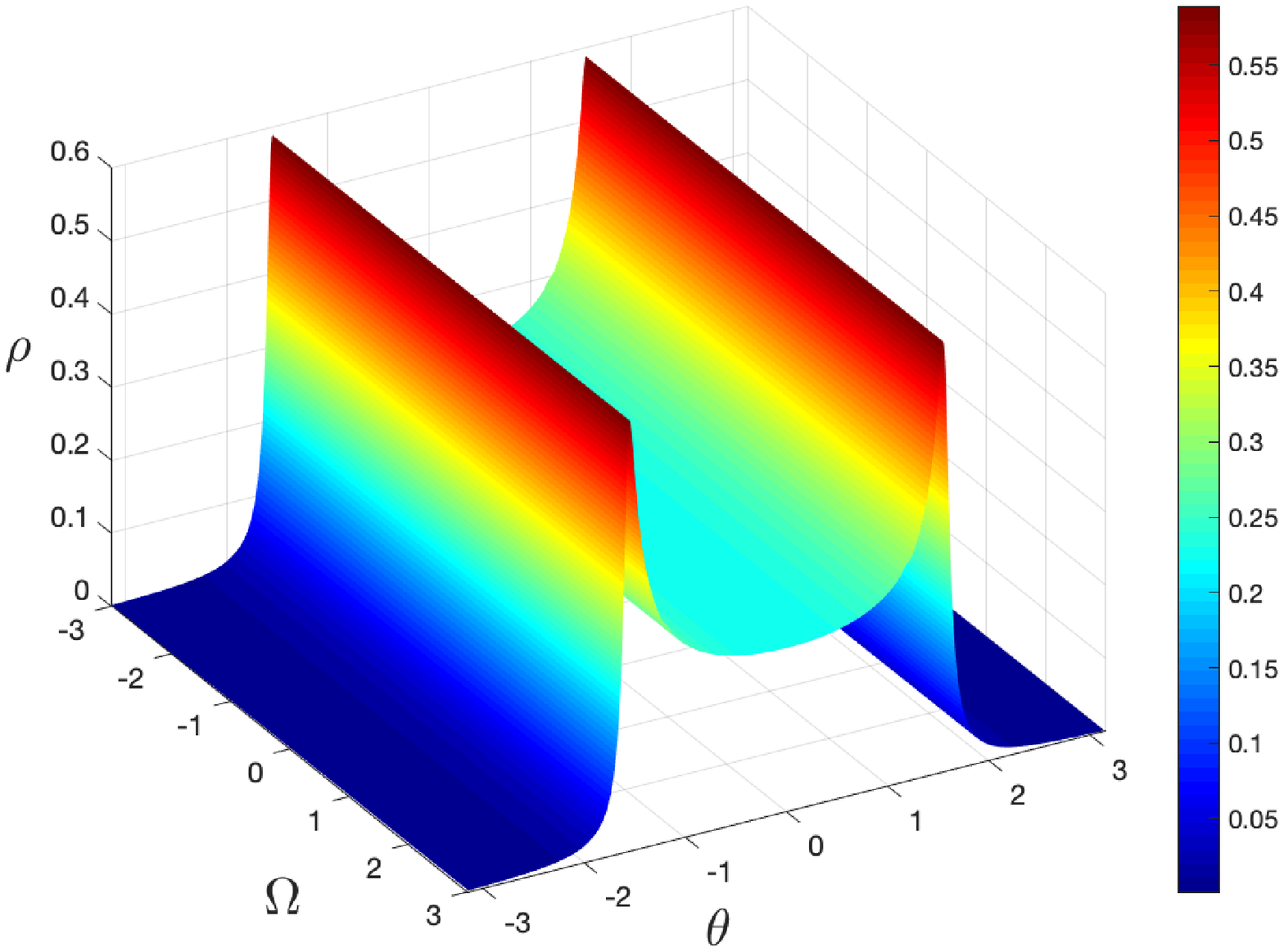}} 
\subfigure[Velocity at $t=0.04$]{  
\includegraphics[width=2.8in,height=1.5in]{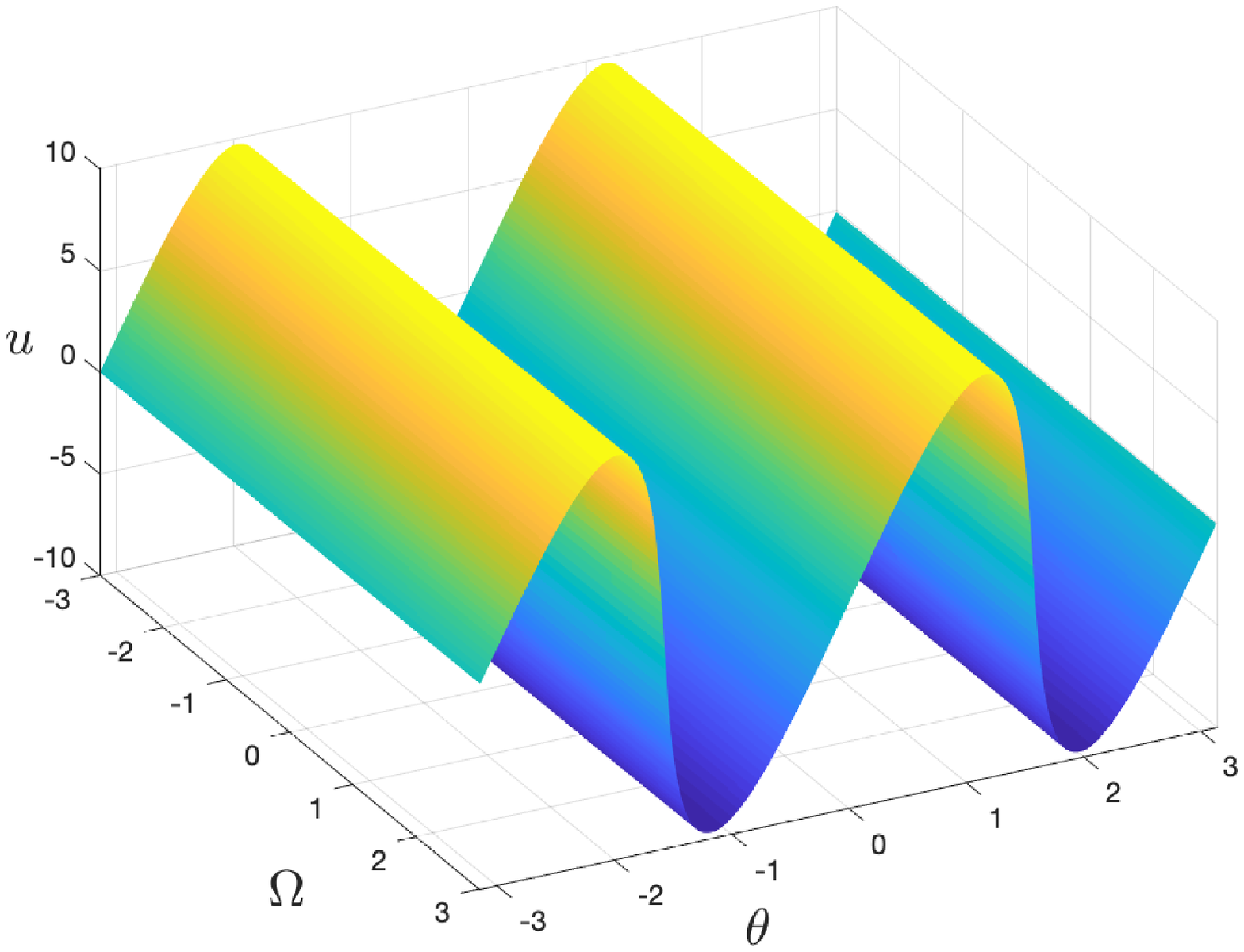}} 
\subfigure[Density at $t=0.1$] { 
\includegraphics[width=2.8in,height=1.5in]{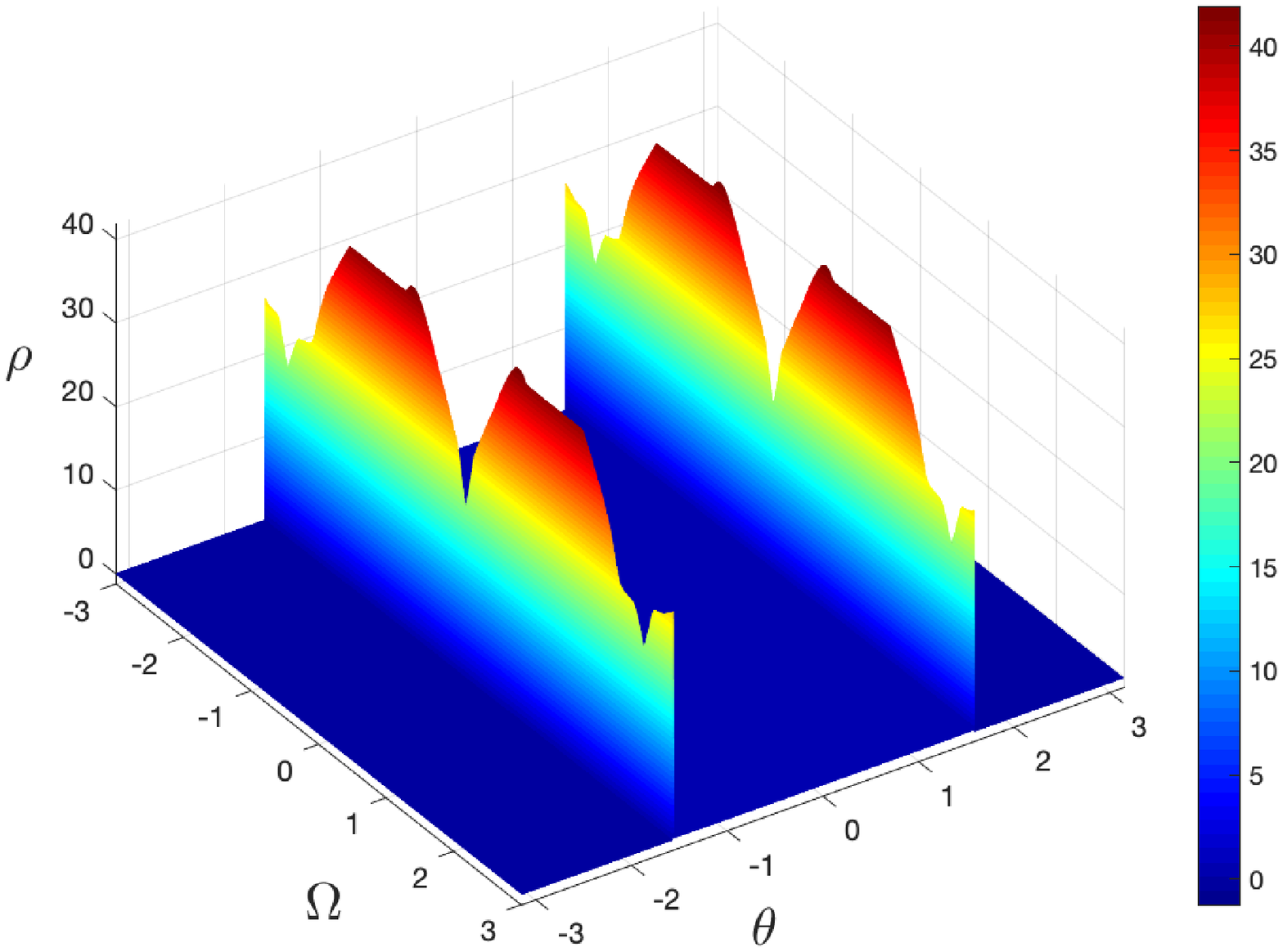}} 
\subfigure[Velocity at $t=0.1$]{  
\includegraphics[width=2.8in,height=1.5in]{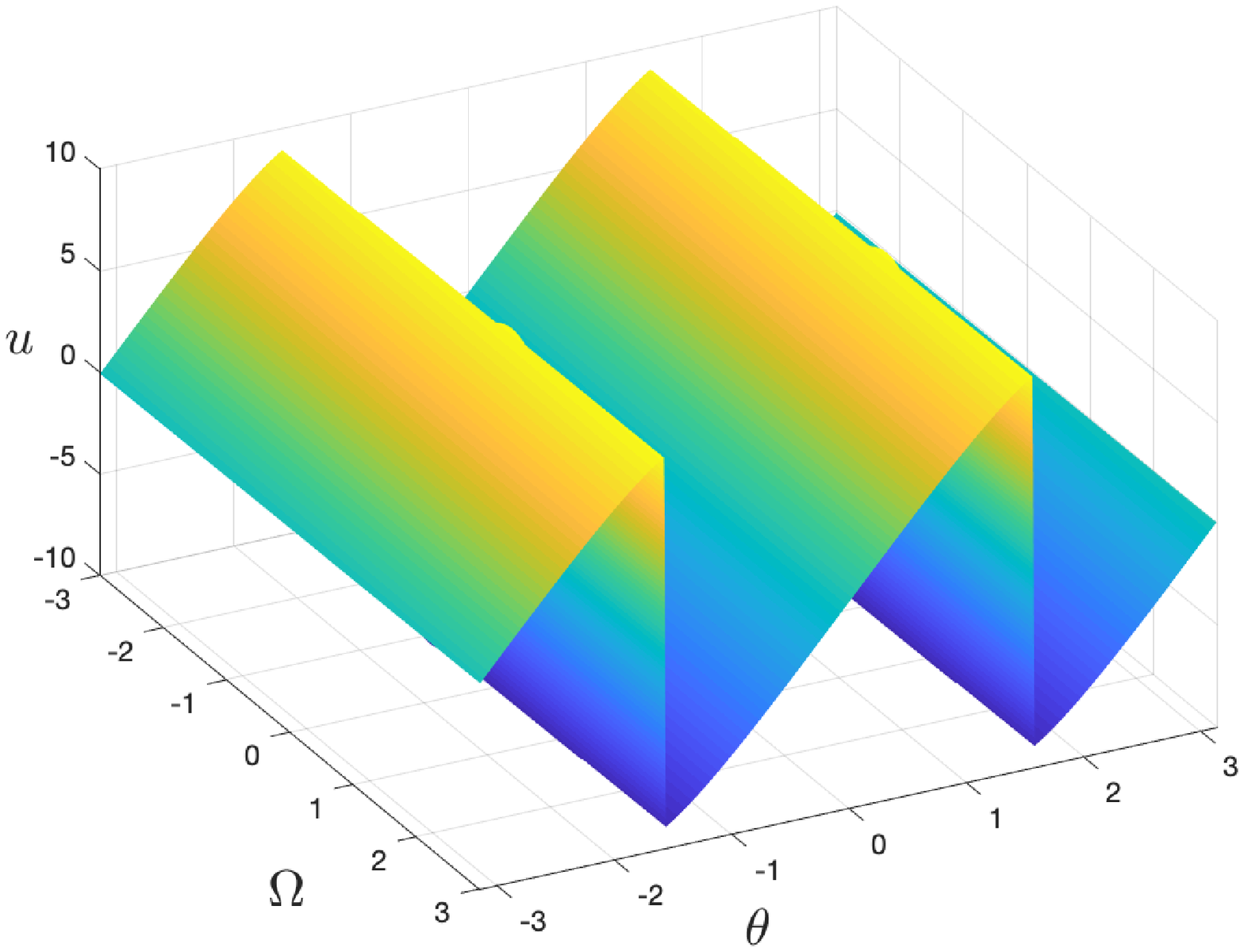}} 
\caption{Supercritical case 2}
\label{fig_nidsup2}
\end{figure}

Figures \ref{fig_nid} (a), (b) and Figure \ref{fig_nidsub} illustrate the profiles for the subcritical cases in the two-dimensional and three-dimensional plots, respectively. In Figures \ref{fig_nid} (a) and (b), we plot modified density $\tilde \rho$ and velocity $\tilde u$ which are given by 
\[
\tilde \rho(\theta,t) := \int_\R \rho(\theta,\Omega,t)g(\Omega)\,d\Omega \quad \mbox{and} \quad \tilde u(\theta,t) := \int_\R u(\theta,\Omega,t)g(\Omega)\,d\Omega
\]
so that the profiles can be observed on $\T$. The parameters and the initial velocity are set
\[
m=2, \quad K=0.1, \quad \mbox{and} \quad  u_0(\theta,\Omega) = -(0.1)\sin \theta,
\]
which guarantees the subcritical region for all $(\theta,\Omega)$ so that we have the global existence of the solution, see Theorem \ref{thm_cri} (i). We can observe that the density does not tend to a Dirac measure unlike identical case.  Note that for each fixed $\Omega$, the density profile $\rho(\theta,\Omega)$ exhibits the bell-shaped distribution whose center varies according to the natural frequency $\Omega$. The frequency synchronization is also observed in the time evolution of $u$. 

The plots of $\tilde \rho, \rho, \tilde u, u$ for the first supercritical case are provided in Figures \ref{fig_nid} (c), (d) and Figure \ref{fig_nidsup} where the parameters and the initial data are chosen as
\[
m=2, \quad K=0.1, \quad \mbox{and} \quad  u_0(\theta,\Omega) = -10\sin \theta.
\]
In this case, the initial data lie in the supercritical region around $\theta=0$. The finite-time blow-up of $\rho$ and $\pa_\theta u$ in the small time interval is easily observed here, which is consistent with our theoretical results Proposition \ref{prop_cri} and Remark \ref{rem_density}. 

In Figures \ref{fig_nid} (e), (f) and Figure \ref{fig_nidsup2}, we present the plots of the density and velocity profiles for the second supercritical case.  As is the case in the identical oscillators(Figures \ref{fig_id} (e) and (f)), we take the following parameters and the initial velocity:
\[
m=2, \quad K=0.1, \quad \mbox{and} \quad  u_0(\theta, \Omega) = 10\sin 2\theta.
\]
to consider the supercritical case around $\theta = \pm \pi/2$, and the figures also exhibit the finite-time blow-up around these points.

\subsection{Phase transitions \& hysteresis phenomena} 
In Figure \ref{fig_hyst},  we show the phase transitions of the order parameter $r^\infty$ for the hydrodynamic model \eqref{hydro_Ku} with the coupling strength on the interval $[0,4]$. It is known that unlike the Kuramoto model without inertia, where the phase transition of the order parameter versus the coupling strength is continuous provided the distribution function $g$ is Gaussian, even the small inertia can lead to a discontinuous and hysteretic phase transition for the system \eqref{particle_Ku}. As noted in Introduction, the hydrodynamic model \eqref{hydro_Ku} also carries the hysteresis phenomena as \eqref{particle_Ku} does. Figure \ref{fig_hyst} exhibits the discontinuous phase transition of \eqref{hydro_Ku} with $m=0.1, 0.5, 1$. The numbers of $\theta$-grid and $\Omega$-grid are set to $100$ and $600$, respectively. We increase $K$ from $0$ to $4$ with the mesh spacing of $0.1$ for $K$. When $K$ reaches $4$, the same procedure is iterated by decreasing $K$ back to $0$. In order to gain the clear observation on the thresholds, the finer mesh spacing in $K$, which is $0.05$, is used around them. The direction of jump is indicated with arrows. The initial density $\rho_0$ and the distribution $g$ are set to the Gaussian distribution and the initial velocity is $u_0(\theta, \Omega) = -(0.5)\sin \theta$. Note that $u_0$ is chosen such that it does not lie in the supercritical region for any of $m \in \{0.1,0.5,1\}$ and $K \in [0,4]$, see Section \ref{sec_cri}. We can observe that as $m$ increases, the hysteresis becomes more noticeable.

\begin{figure}[ht] 
\centering
\includegraphics[width=5.0in]{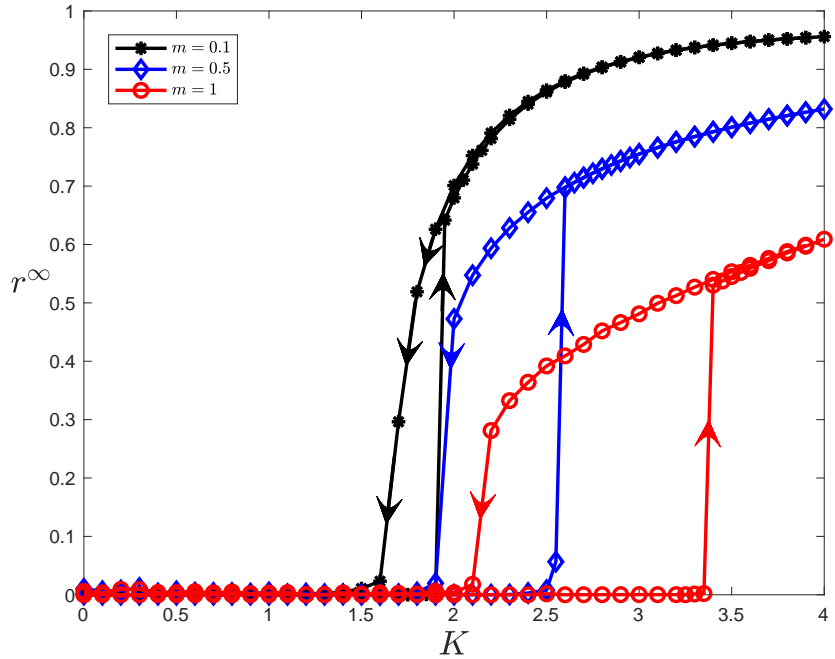}
\caption{Phase transitions \& hysteresis phenomena} 
\label{fig_hyst}
\end{figure}

\section{Conclusion}
In this manuscript, we presented a new hydrodynamic model for the synchronization phenomena and discussed the local-in-time existence theory. For the identical natural frequencies, we provided two different approaches for the synchronization estimates; kinetic energy combined with the order parameter estimates and the second-order Gr\"onwall-type inequality estimates on the phase and velocity diameters. In particular, by the latter strategy, we showed that the limiting density is the form of the Dirac measure. We also analyzed the critical threshold phenomena in our main system. By this analysis, we found that classical solutions can be blow-up in finite time, which is not observed in the classical Kuramoto models. We were not able to prove this finite-time blow-up of solutions implies the finite-time synchronization, however, numerical simulations illustrated that the density with initial data in the supercritical region becomes Dirac measures in finite time. We also presented several numerical simulations validating our analytical results. The numerical results showed that our main system also has similar features, such as phase transitions and hysteresis phenomena, compared to the Kuramoto model with inertia. As briefly mentioned in Introduction, the pressureless Euler-type equations may develop the formation of singularities. For this reason, it is natural to take into account the notion of measure-valued solutions. Thus it would be interesting to study the existence of measure-valued solutions to our main system. This may enable us to have the global-in-time regularity of solutions. As the first step in this hydrodynamic modeling of synchronization phenomena, we only deal with the case of identical oscillators for the synchronization estimates. Hence, our next step is to generalize our analysis for the case of nonidentical natural frequencies. We will investigate these interesting issues in future.


\appendix


\section{Proof of Theorem \ref{thm_local2}} \label{app_local}

For computational simplicity, we set $m = K = 1$. Let $T>0$ be given, and we consider the system: 
\begin{align}\label{app_euler}
\begin{aligned}
&\pa_t \rho + \pa_\theta(\rho \bar u) = 0, \quad (\theta,\Omega) \in \T \times \R, \quad t> 0,\cr
&\rho \pa_t u +  \rho \bar u \pa_\theta u =   -\rho u +  \rho \Omega + \rho\int_{\T \times \R} \sin(\theta_* - \theta)\rho(\theta_*,\Omega_*,t)g(\Omega_*)\,d\theta_* d\Omega_*,
\end{aligned}
\end{align}
with the initial data $(\rho_0, u_0)$ satisfying the assumptions \eqref{asp_thm2}. Here $\bar u$ satisfies
\bq\label{asp_bu}
\sup_{0 \leq t \leq T} \lt(\|\bar u(\cdot,\cdot,t)\|_{H^3_g} + \|\pa_\theta \bar u(\cdot,\cdot,t)\|_{L^\infty} \rt) < M.
\eq
Note that we can use a standard linear theory to show the existence of solutions to the system \eqref{app_euler}. We begin by estimating $\|\rho\|_{H^2_g}$. A direct calculation gives
$$\begin{aligned}
\frac12\frac{d}{dt}\int_{\T \times \R} \rho^2 g\,d\theta d\Omega &= \int_{\T \times \R} \rho \bar u(\pa_\theta \rho)  g\,d\theta d\Omega \cr
&=-\frac12\int_{\T \times \R} \rho^2 (\pa_\theta \bar u) g\,d\theta d\Omega \cr
&\leq \|\pa_\theta \bar u\|_{L^\infty} \|\rho\|_{L^2_g}^2.
\end{aligned}$$
Similarly, we can easily obtain
\[
\frac{d}{dt} \int_{\T \times \R} (\pa_\theta \rho)^2 g\,d\theta d\Omega \ls  \|\pa_\theta \bar u\|_{L^\infty} \|\pa_\theta \rho\|_{L^2_g}^2 + \|\rho\|_{L^\infty}\|\pa_\theta \rho\|_{L^2_g} \|\pa_\theta^2 u\|_{L^2_g}.
\]
We next estimate
\begin{align}\label{est_rho2}
\begin{aligned}
&\frac12\frac{d}{dt} \int_{\T \times \R} (\pa_\theta^2 \rho)^2 g\,d\theta d\Omega \cr
&\quad = -\int_{\T \times \R} \pa_\theta^2 \rho \lt(\bar u\pa_\theta^3 \rho  + 3 \pa_\theta^2 \rho \pa_\theta \bar u + 3 \pa_\theta \rho \pa_\theta^2 \bar u + \rho \pa_\theta^3 \bar u \rt) g \,d\theta d\Omega\cr
&\quad \ls  \|\pa_\theta \bar u\|_{L^\infty} \|\pa_\theta^2 \rho\|_{L^2_g}^2 + \|\rho\|_{L^\infty}\|\pa_\theta^2 \rho\|_{L^2_g} \|\pa_\theta^3 \bar u\|_{L^2_g} \cr
&\qquad + \int_{\T \times \R} (\pa_\theta \rho)^2 (\pa_\theta^3 \bar u) g\,d\theta d\Omega.
\end{aligned}
\end{align}
Note that
\[
0 = \int_{\T \times \R} \pa_\theta (\rho (\pa_\theta \rho)^3)g\,d\theta d\Omega = \int_{\T \times \R} \lt((\pa_\theta \rho)^4 + 3\rho(\pa_\theta \rho)^2 \pa_\theta^2 \rho \rt)g\,d\theta d\Omega,
\]
then this together with applying H\"older's inequality yields
$$\begin{aligned}
\int_{\T \times \R} (\pa_\theta \rho)^4 g\,d\theta d\Omega &= -3\int_{\T \times \R}  \rho(\pa_\theta \rho)^2 (\pa_\theta^2 \rho)g\,d\theta d\Omega \cr
&\leq 3\|\rho\|_{L^\infty}\lt(\int_{\T \times \R} (\pa_\theta \rho)^4 g\,d\theta d\Omega \rt)^{1/2} \lt(\int_{\T \times \R} (\pa_\theta^2 \rho)^2 g\,d\theta d\Omega \rt)^{1/2}.
\end{aligned}$$
Thus we get
\[
\|\pa_\theta \rho\|_{L^4_g}^2 \leq 3\|\rho\|_{L^\infty}\|\pa_\theta^2 \rho\|_{L^2_g}.
\]
Using the above inequality, we estimate the last term on the right hand side of \eqref{est_rho2} as 
\[
\int_{\T \times \R} (\pa_\theta \rho)^2 (\pa_\theta^3 \bar u) g\,d\theta d\Omega \leq \|\pa_\theta \rho\|_{L^4_g}^2\|\pa_\theta^3 \bar u\|_{L^2_g} \leq 3\|\rho\|_{L^\infty}\|\pa_\theta^2 \rho\|_{L^2_g}\|\pa_\theta^3 \bar u\|_{L^2_g}.
\]
We now combine all of the above observations to have
\[
\frac{d}{dt} \|\rho\|_{H^2_g}^2 \leq C\|\pa_\theta \bar u\|_{L^\infty}\|\rho\|_{H^2_g}^2 + C\|\rho\|_{L^\infty}\|\pa_\theta \rho\|_{H^1_g}\|\pa_\theta \bar u\|_{H^2_g},
\]
that is,
\begin{align}\label{est_rho3}
\begin{aligned}
\frac{d}{dt} \|\rho\|_{H^2_g} &\leq C\|\pa_\theta \bar u\|_{L^\infty}\|\rho\|_{H^2_g} + C\|\rho\|_{L^\infty}\|\pa_\theta \bar u\|_{H^2_g}\cr
&\leq CM\lt(\|\rho\|_{L^\infty} + \|\rho\|_{H^2_g}\rt).
\end{aligned}
\end{align}
On the other hand, by taking into account the characteristic flow defined by
\bq\label{app_char}
\pa_t \bar \eta(\theta,\Omega,t) = \bar u(\bar \eta(\theta,\Omega,t),\Omega,t) \quad \mbox{with} \quad \bar \eta(\theta,\Omega,0) = \theta,
\eq
we can easily estimate
\bq \label{est_rhoinf}
\|\rho(\cdot,\cdot,t)\|_{L^\infty} \leq \|\rho_0\|_{L^\infty} \exp\lt(\int_0^t \|\pa_\theta \bar u(\cdot,\cdot,s)\|_{L^\infty}\,ds \rt) \leq \|\rho_0\|_{L^\infty} e^{MT},
\eq
for $t \in [0,T]$. This together with \eqref{est_rho3} asserts
\[
\frac{d}{dt} \|\rho(\cdot,\cdot,t)\|_{H^2_g} \leq CM\lt(\|\rho_0\|_{L^\infty} e^{MT} + \|\rho\|_{H^2_g}\rt).
\]
Applying Gr\"onwall's lemma to the inequality above, we have
\bq\label{est_rhoh2}
\|\rho(\cdot,\cdot,t)\|_{H^2_g} \leq \|\rho_0\|_{H^2_g} e^{CMT} + \|\rho_0\|_{L^\infty} e^{MT}\lt(e^{CMT} - 1 \rt),
\eq
for $t \in [0,T]$. 

We next estimate $\|\pa_\theta u\|_{L^\infty}$ and $\|u\|_{H^3_g}$. For this, similarly as before, we can use the characteristic flow to find
\[
\rho(\theta,\Omega,t)  \geq \inf_{(\theta,\Omega) \in \T \times \R}\rho_0(\theta,\Omega) e^{-MT} > 0,
\]
for all $(\theta,\Omega) \in \T \times \R$. This enables us to divide the momentum equation in \eqref{app_euler} by $\rho$, and this gives
\[
\pa_t u + \bar u \pa_\theta u = - u +  \Omega + \int_{\T \times \R} \sin(\theta_* - \theta)\rho(\theta_*,\Omega_*,t)g(\Omega_*)\,d\theta_* d\Omega_*.
\]
Along the characteristic flow \eqref{app_char}, we estimate
$$\begin{aligned}
\|\pa_\theta u(\cdot,\cdot,t)\|_{L^\infty}e^{t} &\leq \|\pa_\theta u_0\|_{L^\infty} + \int_0^t e^s\|\pa_\theta \bar u(\cdot,\cdot,t)\|_{L^\infty} \|\pa_\theta u(\cdot,\cdot,t)\|_{L^\infty} \,ds\cr
&\quad - \int_0^t e^s \lt(\int_{\T \times \R} \cos(\theta_* - \bar \eta(\theta,\Omega,t))\rho(\theta_*,\Omega_*,s)g(\Omega_*)\,d\theta_* d\Omega_*\rt) ds\cr
&\leq \|\pa_\theta u_0\|_{L^\infty} + M\int_0^t e^s\|\pa_\theta u(\cdot,\cdot,t)\|_{L^\infty} \,ds + e^t  - 1.
\end{aligned}$$
Thus we obtain
\bq\label{est_dtheta}
\|\pa_\theta u(\cdot,\cdot,t)\|_{L^\infty} \leq \|\pa_\theta u_0\|_{L^\infty} e^{CMT} + C(e^{CMT} - 1),
\eq
for $t \in [0,T]$, where $C>0$ is independent of $t$. For the estimate of $\|u\|_{H^3_g}$, we first notice that
\[
\lt|\pa_\theta^k \int_{\T \times \R} \sin(\theta_* - \theta) \rho(\theta_*,\Omega_*,t)g(\Omega_*)\,d\theta_* d\Omega_* \rt| \leq \int_{\T \times \R} \rho(\theta_*,\Omega_*,t)g(\Omega_*)\,d\theta_* d\Omega_* = 1,
\]
for $k \in \N$, due to Remark \ref{rmk_21}. This together with similar estimates as above yields
$$\begin{aligned}
\frac{d}{dt}\|u\|_{L^2_g}^2 &\ls \|\pa_\theta \bar u \|_{L^\infty}\|u\|_{L^2_g}^2 + \|u\|_{L^2_g},\cr
\frac{d}{dt}\|\pa_\theta u\|_{L^2_g}^2 &\ls \|\pa_\theta \bar u \|_{L^\infty}\|\pa_\theta u\|_{L^2_g}^2 + \|\pa_\theta u\|_{L^2_g},\cr
\frac{d}{dt}\|\pa_\theta^2 u\|_{L^2_g}^2 &\ls \|\pa_\theta u\|_{L^\infty}\|\pa_\theta^2 u\|_{L^2_g}\|\pa_\theta^2 \bar u\|_{L^2_g} + \|\pa_\theta \bar u \|_{L^\infty}\|\pa_\theta^2 u\|_{L^2_g}^2 + \|\pa_\theta^2 u\|_{L^2_g},\cr
\frac{d}{dt}\|\pa_\theta^3 u\|_{L^2_g}^2 &\ls \|\pa_\theta u \|_{L^\infty}\|\pa_\theta^3 \bar u\|_{L^2_g}\|\pa_\theta^3 u\|_{L^2_g} + \|\pa_\theta^3 \bar u\|_{L^2_g}\|\pa_\theta^2 u\|_{L^4_g}^2 \cr
&\quad + \|\pa_\theta \bar u \|_{L^\infty}\|\pa_\theta^3 u\|_{L^2_g}^2 + \|\pa_\theta^3 u\|_{L^2_g} \cr
&\ls \|\pa_\theta u \|_{L^\infty}\|\pa_\theta^3 \bar u\|_{L^2_g}\|\pa_\theta^3 u\|_{L^2_g} + \|\pa_\theta \bar u \|_{L^\infty}\|\pa_\theta^3 u\|_{L^2_g}^2 + \|\pa_\theta^3 u\|_{L^2_g},
\end{aligned}$$
where we used the assumption $\int_\R \Omega^2 g(\Omega)\,d\Omega < \infty$ in \eqref{condi_g} and the following inequality:
\[
\| \pa_\theta^2 u\|_{L_g^4} \lesssim  \| \pa_\theta u\|_{L^\infty}^{1/2} \| \pa_\theta^3 u\|_{L_g^2}^{1/2}.
\] 
This asserts
\begin{align}\label{est_u3}
\begin{aligned}
\frac{d}{dt}\|u\|_{H^3_g} &\leq C\|\pa_\theta \bar u \|_{L^\infty}\|u\|_{H^3_g} + C\|\pa_\theta u\|_{L^\infty}\|\pa_\theta^2 \bar u\|_{H^2_g} + C\cr
&\leq CM\|u\|_{H^3_g} + CM\|\pa_\theta u_0\|_{L^\infty} e^{CMT} + CM(e^{CMT} - 1) + C.
\end{aligned}
\end{align}
Here we used the estimate \eqref{est_dtheta} and the assumption \eqref{asp_bu}. Applying Gr\"onwall's lemma to \eqref{est_u3} gives
\[
\|u(\cdot,\cdot,t)\|_{H^3_g} \leq \|u_0\|_{H^3_g} e^{CMT} + C\lt(\|\pa_\theta u_0\|_{L^\infty} e^{CMT} + (e^{CMT} - 1) + \frac1M\rt)(e^{CMT} -1).
\]
By combining this with \eqref{est_rhoinf},  \eqref{est_rhoh2}, and \eqref{est_dtheta}, we have
$$\begin{aligned}
&\sup_{0 \leq t \leq T} \lt(\|\rho(\cdot,\cdot,t)\|_{H^2_g} + \|\rho(\cdot,\cdot,t)\|_{L^\infty} + \|u(\cdot,\cdot,t)\|_{H^3_g} + \|\pa_\theta u(\cdot,\cdot,t)\|_{L^\infty} \rt)  \cr
&\quad \leq \lt(\|\rho_0\|_{L^\infty} + \|\rho_0\|_{H^2_g} + \|\pa_\theta u_0\|_{L^\infty} + \|u_0\|_{H^3_g} \rt)e^{CMT} \cr
&\qquad + C\lt(\|\pa_\theta u_0\|_{L^\infty} e^{CMT} + (e^{CMT} - 1) + \frac1M + \|\rho_0\|_{L^\infty} e^{CMT} + 1\rt)(e^{CMT} -1)\cr
&\leq Ne^{CMT} + C\lt(\|\pa_\theta u_0\|_{L^\infty} e^{CMT}+ \|\rho_0\|_{L^\infty} e^{CMT}  + 1\rt)(e^{CMT} -1).
\end{aligned}$$
We finally choose $T>0$ small enough such that the right hand side of the above inequality is less than $M$. We then deal with the approximations for the system \eqref{hydro_Ku}:
\begin{align}\label{app_euler2}
\begin{aligned}
&\pa_t \rho^{n+1} + \pa_\theta(\rho^{n+1} u^n) = 0, \quad (\theta,\Omega) \in \T \times \R, \quad t> 0,\cr
&\rho^{n+1} \pa_t u^{n+1} +  \rho^{n+1} u^n \pa_\theta u^{n+1}\cr
&\qquad  =   -\rho^{n+1} u^{n+1} +  \rho^{n+1} \Omega + \rho^{n+1}\int_{\T \times \R} \sin(\theta_* - \theta)\rho^{n+1}(\theta_*,\Omega_*,t)g(\Omega_*)\,d\theta_* d\Omega_*
\end{aligned}
\end{align}
with the initial data and the first iteration step given by
\[
(\rho^n(\theta,\Omega,t), u^n(\theta,\Omega,t))|_{t = 0} = (\rho_0(\theta,\Omega), u_0(\theta,\Omega)), \quad n \geq 1, \quad (\theta,\Omega) \in \T \times \R,
\]
and
\[
(\rho^0(\theta,\Omega,t), u^0(\theta,\Omega,t)) = (\rho_0(\theta,\Omega), u_0(\theta,\Omega)), \quad (\theta,\Omega,t) \in \T \times \R \times \R_+.
\]
For the system \eqref{app_euler2}, we use the previous argument to get
\bq\label{est_uni}
\sup_{0 \leq t \leq T_0} \sup_{n \in \N}\lt(\|\rho^n(\cdot,\cdot,t)\|_{H^2_g} + \|\rho^n(\cdot,\cdot,t)\|_{L^\infty} + \|u^n(\cdot,\cdot,t)\|_{H^3_g} + \|\pa_\theta u^n(\cdot,\cdot,t)\|_{L^\infty} \rt) < M,
\eq
for some $T_0 > 0$. We next show that $(\rho^n, u^n)$ is a Cauchy sequence in $\mc([0,T_0]; L^2_g(\T \times \R)) \times \mc([0,T_0]; H^1_g(\T \times \R))$. For this, we introduce the following simplified notations:
\[
\rho^{n+1,n} := \rho^{n+1} - \rho^n \quad \mbox{and} \quad u^{n+1,n} := u^{n+1} - u^n \quad \mbox{for} \quad n \in \N \cup \{0\}.
\]
Then straightforward computations yield
$$
\begin{aligned}
\frac{d}{dt}\|\rho^{n+1,n}\|_{L^2_g}^2 &\ls \|\pa_\theta u^n\|_{L^\infty}\|\rho^{n+1,n}\|_{L^2_g}^2 + \|\rho^n\|_{L^\infty}\|\rho^{n+1,n}\|_{L^2_g}\|\pa_\theta u^{n,n-1}\|_{L^2_g}\cr
&\quad + \|\rho^{n+1,n}\|_{L^2_g}\|\pa_\theta \rho^n\|_{L^4_g}\|u^{n,n-1}\|_{L^4_g}\cr
&\ls \|\rho^{n+1,n}\|_{L^2_g}^2 + \|\rho^{n+1,n}\|_{L^2_g}\|\pa_\theta u^{n,n-1}\|_{L^2_g} +\|\rho^{n+1,n}\|_{L^2_g}\|u^{n,n-1}\|_{L^4_g},
\end{aligned}
$$
where we used
\[
\|\pa_\theta \rho^n\|_{L^4_g} \ls \|\rho^n\|_{L^\infty}^{1/2}\|\pa_\theta^2 \rho^n \|_{L^2_g}^{1/2} \ls 1,
\]
due to \eqref{est_uni}. This asserts
\bq\label{cau_rho2}
\frac{d}{dt}\|\rho^{n+1,n}\|_{L^2_g} \ls \|\rho^{n+1,n}\|_{L^2_g} + \|\pa_\theta u^{n,n-1}\|_{L^2_g}  + \|u^{n,n-1}\|_{L^4_g}.
\eq
We next estimate $\|u^{n+1,n}\|_{L^4_g}$ and $\|\pa_\theta u^{n+1,n}\|_{L^2_g}$. It follows from the momentum equation in \eqref{app_euler2} that
$$\begin{aligned}
&\pa_t u^{n+1,n} + u^{n,n-1}\pa_\theta u^{n+1} + u^{n-1} \pa_\theta u^{n+1,n}\cr
&\quad = -u^{n+1,n} + \int_{\T \times \R} \sin(\theta_* - \theta)\rho^{n+1,n}(\theta_*,\Omega_*,t)g(\Omega_*)\,d\theta_* d\Omega_*.
\end{aligned}$$
Thus we find
$$\begin{aligned}
&\frac{d}{dt}\int_{\T \times \R} |u^{n+1,n}|^4 g\,d\theta d\Omega\cr
&\quad = -4\int_{\T \times \R} (u^{n+1,n})^3 \lt(  u^{n,n-1}\pa_\theta u^{n+1} + u^{n-1} \pa_\theta u^{n+1,n} + u^{n+1,n}\rt)g\,d\theta d\Omega\cr
&\qquad + 4\int_{\T^2 \times \R^2}  (u^{n+1,n}(\theta,\Omega,t))^3 \sin(\theta_* - \theta)\rho^{n+1,n}(\theta_*,\Omega_*,t)g(\Omega_*) g(\Omega)\,d\theta_* d\Omega_* d\theta d\Omega\cr
&\quad \leq 4\|\pa_\theta u^{n+1}\|_{L^\infty}\|u^{n+1,n}\|_{L^4_g}^3\|u^{n,n-1}\|_{L^4_g} + 4\|\pa_\theta u^{n-1}\|_{L^\infty}\|u^{n+1,n}\|_{L^4_g}^4\cr
&\qquad -4\|u^{n+1,n}\|_{L^4_g}^4 + 4\|u^{n+1,n}\|_{L^4_g}^3\|\rho^{n+1,n}\|_{L^2_g}\cr
&\quad \ls \|u^{n+1,n}\|_{L^4_g}^3\|u^{n,n-1}\|_{L^4_g} + \|u^{n+1,n}\|_{L^4_g}^4 + \|u^{n+1,n}\|_{L^4_g}^3\|\rho^{n+1,n}\|_{L^2_g},
\end{aligned}$$
where we used \eqref{condi_g}. This gives
\bq\label{cau_u2}
\frac{d}{dt}\|u^{n+1,n}\|_{L^4_g} \ls \|u^{n,n-1}\|_{L^4_g}  + \|u^{n+1,n}\|_{L^4_g} + \|\rho^{n+1,n}\|_{L^2_g}.
\eq
We also use the similar argument as above to estimate
$$\begin{aligned}
&\frac12\frac{d}{dt} \int_{\T \times \R} (\pa_\theta u^{n+1,n})^2 g\,d\theta d\Omega\cr
&\quad = -\int_{\T \times \R} \pa_\theta u^{n+1,n} \lt(  \pa_\theta u^{n,n-1}\pa_\theta u^{n+1} + u^{n,n-1}\pa_\theta^2 u^{n+1} + \pa_\theta u^{n-1} \pa_\theta u^{n+1,n}\rt)g\,d\theta d\Omega\cr
&\qquad -\int_{\T \times \R} \pa_\theta u^{n+1,n} \lt(  u^{n-1} \pa_\theta^2 u^{n+1,n} + \pa_\theta u^{n+1,n}\rt)g\,d\theta d\Omega\cr
&\qquad  -\int_{\T^2 \times \R^2}  \pa_\theta u^{n+1,n}(\theta,\Omega,t) \cos(\theta_* - \theta)\rho^{n+1,n}(\theta_*,\Omega_*,t)g(\Omega_*) g(\Omega)\,d\theta_* d\Omega_* d\theta d\Omega\cr
&\quad \leq \|\pa_\theta u^{n+1}\|_{L^\infty}\|\pa_\theta u^{n+1,n}\|_{L^2_g}\|\pa_\theta u^{n,n-1}\|_{L^2_g} + \|\pa_\theta u^{n+1,n}\|_{L^2_g}\|u^{n,n-1}\|_{L^4_g}\|\pa_\theta^2 u^{n+1}\|_{L^4_g}\cr
&\qquad  + \|\pa_\theta u^{n-1}\|_{L^\infty}\|\pa_\theta u^{n+1,n}\|_{L^2_g}^2 - \|\pa_\theta u^{n+1,n}\|_{L^2_g}^2 + \|\pa_\theta u^{n+1,n}\|_{L^2_g}\|\rho^{n+1,n}\|_{L^2_g}\cr
&\quad \ls \|\pa_\theta u^{n+1,n}\|_{L^2_g}\|\pa_\theta u^{n,n-1}\|_{L^2_g} + \|\pa_\theta u^{n+1,n}\|_{L^2_g}\|u^{n,n-1}\|_{L^4_g}\cr
&\qquad  + \|\pa_\theta u^{n+1,n}\|_{L^2_g}^2 + \|\pa_\theta u^{n+1,n}\|_{L^2_g}\|\rho^{n+1,n}\|_{L^2_g}.
\end{aligned}$$
Here we used 
\[
\|\pa_\theta^2 u^{n+1}\|_{L^4_g} \ls \|\pa_\theta u^{n+1}\|_{L^\infty}^{1/2} \|\pa_\theta^3 u^{n+1}\|_{L^2_g}^{1/2} \ls 1.
\]
Thus we obtain
\[
\frac{d}{dt}\|\pa_\theta u^{n+1,n}\|_{L^2_g} \ls \|\pa_\theta u^{n,n-1}\|_{L^2_g} + \|u^{n,n-1}\|_{L^4_g} + \|\pa_\theta u^{n+1,n}\|_{L^2_g} + \|\rho^{n+1,n}\|_{L^2_g},
\]
and this together with \eqref{cau_rho2} and \eqref{cau_u2} yields
$$\begin{aligned}
&\|\rho^{n+1,n}(\cdot,\cdot,t)\|_{L^2_g} + \|u^{n+1,n}(\cdot,\cdot,t)\|_{L^4_g} + \|\pa_\theta u^{n+1,n}(\cdot,\cdot,t)\|_{L^2_g} \cr
&\quad \leq C\int_0^t \|\rho^{n+1,n}(\cdot,\cdot,s)\|_{L^2_g} + \|u^{n+1,n}(\cdot,\cdot,s)\|_{L^4_g} + \|\pa_\theta u^{n+1,n}(\cdot,\cdot,s)\|_{L^2_g} \,ds\cr
&\qquad + C\int_0^t \|u^{n,n-1}(\cdot,\cdot,s)\|_{L^4_g} + \|\pa_\theta u^{n,n-1}(\cdot,\cdot,s)\|_{L^2_g} \,ds,
\end{aligned}$$
for $t \in [0,T_0]$, where $C>0$ is independent of $n$. Since 
$$\begin{aligned}
\|u^{n+1,n}(\cdot,\cdot,t)\|_{L^2_g}^2 &= \int_{\T \times \R} |u^{n+1,n}|^2 g\,d\theta d\Omega \cr
&\leq \lt(\int_{\T \times \R} |u^{n+1,n}|^4 g\,d\theta d\Omega\rt)^{1/2}\lt(\int_{\T \times \R} g(\Omega) \,d\theta d\Omega\rt)^{1/2}\cr
& \leq C\|u^{n+1,n}(\cdot,\cdot,t)\|_{L^4_g}^2
\end{aligned}$$
this conclude that $(\rho^n, u^n)$ is a Cauchy sequence in $\mc([0,T_0]; L^2_g(\T \times \R)) \times \mc([0,T_0]; H^1_g(\T \times \R))$. The rest part of the proof is almost the same with Steps 3 - 5 in the proof of Theorem \ref{thm_local}. This completes the proof.
%
%
%
%

\section*{Acknowledgments}
YPC was supported by National Research Foundation of Korea(NRF) grant funded by the Korea government(MSIP) (No. 2017R1C1B2012918) and POSCO Science Fellowship of POSCO TJ Park Foundation. 

%
%
%
%

\end{document}